\documentclass[a4paper, 12pt]{{amsart}}
\usepackage{amsmath, mathtools}
\usepackage{amscd}
\usepackage{bm, amssymb}
\usepackage{amsthm}
\usepackage{ascmac, color}
\usepackage[dvipdfmx]{graphicx, xcolor, pdfpages}
\usepackage[all]{xy}
\usepackage{tikz-cd}
\usepackage{fancybox}
\usepackage{mathrsfs}

\newtheorem{lemma}{{\sc Lemma}}[section]
\newtheorem{corollary}[lemma]{{\sc Corollary}}
\newtheorem{proposition}[lemma]{{\sc Proposition}}
\newtheorem{theorem}[lemma]{{\sc Theorem}}
\theoremstyle{definition}
\newtheorem{remark}[lemma]{{\sc Remark}}
\newtheorem{definition}[lemma]{{\sc Definition}}

\newtheorem{conjecture}[lemma]{{\sc Conjecture}}

\numberwithin{equation}{section}

\def\Gb{{\mathfrak{b}}}

\def\Gg{{\mathfrak{g}}}
\def\Gh{{\mathfrak{h}}}

\def\Gl{{\mathfrak{l}}}

\def\Gn{{\mathfrak{n}}}

\def\Gp{{\mathfrak{p}}}

\def\Gz{{\mathfrak{z}}}

\def\GA{{\mathfrak{A}}}

\def\GD{\bm{\mathfrak{D}}}

\def\GH{{\mathfrak{H}}}

\def\GJ{\bm{\mathfrak{J}}}
\def\GK{{\mathfrak{K}}}

\def\GM{{\mathfrak{M}}}
\def\GN{{\mathfrak{N}}}
\def\GO{{\bm{\mathfrak{O}}}}
\def\GU{\bm{\mathfrak{U}}}

\def\GZ{\bm{\mathfrak{Z}}}

\def\BA{{\mathbb{A}}}
\def\BB{{\mathbb{B}}}
\def\BC{{\mathbb{C}}}

\def\BF{{\mathbb{F}}}

\def\BP{{\mathbb{P}}}
\def\BQ{{\mathbb{Q}}}
\def\BR{{\mathbb{R}}}

\def\BZ{{\mathbb{Z}}}

\def\BFI{{\mathbf{I}}}
\def\BFJ{{\mathbf{J}}}

\def\CA{{\mathcal A}}
\def\CB{{\mathcal B}}

\def\DD{{\mathcal D}}
\def\CE{{\mathcal E}}
\def\CF{{\mathcal F}}

\def\CO{{\mathcal O}}
\def\CI{{\mathcal I}}

\def\CL{{\mathcal L}}
\def\CM{{\mathcal M}}

\def\CP{{\mathcal P}}

\def\CR{{\mathcal R}}
\def\CS{{\mathcal S}}

\def\CU{{\mathcal U}}
\def\CV{{\mathcal V}}

\def\CX{{\mathcal X}}
\def\CY{{\mathcal Y}}

\def\Ad{\mathop{\rm Ad}\nolimits}
\def\ad{{\mathop{\rm ad}\nolimits}}

\def\deru{\partial}
\def\End{\mathop{\rm{End}}\nolimits}

\def\eq{\mathop{\rm eq}\nolimits}

\def\For{{\mathop{\rm For}\nolimits}}
\def\Hom{\mathop{\rm Hom}\nolimits}
\def\hotimes{{\;\mathop{\widehat{\otimes}}\nolimits\;}}
\def\HHom{\mathcal{H}om}
\def\id{\mathop{\rm id}\nolimits}
\def\Id{\mathop{\rm Id}\nolimits}
\def\Ind{{\mathop{\rm Ind}\nolimits}}
\def\inte{{\mathop{\rm int}\nolimits}}
\def\Image{\mathop{\rm Im}\nolimits}
\def\Ker{\mathop{\rm Ker\hskip.5pt}\nolimits}
\def\modu{\mathop{\rm mod}\nolimits}
\def\Mod{\mathop{\rm Mod}\nolimits}
\def\nil{{\rm{nil}}}

\def\Proj{\mathop{\rm Proj}\nolimits}

\def\Res{{\mathop{\rm Res}\nolimits}}

\def\Spec{{\rm{Spec}}}
\def\St{{\mathrm{St}}}

\def\Tor{{\rm{Tor}}}
\def\ur{{\rm{ur}}}
\def\wt{\mathop{\rm wt}\nolimits}
\def\CEnd{\mathop{{\mathcal{E}}nd}\nolimits}

\def\Har{{\rm{Har}}}
\def\Fr{{\rm{Fr}}}
\def\Loc{{\CL{oc}}}

\def\sD{{}^\sharp\GD}
\def\spi{{}^\sharp\pi}
\def\wsD{{}^\sharp\widehat{\GD}}

\def\tomega{{\widetilde{\omega}}}
\def\tW{{\tilde{W}}}\def\tS{{\widetilde{S}}}
\def\tBB{{\tilde{\BB}}}
\def\wU{{\widehat{U}}}
\def\wGamma{{\widehat{\Gamma}}}
\def\wCV{{\widehat{\CV}}}
\def\wCB{{\widehat{\CB}}}
\def\wCP{{\widehat{\CP}}}
\def\wpi{{\widehat{\pi}}}
\def\wGK{{\widehat{\GK}}}

\def\wGU{{\widehat{\GU}}}
\def\wGD{{\widehat{\GD}}}
\def\wK{{\widehat{K}}}
\def\wCO{{\widehat{\CO}}}
\def\wCE{{\widehat{\CE}}}
\def\tCE{{\widetilde{\CE}}}

\def\wE{{\widehat{E}}}

\def\ex{{\mathrm{ex}}}
\def\st{{\mathrm{st}}}

\def\reg{{\mathrm{reg}}}

\begin{document}
\title[Quantized flag manifolds]
{
Quantized flag manifolds and 
non-restricted modules over quantum groups at roots of unity}
\author{Toshiyuki TANISAKI}

\begin{abstract}
We give a proof of Lusztig's conjectural multiplicity formula  for non-restricted modules over the De Concini-Kac type quantized enveloping algebra at the $\ell$-th root of unity, where $\ell$ is an odd prime power satisfying certain reasonable conditions.
\end{abstract}

\vspace{1cm}
\maketitle
\section{{Introduction}}
\subsection{}

Let $G$ be a connected, simply-connected, simple algebraic group over the complex number field $\BC$, and let $\Gg$ be its Lie algebra.
The flag manifold $\CB$ of $G$ plays a crucial role in the geometric representation theory.
By the Borel-Weil theory one can construct a finite-dimensional rational $G$-module as the space of global sections of a $G$-equivariant line bundle on $\CB$.
By the works of Beilinson-Bernstein \cite{BB1} and 
Brylinsky-Kashiwara \cite{BK} around 1980 one can also construct infinite-dimensional $\Gg$-modules 
using $D$-modules on $\CB$ instead of line bundles.
More recently, this construction was generalized by 
Bezurukavnikov-Mirkovi\'{c}-Rumynin \cite{BMR}, \cite{BMR2} and 
Bezurukavnikov-Mirkovi\'{c} \cite{BM}
to the situation where the base field $k$ is of positive characteristic.
In particular, it was proved in \cite{BM} as a consequence that 
Lusztig's conjecture on non-restricted $\Gg_k$-modules
holds true, where $\Gg_k$ is the counterpart of $\Gg$ over $k$.

\subsection{}
The present paper arose out of the effort 
to give analogues for quantum groups of the above mentioned results for  Lie algebras.

Let $U_q(\Gg)$ be the quantized enveloping algebra of $\Gg$ over the rational function field $\BQ(q)$.
It is a $q$-analogue of the enveloping algebra $U(\Gg)$.
The first task is to construct the quantized flag manifold $\CB_q$ for $U_q(\Gg)$.
Note that the ordinary flag manifold $\CB$ is a projective algebraic variety whose homogeneous coordinate algebra $A$ is  a certain subalgebra of the affine coordinate algebra $\CO(G)$ of $G$.
One can easily define a $q$-analogue $A_q$ of $A$ using $U_q(\Gg)$; 
however, $A_q$ is a non-commutative ring (except for the case $G=SL_2$).
Hence in order to give a geometric meaning to $\CB_q$ 
we need the language of non-commutative algebraic geometry, 
developed in 
\cite{AZ}, \cite{V}, \cite{R} following Manin's idea \cite{M}.
Using $A_q$ we can define as in 
\cite{R}, \cite{LR} (see also \cite{Jo0}) the abelian category $\modu(\CO_{\CB_q})$ which is regarded as the category of ``coherent $\CO$-modules'' on the virtual space $\CB_q=\Proj\,(A_q)$.
In general for a not necessarily commutative graded ring $S$ 
we have a virtual space $\Proj\,(S)$ (non-commutative projective scheme) endowed with an (actual) abelian category 
$\Mod(\CO_{\Proj\,(S)})$ of ``quasi-coherent $\CO$-modules'' on it.
If $S$ is noetherian, we have also a smaller category  
$\modu(\CO_{\Proj\,(S)})$ of ``coherent $\CO$-modules'' on it.

In \cite{LR} Lunts and Rosenberg defined for a  general graded ring $S$ a subring $D'_S$ of $\End(S)$ from which they defined a certain 
abelian category which is regarded as the category of ``$\DD$-modules'' on the virtual space $\Proj\,(S)$".
Moreover, 
they formulated a Beilinson-Bernstein type correspondence for the quantized flag manifolds $\CB_q$ as a conjecture.
Namely, they conjectured that there is an equivalence 
between the category of certain $\DD$-modules on $\CB_q$ and that of certain $U_q(\Gg)$-modules.

This conjecture was settled in \cite{T0} by modifying the definition of the category of $\DD$-modules on $\CB_q$.
In \cite{T0} we used a subring $D_{A_q}$ of $\End(A_q)$ generated by a set of standard generators related to $U_q(\Gg)$
to define an abelian category of ``$\DD$-modules on $\CB_q$''.
Our definition of $D_{A_q}$ cannot be generalized to general graded rings $S$.
Moreover, our $D_{A_q}$ is presumably much smaller than $D'_{A_q}$ of Lunts and Rosenberg.
Nevertheless, we believe that our ad-hoc choice of  $D_{A_q}$ is natural in considering $
\CB_q$.

\subsection{}
Let us focus our attention on the quantum
analogues of the results \cite{BMR}, \cite{BMR2}, \cite{BM} 
on Lie algebras in positive characteristics.
It is known that phenomena in positive characteristics in the ordinary world resemble those at roots of unity in the quantum world.
Hence we consider the situation where the parameter $q$ is specialized to a complex number $\zeta$ which is an $\ell$-th root of unity.
By the specialization $q\mapsto\zeta$
we obtain the De Concini-Kac type quantized enveloping algebra $U_\zeta(\Gg)$ and the non-commutative graded ring $A_\zeta$.
We can also construct an abelian category 
$\modu(\CO_{\CB_\zeta})$ which is regarded as the category of ``coherent $\CO$-modules'' on the virtual space $\CB_\zeta=\Proj\,(A_\zeta)$ 
(by \cite{TA} the virtual space $\CB_\zeta$ turns out to be  a non-commutative scheme in the sense of \cite{R}).
The aim of the series \cite{T1}, \cite{T2}, \cite{T3} and the present paper is to give analogues of  \cite{BMR}, \cite{BMR2}, \cite{BM} 
using the quantized enveloping algebras and the quantized flag manifolds at roots of unity
instead of the ordinary enveloping algebras and the ordinary flag manifolds  in positive characteristics.

Our arguments are basically parallel to those in 
\cite{BMR}, \cite{BMR2}, \cite{BM}
although some new ideas are needed.
Sometimes an obvious geometric fact in positive characteristics requires additional algebraic arguments 
in the quantized situation
since we rely on our ad-hoc choice of the ring of differential operators on the quantized flag manifold.
For example, the ring of crystalline differential operators on a general algebraic variety $X$ in positive characteristic can be localized on a twisted cotangent bundle $T^*X^{(1)}$, and it turns out to be an Azumaya algebra over $T^*X^{(1)}$.
However, the corresponding fact for the quantized flag manifold $\CB_\zeta$ is a fairly non-trivial fact proved in
our earlier  work \cite[Theorem 6.1]{T1}
(see also Section \ref{subsec:Azumaya} below for a brief account of it).
For another type of  difficulty arising from the fact that the ring of differential operators on $\CB_q$ is more complicated than that 
of crystalline differential operators on an algebraic variety  in positive characteristic,
see Section \ref{sec:EV} below where 
the $R$-matrices are crucially used.

In the works \cite{TR}, \cite{TC}, accomplished after the present work  was completed, we also obtained some results on the category of $D$-modules on the quantized flag manifold.
In particular, in \cite{TR} the
 local structure of the ring of differential operators on the quantized flag manifold is intensively investigated.
We hope that we would be able to simplify some of the arguments in our preceding works using these results.

\subsection{}
For some reasons we assume 
\begin{itemize}
\item[(a1)]
$\ell>1$ is odd,
\item[(a2)]
$\ell$ is prime to the order of the center of $G$,
\item[(a3)]
$\ell$ is prime to 3 if $G$ is of type $G_2$
\end{itemize}
in the following.
In general the situation becomes more complicated when the parameter $q$ is specialized to a root of unity.
However, a good news in the case $\zeta$ is a root of unity is that we can relate the non-commutative scheme $\CB_\zeta$ with the ordinary scheme $\CB$ using 
Lusztig's quantum Frobenius homomorphism.
In fact, we have a morphism
$\Fr:\CB_\zeta\to\CB$ of non-commutative schemes, 
by which we obtain an actual sheaf of rings 
$\GO=\Fr_*\CO_{\CB_\zeta}$ on the ordinary flag manifold $\CB$.
The category $\modu(\CO_{\CB_\zeta})$ is naturally equivalent to the category 
$\modu(\GO)$ of coherent $\GO$-modules.
In the case $G=SL_2$ we have $\CB_\zeta=\CB=\BP^1$, and $\Fr:\BP^1\to\BP^1$ is given by
$z\mapsto z^{\ell}$ for $z\in\BC=\BA^1\subset\BP^1$.
In this case $\GO=\Fr_*\CO_{\BP^1}$ is commutative; however, in other cases where $G\ne SL_2$ the $\CO_\CB$-algebra $\GO$ turns out to be non-commutative.

\subsection{}
Let us consider $D$-modules on $\CB_\zeta$.
We have a category $\modu(\DD_{\CB_\zeta})$ of ``coherent $\DD_{\CB_\zeta}$-modules'', defined in the framework of non-commutative geometry; however, in our situation we can use $\GD=\Fr_*\DD_{\CB_\zeta}$, which is an actual sheaf of rings on $\CB$ so that $\modu(\DD_{\CB_\zeta})$ is naturally equivalent to the category $\modu(\GD)$ of coherent $\GD$-modules on $\CB$.
We have ring homomorphisms 
$\GO\to\GD$, $U_\zeta(\Gg)\to\GD$, $\CO(H)\to\GD$, 
where $\CO(H)$ is the affine coordinate algebra of a maximal torus $H$ of $G$.
Moreover, $\CO(H)\to\GD$ is a central embedding, and for $t\in H$ the specialization 
$\GD_t=\BC\otimes_{\CO(H)}\GD$ 
of $\GD$ with respect to $\CO(H)\xrightarrow{t}\BC$ is regarded as an analogue  of the ring of twisted differential operators in the ordinary case.

Let $Z_\Har(U_\zeta(\Gg))$ be the Harish-Chandra center of $U_\zeta(\Gg)$.
It is a central subalgebra of $U_\zeta(\Gg)$ isomorphic to 
$\CO(H/W\circ)$, where $W$ is the Weyl group and $H/W\circ$ denotes the quotient of $H$ with respect to a twisted action $W\times H\ni(w,t)\mapsto w\circ t\in H$ of $W$ on $H$.
For $t\in H$ we denote by $[t]$ its image in $H/W\circ$.
For $[t]\in H/W\circ$ we denote by 
$\xi_{[t]}:Z_\Har(U_\zeta(\Gg))\to\BC$ 
the
algebra homomorphism given by 
$\CO(H/W\circ)\xrightarrow{[t]}\BC$.
We set 
$U_\zeta(\Gg)_{[t]}=
\BC\otimes_{Z_\Har(U_\zeta(\Gg))}U_\zeta(\Gg)$ 
with respect to 
$\xi_{[t]}$.
Then the ring homomorphism $U_\zeta(\Gg)\to\GD$ induces $U_\zeta(\Gg)_{[t]}\to\GD_t$ for $t\in H$.

In general, for a ring 
(resp.\ a sheaf of rings on a topological space) $R$ we denote by 
$\modu(R)$ 
 the category of 
 finitely-generated 
 (resp.\ coherent)
 $R$-modules.
 For $t\in H$ we denote by 
 $\modu_t(\GD)$ 
 (resp.\  $\modu_{[t]}(U_\zeta(\Gg))$ 
 the category of coherent $\GD$-modules 
 (resp.\ finitely generated 
$U_\zeta(\Gg)$-modules) killed by some power of 
 the image of $\Ker(t:\CO(H)\to\BC)$ in $\GD$
 (resp.\ 
 $\Ker(\xi_{[t]})$ in $U_\zeta(\Gg)$).
Then we have functors
\begin{align}
\label{eq:I1}
&R\Gamma:
D^b(\modu(\GD_t))\to D^b(\modu(U_\zeta(\Gg)_{[t]})),
\\
\label{eq:I1a}
&R\Gamma:
D^b(\modu_t(\GD))\to D^b(\modu_{[t]}(U_\zeta(\Gg))).
\end{align}
Here, for  an abelian category $\GA$ we denote its bounded derived category by $D^b(\GA)$.
We conjecture that \eqref{eq:I1} (and hence \eqref{eq:I1a}) give equivalences of triangulated categories if $t\in H$ is regular in the sense that $|W\circ t|=|W|$.
In \cite{T2} and \cite{T3} we proved this conjecture in some cases.
In particular, it holds if the following conditions are satisfied:
\begin{itemize}
\item[(p)]
$\ell>1$ is a power of a prime,
\item[(b)]
$t^\ell$ has finite order which is prime to $\ell$.
\end{itemize}
The condition (b) is harmless from the view point of the representation theory
since we are reduced to the situation where (b) is satisfied using the parabolic induction (see \cite{DK2}).
But the assumption (p) should not be necessary.
In fact in the case $G=SL_n$ we can show the conjecture without assuming (p) (see \cite{TK}).
\subsection{}
Let us recall basics on the representations of $U_\zeta(\Gg)$.
Take Borel subgroups $B^+$, $B^-$ of $G$ such that $B^+\cap B^-=H$, and denote by $N^\pm$ the unipotent radical of $B^\pm$.
Let $Z_\Fr(U_\zeta(\Gg))$ be the Frobenius center of $U_\zeta(\Gg)$.
It is a central subalgebra of $U_\zeta(\Gg)$ isomorphic to the coordinate algebra $\CO(K)$ of the affine algebraic group
\[
K=\{(g_+h, g_-h^{-1})\mid
g_\pm\in N^\pm, h\in H\}
\subset B^+\times B^-
\]
(see \cite{DK1}, \cite{DP}). 
For $k\in K$ 
we denote by 
$\xi^k:Z_\Fr(U_\zeta(\Gg))\to\BC$ 
the
algebra homomorphism given by 
$\CO(K)\xrightarrow{k}\BC$.
We set 
$U_\zeta(\Gg)^k=
\BC\otimes_{Z_\Fr(U_\zeta(\Gg))}U_\zeta(\Gg)$ 
with respect to 
$\xi^k$, and consider $U_\zeta(\Gg)^k$-modules for each $k\in K$ 
(by a version of Schur's lemma any irreducible $U_\zeta(\Gg)$-module is a  $U_\zeta(\Gg)^k$-module for some $k\in K$).
Define $\eta:K\to G$ by $\eta(x_1,x_2)=x_1x_2^{-1}$.
By \cite{DKP} $U_\zeta(\Gg)^k$ is isomorphic to $U_\zeta(\Gg)^{k'}$ if 
$\eta(k)$ and $\eta(k')$ belong to the same conjugacy class of $G$.
Hence for each conjugacy class $O$ of $G$ one can choose one $k\in K$ satisfying $\eta(k)\in O$.
Set
\[
\CX=K\times_{H/W}(H/W\circ), 
\qquad
\CY=K\times_{H/W}H.
\]
Here, 
$K\to H/W$ is the composite of $\eta:K\to G$ with the Steinberg map $G\to H/W$, and 
$H/W\circ\to H/W$
(resp.\ 
$H\to H/W$)
 is given by 
 $[t]\mapsto [t^\ell]$
 (resp.\ 
$t\mapsto [t^\ell]$).
By \cite{DK1}, \cite{DP} 
(see also \cite{TZ}) the total center 
$Z(U_\zeta(\Gg))$ of $U_\zeta(\Gg)$ is generated by 
$Z_\Har(U_\zeta(\Gg))$ and $Z_\Fr(U_\zeta(\Gg))$.
 Furthermore, the algebra homomorphisms
 $\xi_{[t]}:Z_\Har(U_\zeta(\Gg))\to\BC$ ($[t]\in H/W\circ$) 
 and  
 $\xi^k:Z_\Fr(U_\zeta(\Gg))\to\BC$ ($k\in K$) 
are compatible if and only if $(k,[t])\in\CX$.
For $(k,[t])\in \CX$ we denote by 
$\xi^k_{[t]}:Z(U_\zeta(\Gg))\to\BC$ 
the algebra homomorphism satisfying 
$\xi^k_{[t]}|_{Z_\Har(U_\zeta(\Gg))}=\xi_{[t]}$, 
$\xi^k_{[t]}|_{Z_\Fr(U_\zeta(\Gg))}=\xi^k$.
We set
$U_\zeta(\Gg)_{[t]}^k=
\BC\otimes_{Z(U_\zeta(\Gg))}U_\zeta(\Gg)$ 
with respect to 
$\xi^k_{[t]}$.
We denote by
$\wU_\zeta(\Gg)^{\dot{k}}_{[t]}$ the completion of 
$U_\zeta(\Gg)$ at the maximal ideal 
$\Ker(\xi^{\dot{k}}_{[t]})$ of $Z(U_\zeta(\Gg))$.

\subsection{}
We will use the identification $\CB=B^-\backslash G$ in the following.
Set
\[
\CV=
\{(B^-g,k,t)\in\CB\times K\times H\mid
g\eta(k)g^{-1}\in t^\ell N^-\}.
\]
Let 
$p^{\CV}_H:\CV\to H$, 
$p^{\CV}_\CB:\CV\to \CB$ and 
$p^{\CV}_\CY:\CV\to \CY$ be the projections.
By \cite{T1} $\GD$ contains a central $\CO_\CB$-subalgebra isomorphic to $(p^\CV_\CB)_*\CO_\CV$.
Since $p^\CV_\CB$ is an affine morphism, we can consider the localization $\sD$ of $\GD$ on $\CV$.
It is an $\CO_\CV$-algebra satisfying $(p^\CV_\CB)_*\sD=\GD$.
For $(k,t)\in\CY$ we define a closed subvariety $\CV^k_t$ of $\CV$ by $\CV^k_t=(p^\CV_\CY)^{-1}(k,t)$, 
and denote by $\widehat{\CV}^k_t$ the formal neighborhood of $\CV^k_t$ in $\CV$.
We set
\[
\sD^k_t=\CO_{\CV^k_t}\otimes_{\CO_\CV}\sD,
\quad
\wsD^k_t=\CO_{\widehat{\CV}^k_t}\otimes_{\CO_\CV}\sD.
\]
They are
an $\CO_{\CV^k_t}$-algebra and 
an $\CO_{\widehat{\CV}^k_t}$-algebra
respectively.
We denote by 
$\modu^k_t(\sD)$
the category of 
coherent $\sD$-modules  supported on $\CV^k_t$.
We also denote by 
$\modu^k_{[t]}({U}_\zeta(\Gg))$
the category of 
finitely-generated 
${U}_\zeta(\Gg)$-modules
killed by some power of
$\Ker(\xi^k_{[t]})$.
Then we have embeddings 
\begin{align}
\label{eq:II1}
&\modu(\sD^k_t)
\hookrightarrow 
\modu^k_t(\sD)
\hookrightarrow 
\modu(\wsD^k_t),
\\
\label{eq:II2}
&\modu(U_\zeta(\Gg)^k_{[t]})
\hookrightarrow 
\modu^k_{[t]}(U_\zeta(\Gg))
\hookrightarrow 
\modu(\widehat{U}_\zeta(\Gg)^k_{[t]})
\end{align}
of abelian categories.
Moreover, \eqref{eq:II1} and \eqref{eq:II2} induce isomorphisms 
\begin{align}
\label{eq:III1}
&K(\modu(\sD^k_t))
\cong
K(\modu^k_t(\sD)),
\\
\label{eq:III2}
&K(\modu(U_\zeta(\Gg)^k_{[t]}))
\cong
K(\modu^k_{[t]}(U_\zeta(\Gg)))
\end{align}
of the Grothendieck groups.
Here, 
for an abelian category $\GA$ we denote its Grothendieck group by $K(\GA)$.
Note that \eqref{eq:I1} and  \eqref{eq:I1a} induce functors
\begin{align}
\label{eq:I2}
&R\Gamma:D^b(\modu(\sD^k_t))\to D^b(\modu(U_\zeta(\Gg)^k_{[t]})),
\\
\label{eq:I2aa}
&R\Gamma:D^b(\modu^k_t(\sD))\to D^b(\modu^k_{[t]}(U_\zeta(\Gg))),
\\
\label{eq:I2b}
&R\Gamma:D^b(\modu(\wsD^k_t))\to D^b(\modu(\widehat{U}_\zeta(\Gg)^k_{[t]})).
\end{align}

The main result of \cite{T1} is the split Azumaya property of 
$\sD^k_t$ and
$\wsD^k_t$
for $(k,t)\in\CY$ under a certain condition on $\ell$ depending on $(k,t)\in\CY$.
The condition is void if $\eta(k)$ is unipotent.
In order that the property holds for any $(k,t)\in\CY$ we need to assume 
\begin{itemize}
\item[(c1)]
$\ell$ is prime to 3 if $G$ is of type $F_4$, $E_6$, $E_7$, $E_8$;
\item[(c2)]
$\ell$ is prime to 5 if $G$ is of type $E_8$.
\end{itemize}
If (c1), (c2) are satisfied, then 
there exist 
a locally free $\CO_{\CV^k_t}$-module $\GK^k_t$ 
and
a locally free $\CO_{\widehat{\CV}^k_t}$-module $\widehat{\GK}^k_t$ 
of finite rank such that 
$\sD^k_t\cong \CEnd_{\CO_{\CV^k_t}}(\GK^k_t)$ and
$\wsD^k_t\cong \CEnd_{\CO_{\widehat{\CV}^k_t}}(\widehat{\GK}^k_t)$.
Hence we have equivalences
\begin{gather*}
\modu(\sD^k_t)
\cong
\modu(\CO_{\CV^k_t}),
\quad
\modu(\wsD^k_t)
\cong
\modu(\CO_{\widehat{\CV}^k_t})
\end{gather*}
by the Morita theory.
By the second equivalence 
we have also
\[
\modu^k_t(\sD)
\cong
\modu^k_t(\CO_{\CV}),
\]
where
$\modu^k_t(\CO_{\CV})$
denotes the category of coherent $\CO_{\CV}$-modules
supported on $\CV^k_t$.

We assume in the following that $\ell$ satisfies (p) in addition to (a1), (a2), (a3).
As noted above, \eqref{eq:I1} is an equivalence if $t$ is regular and satisfies (b).
In this case we can also show that 
\eqref{eq:I2aa} and \eqref{eq:I2b} give equivalences of triangulated categories.
Therefore, we obtain the following equivalences
of triangulated categories 
\begin{align}
\label{eq:I3b}
D^b(\modu^k_t(\CO_{{\CV}}))
&\cong
D^b(\modu^k_{[t]}({U}_\zeta(\Gg))),
\\
\label{eq:I3}
D^b(\modu(\CO_{\widehat{\CV}^k_t}))
&\cong
D^b(\modu(\widehat{U}_\zeta(\Gg)^k_{[t]}))
\end{align}
if $t$ is regular and the conditions (b), (c1), (c2) are satisfied.

\subsection{}
The aim of the present paper is to give a more refined version of \eqref{eq:I3} in terms of abelian categories. 
We fix $\dot{k}\in K$ such that 
 $\dot{x}=\eta(\dot{k})$ is a unipotent element of $G$,
 and consider the abelian category $\modu(U_\zeta(\Gg)^{\dot{k}}_{[t]})$ and its variants 
for $t\in H$ such that $(\dot{k},t)\in\CY$.
We do not need to assume (c1), (c2) since 
$\eta(\dot{k})$ is unipotent.
Moreover, 
we have $(\dot{k},t)\in\CY$  for $t\in H$  if and only if  $t^\ell=1$.
Hence any $t$ such that $(\dot{k},t)\in\CY$ satisfies the condition (b).
Denote by $H_\ell$ the set of $t\in H$ satisfying $t^\ell=1$.
The set of regular elements in $H_\ell$ is denoted by $H_\ell^\reg$.
Note that for $t\in H_\ell$ the variety $\CV^{\dot{k}}_t$ is naturally isomorphic to the Springer fiber
$\CB^{\dot{x}}=\{B^-g\in\CB\mid g\dot{x}g^{-1}\in N^-\}$, 
and 
$\widehat{\CV}^{\dot{k}}_t$ is naturally isomorphic to the formal neighborhood
$\wCB^{\dot{x}}$  of $\CB^{\dot{x}}$ in $\tilde{G}=\{(B^-g,x)\in\CB\times G\mid
gxg^{-1}\in B^-\}$.
Hence we have an equivalence 
\begin{equation}
\label{eq:I4}
D^b(\modu(\widehat{U}_\zeta(\Gg)^{\dot{k}}_{[t]}))
\cong
D^b(\modu(\CO_{\widehat{\CB}^{\dot{x}}}))
\qquad(t\in H^\reg_\ell)
\end{equation}
of derived categories.
Denote by
$\modu_{\CB^{\dot{x}}}(\CO_{\tilde{G}})$ 
the category of coherent $\CO_{\tilde{G}}$-modules 
(set-theoretically) supported on $\CB^{\dot{x}}$.
We have 
$\modu^{\dot{k}}_t(\CO_{\CV})
\cong
\modu_{\CB^{\dot{x}}}(\CO_{\tilde{G}})$, and hence
\begin{equation}
\label{eq:I4a}
D^b(\modu^{\dot{k}}_{[t]}({U}_\zeta(\Gg)))
\cong
D^b(\modu_{\CB^{\dot{x}}}(\CO_{\tilde{G}}))
\qquad(t\in H^\reg_\ell).
\end{equation}
Note that to ensure $H^\reg_\ell\ne\emptyset$ we need the condition 
\begin{itemize}
\item[(a4)]
$\ell\geqq h_G$,
\end{itemize}
where $h_G$ is the Coxeter number.
We assume (a4) in the following.

Note that we have a non-standard $t$-structure of 
$D^b(\modu(\CO_{\wCB^{\dot{x}}}))$ called the exotic $t$-structure (see \cite{Be}, \cite{BM}).
We denote its heart by $\modu^\ex(\CO_{\wCB^{\dot{x}}})$.
Our main result is the following.
\begin{theorem}
\label{thm:I1}
The equivalence \eqref{eq:I4} of the triangulated categories induces an equivalence
\begin{equation}
\label{eq:I5}
\modu(\wU_\zeta(\Gg)^{\dot{k}}_{[t]})
\cong
\modu^\ex(\CO_{\wCB^{\dot{x}}})
\qquad(t\in H^\reg_\ell).
\end{equation}
of abelian categories.
\end{theorem}

The exotic $t$-structure is defined using an affine braid group action on $D^b(\modu(\CO_{\wCB^{\dot{x}}}))$.
Theorem \ref{thm:I1} is proved by showing that the corresponding affine braid group action on $D^b(\modu(\wU_\zeta(\Gg)^{\dot{k}}_{[t]}))$ is 
given by the wall-crossing functor.

\subsection{}
To give an application to the representation theory we also need an equivariant version of \eqref{eq:I5}.
By replacing $\dot{k}$ with suitable $\dot{k}'\in K$ such $\eta(\dot{k}')$ is conjugate to $\eta(\dot{k})$ we may assume from the beginning that there exists a maximal torus $C$ of the centralizer $Z_G(\dot{x})$ such that $C\subset H$.
Let $\Lambda_C$ be the character group of $C$.
The natural grading of $U_\zeta(\Gg)$ by weights induces a $\Lambda_C$-grading of $\wU_\zeta(\Gg)^{\dot{k}}_{[t]}$.
We fix $t\in H_\ell^\reg$ in the following.
Then \eqref{eq:I5} induces an equivalence
\begin{equation}
\label{eq:I6}
\modu(\wU_\zeta(\Gg)^{\dot{k}}_{[t]};C)
\cong
\modu^\ex(\CO_{\wCB^{\dot{x}}};C)
\end{equation}
of abelian categories.
Here, 
$\modu(\wU_\zeta(\Gg)^{\dot{k}}_{[t]};C)$
denotes the category of finitely generated $\Lambda_C$-graded 
$\wU_\zeta(\Gg)^{\dot{k}}_{[t]}$-modules
satisfying a certain compatibility condition, 
and 
$\modu^\ex(\CO_{\wCB^{\dot{x}}};C)$
denotes the category of $C$-equivariant exotic coherent $\CO_{\wCB^{\dot{x}}}$-modules.

 \subsection{}
 Let us describe application of \eqref{eq:I6} to the representation theory.
We denote by $\modu_{[t]}(U_\zeta(\Gg)^{\dot{k}})$ 
the category of finitely generated $U_\zeta(\Gg)^{\dot{k}}$-modules 
killed by some power of the image of 
$\Ker(\xi_{[t]})$ in $U_\zeta(\Gg)^{\dot{k}}$.
We also denote by 
$\modu_{[t]}(U_\zeta(\Gg)^{\dot{k}};C)$
(resp.\
$\modu^{\dot{k}}_{[t]}(U_\zeta(\Gg);C)$)
the category of $\Lambda_C$-graded objects in  
$\modu_{[t]}(U_\zeta(\Gg)^{\dot{k}})$
(resp.\
$\modu^{\dot{k}}_{[t]}(U_\zeta(\Gg))$)
satisfying a certain compatibility condition. We have embeddings
\[
\modu(U_\zeta(\Gg)_{[t]}^{\dot{k}};C)
\subset
\modu_{[t]}(U_\zeta(\Gg)^{\dot{k}};C)
\subset
\modu_{[t]}^{\dot{k}}(U_\zeta(\Gg);C)
\subset
\modu(\wU_\zeta(\Gg)^{\dot{k}}_{[t]};C)
\]
of abelian categories.
Denote by $\modu(\CO_{\CB^{\dot{x}}};C)$ and 
$\modu_{\CB^{\dot{x}}}(\CO_{\tilde{G}};C)$ 
the category of $C$-equivariant objects in 
$\modu(\CO_{\CB^{\dot{x}}})$ and 
$\modu_{\CB^{\dot{x}}}(\CO_{\tilde{G}})$ respectively.
Then we have embeddings
\[
\modu(\CO_{\CB^{\dot{x}}};C)
\subset
\modu_{\CB^{\dot{x}}}(\CO_{\tilde{G}};C)
\subset
\modu(\CO_{\wCB^{\dot{x}}};C)
\]
of abelian categories.
Moreover, we have the following commutative diagram of functors:
\[
\xymatrix{
D^b(\modu(\CO_{\CB^{\dot{x}}};C))
\ar@{^{(}->}[r]^{}
\ar[d]
&
D^b(
\modu_{\CB^{\dot{x}}}(\CO_{\tilde{G}};C))
\ar@{^{(}->}[r]^{}
\ar[d]^{\cong}
&
D^b(
\modu(\CO_{\wCB^{\dot{x}}};C))
\ar[d]^{\cong}
\\
D^b(
\modu(U_\zeta(\Gg)_{[t]}^{\dot{k}};C))
\ar@{^{(}->}[r]^{}
&
D^b(
\modu_{[t]}^{\dot{k}}(U_\zeta(\Gg);C))
\ar@{^{(}->}[r]^{}
&
D^b(
\modu(\wU_\zeta(\Gg)^{\dot{k}}_{[t]};C)
).
}
\]

Let $\{L_\sigma\}_{\sigma\in\Theta}$ be the set of irreducible objects of $\modu(U_\zeta(\Gg)_{[t]}^{\dot{k}};C)$.
We denote by $E_\sigma$ (resp.\ $\wE_\sigma$) the projective cover of $L_\sigma$ in 
$\modu_{[t]}(U_\zeta(\Gg)^{\dot{k}};C)$ 
(resp.\ $\modu(\wU_\zeta(\Gg)^{\dot{k}}_{[t]};C)$).
Lusztig's conjecture gives a formula 
expressing $[E_\sigma]$ as a $\BZ$-linear combination of the elements $[L_\tau]$ for $\tau\in\Theta$ 
in the Grothendieck group
\[
K(\modu(U_\zeta(\Gg)_{[t]}^{\dot{k}};C))
\cong
K(\modu_{[t]}(U_\zeta(\Gg)^{\dot{k}};C))
\cong
K(\modu_{[t]}^{\dot{k}}(U_\zeta(\Gg);C)).
\]
For $\sigma\in\Theta$ we denote by $\CL_\sigma$, 
$\CE_\sigma$, $\wCE_\sigma$ the objects of 
$\modu^\ex(\CO_{\wCB^{\dot{x}}};C)$
corresponding to 
$L_\sigma$, $E_\sigma$, $\wE_\sigma$ respectively.
By $L_\sigma, E_\sigma\in \modu_{[t]}^{\dot{k}}(U_\zeta(\Gg);C)$
we have
$\CL_\sigma, \CE_\sigma\in D^b(
\modu_{\CB^{\dot{x}}}(\CO_{\tilde{G}};C))$.
Our problem is 
to give a  formula 
expressing $[\CE_\sigma]$ as a $\BZ$-linear combination of the elements $[\CL_\tau]$ for $\tau\in\Theta$ 
in the Grothendieck group
\[
K(\modu(\CO_{\CB^{\dot{x}}};C))
\cong
K(\modu_{\CB^{\dot{x}}}(\CO_{\tilde{G}};C)).
\]

Note that 
 $\{\CL_\sigma\}_{\sigma\in\Theta}$ is the set of irreducible objects of $\modu^\ex(\CO_{\wCB^{\dot{x}}};C)$ contained in 
 $D^b(\modu_{\CB^{\dot{x}}}(\CO_{\tilde{G}};C))$, 
 and $\wCE_\sigma$ is the projective cover of $\CL_\sigma$ in 
 $\modu^\ex(\CO_{\wCB^{\dot{x}}};C)$.
 Moreover, by 
 \[
 E_\sigma=\BC\otimes_{Z_\Fr(U_\zeta(\Gg))}\wE_\sigma
 =
 \BC\otimes^L_{Z_\Fr(U_\zeta(\Gg))}\wE_\sigma
 \] 
 with respect to $\xi^{\dot{k}}$, we have 
 $\CE_\sigma=\CO_{\{\dot{x}\}}\otimes^L_{\CO_G}\wCE_\sigma$, 
 where $\wCB^{\dot{x}}\to G$ is induced by the projection $\tilde{G}\to G$.
 From these facts together with the result of \cite{BM} we obtain the desired formula 
expressing $[\CE_\sigma]$ as a $\BZ$-linear combination of the elements $[\CL_\tau]$ for $\tau\in\Theta$ 
 in terms of a certain geometrically defined bilinear form and Lusztig's conjectural canonical bases of certain equivariant $K$-groups (whose existence is proved in \cite{BM}).
See Section \ref{sec:Lusztig} below for more details.

In the case $\dot{x}$ is a regular unipotent element of the reductive part of a parabolic subgroup of $G$, this formula  implies a combinatorial character formula of $L_\sigma$ in terms of baby Verma modules (see \cite{LP} and \cite[Section 10]{LK}).
We note that a recent work of Bezrukavnikov-Losev \cite{BL} should shed light to the character formula in the general situation.

In the special case $\dot{x}=1$ we obtain a new proof of the character formula of the irreducible highest weight modules over the Lusztig type quantized enveloping algebra for $\ell$ satisfying (a1), (a2), (a3), (a4),  and (p),
which was originally proved through the combination of 
\cite{KL}, \cite{LM} and \cite{KT} (see also \cite{ABG}).

We also note that our result gives a proof of a conjecture of Bezurukavnikov-Mirkovi\'{c} (see 
\cite[1.7.1]{BM}) for $\ell$ satisfying (a1), (a2), (a3), (a4), (p).'

\subsection{}
As described above we proved, under certain conditions on the multiplicative order $\ell$ of $\zeta$, Lusztig's conjectural multiplicity formula for regular blocks in $\modu(U_\zeta(\Gg)^k)$  in the case where  
$\eta(k)\in G$ is unipotent.
In general we should consider the Jordan decomposition 
$\eta(k)=\eta(k)_s\eta(k)_u$ so that $\eta(k)_u$ is the unipotent element of the centralizer $G'$ of $\eta(k)_s$ in $G$, and we can formulate a similar conjectural multiplicity formula 
using $G'$ instead of $G$.
This general formula follows from the unipotent case by the theory of parabolic induction 
(see \cite{DK2}) 
only in the case where $G'$ is a Levi subgroup of some parabolic subgroup of $G$.
We do not know whether it is true in general.

In the case of Lie algebras in positive characteristics the centralizer of a semisimple element is always a Levi subgroup of a parabolic subgroup and we do not meet the above mentioned situation.
In \cite{BMR2} Bezrukavnikov, 
Mirkovi\'{c} and 
Rumynin established a framework working in the case where the Harish-Chandra central character is singular  using partial flag manifolds.
This cannot be generalized to the quantum case by the same reason.

\subsection{}
The organization of this paper is as follows.
In Section 2 we introduce basic notation on the quantized enveloping algebras and 
the quantized coordinate algebras.
In Section 3 and 4 we recall fundamental results on the quantized flag manifold and $D$-modules on it in the situation where the parameter $q$ is specialized to an arbitrary non-zero complex number.
We assume that $q$ is specialized to an $\ell$-th root of 1 starting from Section 5.
In Section 5 we recall the structure theorem for the center of the quantized enveloping algebra. 
We also give a formula concerning the action the Harish-Chandra center on the tensor product of a $U_\zeta(\Gg)$-module and an integrable module over the Lusztig type quantized enveloping algebra.
This is an 
analogue of Kostant's result on the tensor product of a $\Gg$-module and a finite-dimensional $\Gg$-module.
This plays essential role in the investigation of the translation functor.
In Section 6 we use Frobenius morphism $\Fr:\CB_\zeta\to\CB$  to reformulate results on $D$-modules on $\CB_\zeta$ using 
$\GD$-modules and $\sD$-modules.
In Section 7 we recall the Beilinson-Bernstein type correspondence given in \cite{T2}, \cite{T3}.
In Section 8 we formulate a torus equivariant version of the correspondence.
In Section 9 after recalling the notion of the exotic sheaves 
we give a proof of the main theorem.
An essential role is played by the wall crossing functor which is closely related to the translation functor.
In Section 10 we explain how we can deduce  Lusztig's conjecture on non-restricted representations from the main theorem.

\subsection{}
If $R$ is a ring,  we denote its center by $Z(R)$.
We also denote by
$\Mod(R)$ 
(resp.\ $\Mod^r(R)$) 
the category of left
(resp.\ right) 
$R$-modules.
Its subcategory consisting of 
finitely generated left
(resp.\ right) 
$R$-modules
is denoted by $\modu(R)$
(resp.\ $\modu^r(R)$).

If $\Gamma$ is a free abelian group of finite rank and $R$ is a $\Gamma$-graded ring,  we denote by 
$\Mod_\Gamma(R)$ 
(resp. $\modu_\Gamma(R)$) 
the category of $\Gamma$-graded left $R$-modules
(resp. $\Gamma$-graded finitely generated left $R$-modules).

If $\CR$ is a sheaf of rings on a topological space, we denote by $\Mod(\CR)$ the category of quasi-coherent left $\CR$-modules.
Its subcategory consisting of coherent $\CR$-modules is denoted by $\modu(\CR)$.

The structure sheaf of a scheme $X$ is denoted by $\CO_X$.
If $X$ is an affine algebraic variety, its coordinate algebra $\Gamma(X,\CO_X)$ is denoted by $\CO(X)$.

If $U$ is a Hopf algebra, we sometimes use Sweedler's notation
\[
\Delta^{(n)}(u)
=\sum_{(u)}u_{(0)}\otimes\cdots\otimes u_{(n)}
\qquad(u\in U)
\]
for the iterated comultiplication 
$\Delta^{(n)}:U\to U^{\otimes n+1}$.

\section{{Quantum groups}}
\subsection{Simple Lie algebra}
Let $G$ be a connected, simply-connected, simple algebraic group 
over $\BC$ with maximal torus $H$, 
and let $\Gg$ and $\Gh$ be the Lie algebras of $G$ and $H$ respectively.
We denote by $\Delta\subset \Gh^*$ the root system.
Let $W\subset GL(\Gh^*)$ be the Weyl group, and let
\begin{equation}
\label{eq:bil}
(\;,\;):\Gh^*\times\Gh^*\to\BC
\end{equation}
be the $W$-invariant symmetric bilinear form such that $(\alpha,\alpha)=2$ for short roots $\alpha$.
For $\alpha\in\Delta$ the corresponding coroot is given by $\alpha^\vee=2\alpha/(\alpha,\alpha)\in\Gh^*$.
We denote by $Q$ and $\Lambda$ the root lattice and the weight lattice respectively.
We also set $Q^\vee=\sum_{\alpha\in \Delta^\vee}\BZ\alpha^\vee$.
Fix a set $\{\alpha_i\}_{i\in I}$ of simple roots, and denote by $\Delta^+$ the set of positive roots.
We set 
\[
Q^+=\sum_{\alpha\in\Delta^+}\BZ\alpha,
\qquad
\Lambda^+=\{\lambda\in\Lambda
\mid
(\lambda,\alpha^\vee)\geqq0\;\;(\alpha\in\Delta^+)\}.
\]
For $i\in I$ we denote by $s_i\in W$ the simple reflection corresponding to $i$.
For $\lambda\in\Lambda$ we denote the corresponding character of $H$ by
\begin{equation}
\label{eq:theta}
\theta_\lambda:H\to\BC^\times.
\end{equation}
We define subalgebras
 $\Gn^+$, $\Gn^-$,  $\Gb^+$, $\Gb^-$ of $\Gg$ by 
\[
\Gn^\pm=
\bigoplus_{\alpha\in\Delta^+}\Gg_{\pm\alpha}
,
\qquad
\Gb^\pm=\Gh\oplus\Gn^\pm,
\]
where $\Gg_\alpha$ for $\alpha\in\Delta$ denotes the root subspace.
The corresponding connected closed subgroups of $G$ are denoted as 
$N^+$, $N^-$,  $B^+$, $B^-$ respectively.
The variety 
\[
\CB=B^-\backslash G
\]
is called the flag manifold for $G$.

For $J\subset I$
we set 
\begin{align*}
\Delta_J=&\left(\sum_{j\in J}\BZ\alpha_j\right)\cap\Delta,
\qquad
\Delta_J^+=\Delta_J\cap\Delta^+,
\end{align*}
\begin{align*}
\Lambda^J=\{\lambda\in \Lambda\mid
(\lambda,\alpha_j^\vee)=0
\;\;(j\in J)\},
\qquad
\Lambda^{J+}=\Lambda^J\cap\Lambda^+.
\end{align*}
We define subalgebras
 $\Gl_J$,  $\Gn_J^\pm$, $\Gp_J^-$, $\Gb_J^-$ of $\Gg$ by 
\begin{gather*}
\Gl_J=\Gh\oplus\left(
\bigoplus_{\alpha\in\Delta_J}\Gg_{\alpha}
\right),
\qquad
\Gn_J^\pm=\bigoplus_{\alpha\in\Delta^+\setminus\Delta_J}
\Gg_{\pm\alpha},
\\
\Gp_J^-=\Gl_J\oplus\Gn_J^-,
\qquad
\Gb_J^-=\Gh\oplus\left(
\bigoplus_{\alpha\in\Delta_J^+}\Gg_{-\alpha}
\right)\subset\Gl_J.
\end{gather*}
The corresponding connected closed subgroups of $G$ are denoted as 
$L_J$, $N_J^\pm$, $P_J^-$, $B_J^-$ respectively.
The variety 
\[
\CP_J=P_J^-\backslash G
\]
is called the partial flag manifold.
For $J_1\subset J_2\subset I$ we denote by 
\[
\pi^{J_2J_1}:\CP_{J_1}\to\CP_{J_2}
\]
the natural morphism of algebraic varieties.
We write 
\[
\pi^J=\pi^{J\emptyset}:\CB\to\CP_{J}.
\]

\subsection{Quantized enveloping algebra}
For $i\in I$ we set $d_i=(\alpha_i,\alpha_i)/2\in\BZ$, and 
for $i, j\in I$ we set $a_{ij}=(\alpha_i^\vee,\alpha_j)\in\BZ$.
Let $\BF=\BQ(q^{1/|\Lambda/Q|})$ be the field of rational functions over $\BQ$ in the variable $q^{1/|\Lambda/Q|}$.
We define $U_\BF(\Gg)$ to be the $\BF$-algebra with $1$ generated by the elements 
$k_\lambda$ $(\lambda\in \Lambda)$, 
$e_i$, $f_i$ $(i\in I)$ satisfying the fundamental relations
\begin{align}
&
k_0=1, \qquad
k_\lambda k_{\mu}=k_{\lambda+\mu}
&(\lambda, \mu\in\Lambda),
\\
&k_\lambda e_i=q^{(\lambda,\alpha_i)}e_ik_\lambda,
\qquad
k_\lambda f_i=q^{-(\lambda,\alpha_i)}f_ik_\lambda
&(i\in I, \;\lambda\in\Lambda),
\\
&
e_if_j-f_je_i=
\delta_{ij}
\frac{k_i-k_i^{-1}}{q_i-q_i^{-1}}
&(i, j\in I),
\\
&
\sum_{r+s=1-a_{ij}}
(-1)^{s}e_i^{(r)}e_je_i^{(s)}=
\sum_{r+s=1-a_{ij}}
(-1)^{s}f_i^{(r)}f_jf_i^{(s)}=0
&(i, j\in I, i\ne j),
\end{align}
where
$q_i=q^{d_i}$, $k_i=k_{\alpha_i}$ for $i\in I$, and 
$e_i^{(r)}=e_i^r/[r]_{q_i}!$, $f_i^{(r)}=f_i^r/[r]_{q_i}!$ for $i\in I$, $r\in\BZ_{\geqq0}$.
Here,
\[
[r]_t!=\prod_{n=1}^r\frac{t^n-t^{-n}}{t-t^{-1}}
\in\BZ[t,t^{-1}].
\]
Then $U_\BF(\Gg)$ turns out to be a Hopf algebra 
by
\begin{align}
&
\Delta(k_\lambda)=k_\lambda\otimes k_\lambda,
\quad
\Delta(e_i)=e_i\otimes1+k_i\otimes e_i,
\quad
\Delta(f_i)=f_i\otimes k_i^{-1}+1\otimes f_i,
\\
&\varepsilon(k_\lambda)=1,
\quad
\varepsilon(e_i)=
\varepsilon(f_i)=0,
\\
&
S(k_\lambda)=k_{-\lambda},
\quad
S(e_i)=-k_i^{-1}e_i,
\quad
S(f_i)=-f_ik_i
\end{align}
for $\lambda\in\Lambda$, $i\in I$.
We define a subalgebra $U_\BF(\Gh)$ 
of $U_\BF(\Gg)$ 
by
\[
U_\BF(\Gh)=\langle
k_\lambda\mid\lambda\in\Lambda\rangle.
\]
Then we have
$
U_\BF(\Gh)=\bigoplus_{\lambda\in\Lambda}\BF k_\lambda.
$
For $\lambda\in\Lambda$ we define a character
\begin{equation}
\label{eq:chi1}
\chi_\lambda:U_\BF(\Gh)\to\BF
\end{equation}
by 
$\chi_\lambda(k_\mu)=q^{(\lambda,\mu)}
=(q^{1/|\Lambda/Q|})^{(\lambda,\mu)|\Lambda/Q|}$ 
($\lambda, \mu\in\Lambda$).
Note $(\Lambda,\Lambda)\subset\frac1{|\Lambda/Q|}\BZ$ by our choice of \eqref{eq:bil}.

Set $\BA=\BZ[q^{1/{|\Lambda/Q|}},q^{-1/{|\Lambda/Q|}}]$.
We define $\BA$-forms $U_\BA(\Gg)$ and $U_\BA^L(\Gg)$ of 
$U_\BF(\Gg)$ as follows.
The De Concini-Kac form $U_\BA(\Gg)$ is the $\BA$-subalgebra of $U_\BF(\Gg)$ generated by 
$k_\lambda$\;($\lambda\in \Lambda$), 
$e_i$, $f_i$ ($i\in I$).
Set 
\[
U_\BA^L(\Gh)=\{f\in U_\BF(\Gh)\mid
\chi_\lambda(h)\in\BA\;(\forall\lambda\in\Lambda)\}.
\]
The Lusztig form 
$U_\BA^L(\Gg)$ is the $\BA$-subalgebra of $U_\BF(\Gg)$
generated by
$U_\BA^L(\Gh)$, 
$e_i^{(r)}$, $f_i^{(r)}$ $(i\in I, r\in\BZ_{\geqq0})$.
Then 
$U_\BA(\Gg)$ and $U_\BA^L(\Gg)$ are Hopf algebras over $\BA$.

Now assume that we are given $\zeta\in\BC^\times$ together with a choice of $\zeta^{1/{|\Lambda/Q|}}\in\BC^\times$ satisfying 
$(\zeta^{1/{|\Lambda/Q|}})^{|\Lambda/Q|}=\zeta$.
We define Hopf algebras 
$U_\zeta(\Gg)$ and $U_\zeta^L(\Gg)$ over $\BC$ by
\begin{equation}
U_\zeta(\Gg)=\BC\otimes_\BA U_\BA(\Gg),
\qquad
U_\zeta^L(\Gg)=\BC\otimes_\BA U_\BA^L(\Gg)
\end{equation}
with respect to $\BA\to\BC$ ($q^{1/{|\Lambda/Q|}}\mapsto\zeta^{1/{|\Lambda/Q|}}$).

\subsection{Harish-Chandra center}
Define subalgebras $U_\zeta(\Gh)$,
$U_\zeta(\Gn^+)$, 
$U_\zeta(\Gn^-)$ of 
$U_\zeta(\Gg)$ by
\[
U_\zeta(\Gh)=\langle k_\lambda\mid\lambda\in\Lambda\rangle,
\qquad
U_\zeta(\Gn^+)=\langle
e_i\mid
i\in I\rangle,
\qquad
U_\zeta(\Gn^-)=\langle
f_i\mid
i\in I\rangle.
\]
Then we have an isomorphism
\[
U_\zeta(\Gg)
\cong
U_\zeta(\Gn^-)
\otimes
U_\zeta(\Gh)
\otimes
U_\zeta(\Gn^+)
\qquad
(xyz\leftrightarrow x\otimes y\otimes z)
\]
of $\BC$-modules.
Moreover, we have
$
U_\zeta(\Gh)=
\bigoplus_{\lambda\in \Lambda}
\BC k_\lambda.
$
We set
\begin{equation}
{}^eU_\zeta(\Gh)
=
\bigoplus_{\lambda\in \Lambda}
\BC k_{2\lambda},
\end{equation}
and define a twisted action of $W$ on 
${}^eU_\zeta(\Gh)$ by
\begin{equation}
w\circ k_{2\lambda}=
\zeta^{2(w\lambda-\lambda,\rho)}k_{2w\lambda}
\qquad(w\in W, \lambda\in\Lambda),
\end{equation}
where $\rho\in\Lambda$ is defined by 
$(\rho,\alpha_i^\vee)=1$ for any $i\in I$.

We define $Z_\Har(U_\zeta(\Gg))$ 
to be the $\BC$-subalgebra of
$Z(U_\zeta(\Gg))$ generated by the image of 
$
Z(U_\BA(\Gg))\to U_\zeta(\Gg)
$.
We call
$Z_\Har(U_\zeta(\Gg))$
the Harish-Chandra center of $U_\zeta(\Gg)$.
Define 
\[
\Xi:Z_\Har(U_\zeta(\Gg))\to 
U_\zeta(\Gh)
\]
as the composite of 
\[
Z_\Har(U_\zeta(\Gg))
\hookrightarrow
U_\zeta(\Gg)
\cong
U_\zeta(\Gn^-)
\otimes
U_\zeta(\Gh)
\otimes
U_\zeta(\Gn^+)
\xrightarrow{\varepsilon\otimes1\otimes\varepsilon}
U_\zeta(\Gh).
\]
The following is well-known.
\begin{proposition}
The linear map $\Xi$ is an injective algebra homomorphism whose image is equal to
\[
{}^eU_\zeta(\Gh)^{W\circ}
=
\{x\in {}^eU_\zeta(\Gh)\mid
w\circ x=x\;(\forall w\in W)\}.
\]
\end{proposition}

We identify the coordinate algebra $\CO(H)$  of $H$ with 
${}^eU_\zeta(\Gh)$ via $\theta_\lambda\leftrightarrow k_{2\lambda}$.
Then ${}^eU_\zeta(\Gh)^{W\circ}$ is identified with 
$\CO(H)^{W\circ}$ with respect to the twisted action of $W$ on $H$ given by
\begin{equation}
\theta_\lambda(w\circ t)
=
\zeta^{2(w^{-1}\lambda-\lambda,\rho)}
\theta_{w^{-1}\lambda}(t)
\qquad(t\in H,\; w\in W,\; \lambda\in\Lambda).
\end{equation}
Therefore, the characters of 
$Z_\Har(U_\zeta(\Gg))$ are in one-to-one correspondence with the points of $H/W\circ$.
For $[t]\in H/W\circ$ represented by $t\in H$ we denote by
\begin{equation}
\xi_{[t]}:Z_\Har(U_\zeta(\Gg))\to\BC
\end{equation}
the corresponding character, and set
\begin{equation}
\label{eq:Ut}
U_\zeta(\Gg)_{[t]}
=U_\zeta(\Gg)/U_\zeta(\Gg)\Ker(\xi_{[t]}).
\end{equation}

\subsection{Quantized coordinate algebra}
We say that 
a left
(resp.\ right) 
$U_\BF(\Gg)$-module $M$ is integrable if it is a direct sum of weight spaces
\[
M_\mu=\{m\in M
\mid
am=\chi_\mu(a)m
\;(
\text{resp.}\;
ma=\chi_\mu(a)m)
\;\;
\text{for}\: a\in U_\BF(\Gh)\}
\]
for $\mu\in\Lambda$, and if for any $m\in M$ there exists $n>0$ such that
$
e_i^{n}m=f_i^{n}m=0
$
(resp.\
$
me_i^{n}=mf_i^{n}=0
$)
for $i\in I$.
For $\lambda\in\Lambda^+$ we denote by $\Delta_\BF(\lambda)$ the irreducible (left) $U_\BF(\Gg)$-module with highest weight $\lambda$.
It is known that integrable irreducible 
$U_\BF(\Gg)$-modules are in one-to-one correspondence with $\Lambda^+$ via $\Delta_\BF(\lambda)\leftrightarrow\lambda$.
Define an $\BA$-form $\Delta_\BA(\lambda)$ of 
$\Delta_\BF(\lambda)$ by
$\Delta_\BA(\lambda)=U^L_\BA(\Gg)v_\lambda$, where 
$v_\lambda$ is the highest weight vector of 
$\Delta_\BF(\lambda)$.

Set
$
U_\BF(\Gg)^*
=
\Hom_\BF(U_\BF(\Gg),\BF)
$.
It is a 
$U_\BF(\Gg)$-bimodule by
\begin{equation}
\label{eq:bim}
\langle
u_1\varphi u_2,u\rangle
=
\langle
\varphi, u_2uu_1
\rangle
\qquad(
\varphi\in U_\BF(\Gg)^*,\;
u, u_1, u_2\in
U_\BF(\Gg))
.
\end{equation}
We denote by $\CO_\BF(G)$ the 
subspace of 
$U_\BF(\Gg)^*$ consisting of 
$\varphi\in U_\BF(\Gg)^*$ 
satisfying the following equivalent conditions:
\begin{itemize}
\item[(1)]
the left $U_\BF(\Gg)$-module 
$U_\BF(\Gg)\varphi$ is integrable,
\item[(2)]
the right $U_\BF(\Gg)$-module 
$\varphi U_\BF(\Gg)$ is integrable.
\end{itemize}
We have a natural Hopf algebra structure of
$\CO_\BF(G)$ such that
the canonical pairing $U_\BF(\Gg)\times\CO_\BF(G)\to\BF$ turns out to be a Hopf pairing.

We set 
\[
\CO_\BA(G)=
\{
\varphi\in\CO_\BF(G)
\mid
\langle\varphi,U^L_\BA(\Gg)\rangle\subset\BA\}.
\]
It is a Hopf algebra over $\BA$ such that $\BF\otimes_\BA\CO_\BA(G)\cong\CO_\BF(G)$.
Moreover, it is a $U^L_\BA(\Gg)$-bimodule
(see e.g.\ \cite{T1}).

We define a subalgebra $U_\zeta^L(\Gh)$ of $U^L_\zeta(\Gg)$
to be the linear span of $\Image(U_\BA^L(\Gh)\to U_\zeta^L(\Gg))$.
For $J\subset I$ we also define a subalgebra $U_\zeta^L(\Gp_J^-)$ 
of 
$U^L_\zeta(\Gg)$ by 
\[
U_\zeta^L(\Gp_J^-)=
\langle U_\zeta^L(\Gh), e_j^{(r)}, f_i^{(r)}\mid
j\in J, i\in I, r\geqq0\rangle.
\]
We write $U^L_\zeta(\Gb^-)=U^L_\zeta(\Gp_\emptyset^-)$.
For $\lambda\in\Lambda$ the restriction $\chi_\lambda|_{U_\BA^L(\Gh)}:U_\BA^L(\Gh)\to\BA$ induces a character
\begin{equation}
\label{eq:chi2}
\chi_\lambda:U^L_\zeta(\Gh)\to\BC
\end{equation}
of $U^L_\zeta(\Gh)$.
For $\lambda\in\Lambda^J$ the character  
$\chi_\lambda:U^L_\zeta(\Gh)\to\BC$ is extended to
the character
\begin{equation}
\chi_{J,\lambda}:U^L_\zeta(\Gp^-_J)\to \BC
\end{equation}
of $U^L_\zeta(\Gp^-_J)$ 
by $\chi_{J,\lambda}(e_j^{(r)})=\chi_{J,\lambda}(f_i^{(r)})=0$ for
$j\in J$, $i\in I$, $r\in\BZ_{>0}$.
By abuse of notation we write
\begin{equation}
\label{eq:chi3}
\chi_\lambda:=\chi_{\emptyset,\lambda}:U^L_\zeta(\Gb^-)\to\BC
\end{equation}
for $\lambda\in\Lambda$.

We say that 
a left 
$U^L_\zeta(\Gp_J^-)$-module
$M$ is integrable if it is a direct sum of weight spaces
\[
M_\mu=\{m\in M
\mid
am=\chi_\mu(a)m
\;\;
\text{for}\;\; a\in U^L_\zeta(\Gh)\}
\quad(\mu\in\Lambda),
\]
and if for any $m\in M$ 
there exists $N>0$ such that
$
e_j^{(n)}m=f_i^{(n)}m=0
$
for $i\in I$, $j\in J$, $n>N$.
The notion of an integrable right $U^L_\zeta(\Gp_J^-)$-module is defined similarly.
We denote by $\Mod_\inte(U^L_\zeta(\Gp_J^-))$ 
(resp.\ $\Mod^r_\inte(U^L_\zeta(\Gp_J^-))$)
the category of left 
(resp.\ right) integrable
$U^L_\zeta(\Gp_J^-)$-modules.
Its full subcategory consisting of finite-dimensional objects is denoted by $\modu_\inte(U^L_\zeta(\Gp_J^-))$
(resp.\
$\modu^r_\inte(U^L_\zeta(\Gp_J^-))$).
Set $
U^L_\zeta(\Gp_J^-)^*=
\Hom_\BC(U^L_\zeta(\Gp_J^-),\BC)
$.
It is a $U^L_\zeta(\Gp_J^-)$-bimodule similarly to \eqref{eq:bim}.
We define 
$\CO_\zeta(P_J^-)$
to be the subspace of 
$U^L_\zeta(\Gp_J^-)^*$ 
consisting of $\varphi\in U^L_\zeta(\Gp_J^-)^*$ satisfying the following equivalent conditions:
\begin{itemize}
\item[(1)]
the left $U^L_\zeta(\Gp_J^-)$-module $U^L_\zeta(\Gp_J^-)\varphi$ is integrable,
\item[(2)]
the right $U^L_\zeta(\Gp_J^-)$-module $\varphi U^L_\zeta(\Gp_J^-)$ is integrable.
\end{itemize}
Then $\CO_\zeta(P_J^-)$  is naturally a Hopf algebra and a $U_\zeta^L(\Gp_J^-)$-bimodule.
We write $\CO_\zeta(G)=\CO_\zeta(P_I^-)$, $\CO_\zeta(B^-)=\CO_\zeta(P_\emptyset^-)$.
It is known that 
$\CO_\zeta(G)\cong\BC\otimes_\BA\CO_\BA(G)$ with respect to $q^{1/|\Lambda/Q|}\mapsto\zeta^{1/\Lambda/Q|}$
(see e.g.\ \cite[Proposition 3.2]{T3}).

For $\lambda\in\Lambda^+$ we set 
$\Delta_\zeta(\lambda)
=\BC\otimes_\BA\Delta_\BA(\lambda)$.
It is a finite-dimensional integrable $U^L_\zeta(\Gg)$-module, called the Weyl module.

\subsection{Induction functor}
Let 
$
J_1\subset J_2\subset I
$.
We have the obvious restriction functor
\begin{align}
\label{eq:Res}
\Res^{J_1J_2}:
\Mod_\inte(U^L_\zeta(\Gp_{J_2}^-))
\to
\Mod_\inte(U^L_\zeta(\Gp_{J_1}^-)).
\end{align}
The induction functor
\begin{align}
\label{eq:Ind}
\Ind^{J_2J_1}:
\Mod_\inte(U^L_\zeta(\Gp_{J_1}^-))
\to
\Mod_\inte(U^L_\zeta(\Gp_{J_2}^-)),
\end{align}
which is right adjoint to \eqref{eq:Res} is given as follows.
For $M\in \Mod_\inte(U^L_\zeta(\Gp_{J_1}^-))$ we
define a left $U^L_\zeta(\Gp_{J_1}^-)$-module structure of 
$\CO_\zeta(P_{J_2}^-)\otimes M$ by 
\[
y\star
(\varphi\otimes m)=
\sum_{(y)}\varphi\cdot (S^{-1}y_{(0)})\otimes y_{(1)}m
\qquad
(y\in U^L_\zeta(\Gp_{J_1}^-),
\;
\varphi\in \CO_\zeta(P_{J_2}^-),
\;
m\in M),
\]
and set
\[
\Ind^{J_2J_1}M=
\{z\in \CO_\zeta(P_{J_2}^-)\otimes M
\mid
y\star z=\varepsilon(y)z
\quad(y\in U^L_\zeta(\Gp_{J_1}^-))\}.
\]
Then the left $U^L_\zeta(\Gp_{J_2}^-)$-module structure of 
$\CO_\zeta(P_{J_2}^-)\otimes M$  given by
\[
u(\varphi\otimes m)=u\varphi\otimes m
\qquad
(u\in U^L_\zeta(\Gp_{J_2}^-),
\;
\varphi\in \CO_\zeta(P_{J_2}^-),
\;
m\in M)
\]
induces a left integrable $U^L_\zeta(\Gp_{J_2}^-)$-module structure of
$\Ind^{J_2J_1}M$.

For a left 
(resp.\ right) $U^L_\zeta(\Gp_J^-)$-module $M$ we denote by  $M^{[r]}$ 
(resp.\ $M^{[l]}$) the right 
(resp.\ left) $U^L_\zeta(\Gp_J^-)$-module with the same underlying vector space as $M$ and the right 
(resp.\ left) action of $U^L_\zeta(\Gp_J^-)$ given by
\[
my=(Sy)\cdot m
\quad(\text{resp.}\;
ym=m\cdot(S^{-1}y))
\qquad(m\in M,\; y\in U^L_\zeta(\Gp_J^-)).
\]
We have $M^{[r][l]}=M$ and $M^{[l][r]}=M$.

We have also the right module versions 
\begin{align}
\label{eq:rRes}
\Res^{rJ_1J_2}:
\Mod^r_\inte(U^L_\zeta(\Gp_{J_2}^-))
\to
\Mod^r_\inte(U^L_\zeta(\Gp_{J_1}^-)),
\\
\label{eq:rInd}
\Ind^{rJ_2J_1}:
\Mod^r_\inte(U^L_\zeta(\Gp_{J_1}^-))
\to
\Mod^r_\inte(U^L_\zeta(\Gp_{J_2}^-))
\end{align}
of \eqref{eq:Res} and \eqref{eq:Ind} respectively.
Here, $\Res^{rJ_1J_2}$ is the obvious restriction functor, and
$\Ind^{rJ_2J_1}$ is given by 
$\Ind^{rJ_2J_1}(M)=(\Ind^{J_2J_1}(M^{[l]}))^{[r]}$.

\subsection{$R$-matrix}
Define subalgebras $U_\BF(\Gb^+)$ and $U_\BF(\Gb^-)$ 
of $U_\BF(\Gg)$ by
\[
U_\BF(\Gb^+)=\langle U_\BF(\Gh), \; e_i\mid i\in I\rangle,
\qquad
U_\BF(\Gb^-)=\langle U_\BF(\Gh), \; f_i\mid i\in I\rangle.
\]
For $\gamma\in Q^+$ we denote by 
$U_\BF(\Gn^+)_\gamma$ (resp.\ $U^L_\BA(\Gn^+)_\gamma$) 
the $\BF$-submodule 
(resp.\ $\BA$-submodule) of $U_\BF(\Gb^+)$ 
spanned over $\BF$ (resp.\ $\BA$) by the elements of the form
$e_{i_1}^{(m_1)}\cdots e_{i_r}^{(m_r)}$ 
with $i_1,\dots, i_r\in I$, $\sum_{s=1}^rm_s\alpha_{i_s}=\gamma$.
We similarly define 
the $\BF$-submodule 
$U_\BF(\Gn^-)_{-\gamma}$ (resp.\ $\BA$-submodule  $U^L_\BA(\Gn^-)_{-\gamma}$) of $U_\BF(\Gb^-)$ 
by replacing $e$ with $f$.
We denote by $U^L_\zeta(\Gn^\pm)_{\pm\gamma}$ the linear span of the image of 
$U^L_\BA(\Gn^\pm)_{\pm\gamma}\to U_\zeta(\Gg)$.

There exists a bilinear form
\begin{equation}
\label{eq:Dpairing}
\tau:U_\BF(\Gb^+)\times U_\BF(\Gb^-)\to \BF
\end{equation}
called the Drinfeld pairing which is characterized by the properties:
\begin{align*}
&\tau(x,y_1y_2)
=(\tau\otimes\tau)(\Delta(x),y_1\otimes y_2)
\quad&
(x\in U_\BF(\Gb^+), \; y_1, y_2\in U_\BF(\Gb^-)),
\\
&\tau(x_1x_2,y)
=(\tau\otimes\tau)(x_2\otimes x_1,\Delta(y))
\quad&
(x_1, x_2\in U_\BF(\Gb^+), \; y\in U_\BF(\Gb^-)),
\\
&\tau(k_\lambda, k_\mu)
=
q^{-(\lambda,\mu)}
\quad&
(\lambda, \mu\in\Lambda),
\\
&\tau(e_i,f_j)=
-\delta_{ij}(q_i-q_i^{-1})^{-1}
\quad&
(i, j\in I),
\\
&\tau(e_i,k_\lambda)=
\tau(k_\lambda,f_i)=0
\quad&
(i\in I, \;\lambda\in\Lambda).
\end{align*}
It is known that  the restriction 
$\tau|_{U_\BF(\Gn^+)_{\gamma}\times U_\BF(\Gn^-)_{-\gamma}}$ is non-degenerate for any $\gamma\in Q^+$.
We denote by $\CR_\gamma\in
U_\BF(\Gn^+)_{\gamma}\otimes_\BF U_\BF(\Gn^-)_{-\gamma}$ the corresponding canonical element.
Then we have 
\[
\CR_\gamma\in
U^L_\BA(\Gn^+)_{\gamma}\otimes_\BA U^L_\BA(\Gn^-)_{-\gamma}
\subset
U_\BF(\Gn^+)_{\gamma}\otimes_\BF U_\BF(\Gn^-)_{-\gamma}.
\]
We denote by 
$\CR_{\zeta,\gamma}\in
U^L_\zeta(\Gn^+)_{\gamma}\otimes_\BC U^L_\zeta(\Gn^-)_{-\gamma}$ 
the specialization of $\CR_\gamma$.

Let $M, N\in\Mod_\inte(U_\zeta^L(\Gg))$.
We regard $M\otimes N$ as an object of $\Mod_\inte(U_\zeta^L(\Gg))$ via the comultiplication 
$\Delta:U_\zeta^L(\Gg)\to U_\zeta^L(\Gg)\otimes U_\zeta^L(\Gg)$.
We define a linear map
\begin{equation}
\label{eq:R}
R^{M,N}:M\otimes N\to N\otimes M
\end{equation}
by
\[
R^{M,N}(m\otimes n)
=\zeta^{-(\lambda,\mu)}\sum_{\gamma\in Q^+}
P(\CR_{\zeta,\gamma}(m\otimes n))
\qquad(m\in M_\lambda,\; n\in N_\mu),
\]
where $P:M\otimes N\to N\otimes M$ is the transposition.

The following fact is standard.
\begin{proposition}
\label{prop:R}
For $M, N\in\Mod_\inte(U_\zeta^L(\Gg))$ the linear map \eqref{eq:R} gives an isomorphism in $\Mod_\inte(U_\zeta^L(\Gg))$.
\end{proposition}

\section{{$\CO$-modules on the quantized flag manifolds}}
\subsection{Quantized partial flag manifold}
Define a subalgebra $A_\BF$ of 
$\CO_\BF(G)$ by
\[
A_{\BF}=
\{\varphi\in\CO_\BF(G)
\mid
\varphi f_i=0\;\;(i\in I)
\}.
\]
The $U_\BF(\Gg)$-bimodule structure of $\CO_\BF(G)$ given by \eqref{eq:bim} induces a 
$(U_\BF(\Gg),U_\BF(\Gh))$-bimodule structure of $A_\BF$.
Setting
\[
A_\BF(\lambda)=
\{\varphi\in A_\BF
\mid
\varphi a=\chi_\lambda(a)\varphi\;
(a\in U_\BF(\Gh))\}
\qquad
(\lambda\in\Lambda^+),
\]
we have
\[
A_\BF=
\bigoplus_{\lambda\in\Lambda^+}A_\BF(\lambda).
\]

Set $A_\BA=\CO_\BA(G)\cap A_\BF$.
Then 
$A_\BA$ is an $\BA$-subalgebra of $\CO_\BA(G)$ satisfying 
\[
A_{\BA}=
\bigoplus_{\lambda\in \Lambda^{+}}
A_{\BA}(\lambda),
\qquad
A_{\BA}(\lambda)=
A_\BF(\lambda)\cap A_\BA.
\]
Moreover,
it is a $(U^L_\BA(\Gg),U^L_\BA(\Gh))$-bimodule.

Setting
\[
A_\zeta=\BC\otimes_\BA A_\BA,
\qquad
A_\zeta(\lambda)=\BC\otimes_\BA A_\BA(\lambda)
\quad(\lambda\in\Lambda^+)
\]
with respect to $\BA\to\BC$ ($q^{1/|\Lambda/Q|}\mapsto\zeta^{1/|\Lambda/Q|}$)
we have 
\[
A_{\zeta}=
\bigoplus_{\lambda\in \Lambda^{+}}
A_{\zeta}(\lambda).
\]
Not that $A_\zeta$ is a $\Lambda$-graded $\BC$-algebra as well as a 
$(U^L_\zeta(\Gg),U^L_\zeta(\Gh))$-bimodule 
satisfying 
\[
A_\zeta(\lambda)=
\{\varphi\in A_\zeta
\mid
\varphi a=\chi_\lambda(a)\varphi\;
(a\in U_\zeta^L(\Gh))\}
\qquad
(\lambda\in\Lambda^+).
\]

For $J\subset I$ we define a $\BC$-subalgebra $A_{J,\zeta}$ of $A_\zeta$ by
\[
A_{J,\zeta}=
\bigoplus_{\lambda\in \Lambda^{J+}}
A_{\zeta}(\lambda).
\]
It is a $\Lambda^J$-graded $\BC$-algebra.
We denote by $\Tor_{\Lambda^{J+}}(A_{J,\zeta})$ the full subcategory of $\Mod_{\Lambda^J}(A_{J,\zeta})$
consisting of $M\in \Mod_{\Lambda^J}(A_{J,\zeta})$
satisfying
\[
\forall m\in M,\;\;\exists N\;\;\text{s.t.}\;\;
\lambda\in\Lambda^{J+}, \;\;
\langle\lambda,\alpha_i^\vee\rangle>N\;\;(\forall i\in I\setminus J)
\;\;
\Longrightarrow
\;\;
A_{\zeta}(\lambda)m=0,
\]
and set
\begin{equation}
\Mod(\CO_{\CP_{J,\zeta}})
=\Mod_{\Lambda^J}(A_{J,\zeta})/\Tor_{\Lambda^{J+}}(A_{J,\zeta}).
\end{equation}
By definition $\Mod(\CO_{\CP_{J,\zeta}})$ is the localization of the category $\Mod_{\Lambda^J}(A_{J,\zeta})$ with respect to the multiplicative system 
consisting of morphisms in $\Mod_{\Lambda^J}(A_{J,\zeta})$ whose kernel and cokernel belong to $\Tor_{\Lambda^{J+}}(A_{J,\zeta})$.
\begin{remark}
The notations $\CP_{J,\zeta}$ and $\CO_{\CP_{J,\zeta}}$ do not have actual meanings; however, $\Mod(\CO_{\CP_{J,\zeta}})$ is an actual abelian category. 
\end{remark}

By a general result the  natural exact functor 
\begin{equation}
\label{eq:O1}
\omega_{J}^*:\Mod_{\Lambda^J}(A_{J,\zeta})\to\Mod(\CO_{\CP_{J,\zeta}})
\end{equation}
admits a right adjoint
\begin{equation}
\label{eq:O2}
\omega_{J*}:\Mod(\CO_{\CP_{J,\zeta}})\to
\Mod_{\Lambda^J}(A_{J,\zeta}),
\end{equation}
which is left exact, and we have
$
\omega_{J}^*\circ\omega_{J*}
=\Id
$ (see \cite[Ch4]{P}).
We define a left exact functor 
\begin{equation}
\label{eq:gammaO}
\Gamma(\CP_{J,\zeta},\bullet):\Mod(\CO_{\CP_{J,\zeta}})\to
\Mod(\BC)
\end{equation}
by 
$\Gamma(\CP_{J,\zeta},N)=(\omega_{J*}N)(0)$.

We define a full subcategory 
$\modu(\CO_{\CP_{J,\zeta}})$
of 
$\Mod(\CO_{\CP_{J,\zeta}})$ by
\begin{equation}
\modu(\CO_{\CP_{J,\zeta}})
=\modu_{\Lambda^J}(A_{J,\zeta})/
(\Tor_{\Lambda^{J+}}(A_{J,\zeta})
\cap
\modu_{\Lambda^J}(A_{J,\zeta}))
.
\end{equation}
The functors \eqref{eq:O1}, \eqref{eq:O2} induce
\begin{equation}
\omega_{J}^*:\modu_{\Lambda^J}(A_{J,\zeta})\to\modu(\CO_{\CP_{J,\zeta}}),
\end{equation}
\begin{equation}
\omega_{J*}:\modu(\CO_{\CP_{J,\zeta}})\to
\modu_{\Lambda^J}(A_{J,\zeta}).
\end{equation}

For $\mu\in\Lambda^J$ we define and exact functor
\begin{equation}
\label{eq:shift2}
(\bullet)[\mu]:
{\Mod}(\CO_{\CP_{J,\zeta}})
\to
{\Mod}(\CO_{\CP_{J,\zeta}})
\qquad
(M\mapsto M[\mu])
\end{equation}
to be the functor induced by 
\[
(\bullet)[\mu]:
\Mod_{\Lambda^J}(A_{J,\zeta})
\to
\Mod_{\Lambda^J}(A_{J,\zeta})
\qquad
(M\mapsto M[\mu]),
\]
where
\[
(M[\mu])(\lambda)
=M(\lambda+\mu)
\qquad(\lambda\in\Lambda^J).
\]
\begin{remark}
\label{rem:P1}
In the case 
$\zeta=1$ the category $\Mod(\CO_{\CP_{J,1}})$
(resp.\ $\modu(\CO_{\CP_{J,1}})$)  
is naturally identified with the category $\Mod(\CO_{\CP_{J}})$ 
(resp.\ $\modu(\CO_{\CP_{J}})$)  
consisting of quasi-coherent 
(resp.\ coherent)
$\CO_{\CP_J}$-modules.
\end{remark}

We write 
\[
\Mod(\CO_{\CB_{\zeta}})=
\Mod(\CO_{\CP_{\emptyset,\zeta}}),
\qquad
\modu(\CO_{\CB_{\zeta}})=
\modu(\CO_{\CP_{\emptyset,\zeta}}),
\]
and
\begin{equation}
\label{eq:OB1}
\omega^*=\omega^*_{\emptyset}:\Mod_{\Lambda}(A_{\zeta})
\to\Mod(\CO_{\CB_{\zeta}}),
\qquad
\omega^*=\omega^*_{\emptyset}:\modu_{\Lambda}(A_{\zeta})\to\modu(\CO_{\CB_{\zeta}}),
\end{equation}
\begin{equation}
\label{eq:OB2}
\omega_{*}=\omega_{\emptyset*}:\Mod(\CO_{\CB_{\zeta}})\to
\Mod_{\Lambda}(A_{\zeta}), 
\quad
\omega_{*}=\omega_{\emptyset*}:\modu(\CO_{\CB_{\zeta}})\to
\modu_{\Lambda}(A_{\zeta}),
\end{equation}
\begin{equation}
\label{eq:Gamma}
\Gamma=\Gamma(\CP_{\emptyset,\zeta},\bullet):\Mod(\CO_{\CB_{\zeta}})\to
\Mod(\BC).
\end{equation}

\subsection{Another description of quantized partial flag manifold}

Following \cite{BK1} we give another description of the quantized flag manifolds.
\begin{definition}
We define an abelian category 
$\widetilde{\Mod}(\CO_{\CP_{J,\zeta}})$ 
as follows.
An object of $\widetilde{\Mod}(\CO_{\CP_{J,\zeta}})$
is a vector space $M$ equipped with 
a left $\CO_\zeta(G)$-module structure
and a right integrable $U^L_\zeta(\Gp_J^-)$-module structure satisfying 
\[
(fm)y=\sum_{(y)}(fy_{(0)})(my_{(1)})
\qquad(f\in \CO_\zeta(G), \; m\in M, \; y\in U^L_\zeta(\Gp_J^-)).
\]
A morphism in 
$\widetilde{\Mod}(\CO_{\CP_{J,\zeta}})$ is a linear map 
preserving the left $\CO_\zeta(G)$-module structure and the 
right $U^L_\zeta(\Gp_J^-)$-module structure.
\end{definition}

\begin{remark}
\label{rem:CC}
In the case $\zeta=1$ the category
$\widetilde{\Mod}(\CO_{\CP_{J,1}})$ is naturally identified with the category $\widetilde{\Mod}(\CO_{\CP_{J}})$ consisting of quasi-coherent $\CO_G$-modules equivariant with respect to the action of $P_J^-$ on $G$ given by the left multiplication.
Hence it is also equivalent to the category $\Mod(\CO_{\CP_J})$ of quasi-coherent $\CO_{\CP_J}$-modules.
\end{remark}

We define a functor
\begin{equation}
\label{eq:tomega}
\widetilde{\omega}_{J*}:
\widetilde{\Mod}(\CO_{\CP_{J,\zeta}})\to\Mod_{\Lambda^J}(A_{J,\zeta})
\end{equation}
by
\[
\widetilde{\omega}_{J*}M
=
\sum_{\lambda\in\Lambda^J}
(\widetilde{\omega}_{J*}M)(\lambda)
=
\bigoplus_{\lambda\in\Lambda^J}
(\widetilde{\omega}_{J*}M)(\lambda)
\subset M,
\]
\[
(\widetilde{\omega}_{J*}M)(\lambda)
=\{m\in M\mid 
my=\chi_{J,\lambda}(y)m\;\;(y\in U_\zeta^L(\Gp_J^-))\}
\qquad(\lambda\in\Lambda^J)
\]
for $M\in\widetilde{\Mod}(\CO_{\CP_{J,\zeta}})$.

For $\mu\in\Lambda^J$ we denote by $\BC_\mu$ the one-dimensional 
right $U^L_\zeta(\Gp_J^-)$-module  corresponding to $\chi_{J,\mu}:U^L_\zeta(\Gp_J^-)\to\BC$, and define an exact functor 
\begin{equation}
\label{eq:tshift}
(\bullet)[\mu]:
\widetilde{\Mod}(\CO_{\CP_{J,\zeta}})\to
\widetilde{\Mod}(\CO_{\CP_{J,\zeta}})
\end{equation}
by 
$
M[\mu]=M\otimes\BC_{-\mu}
$.
Here, the left $\CO_\zeta(G)$-module structure of $M[\mu]$ is given by that of the first factor $M$, and 
the right $U^L_\zeta(\Gp_J^-)$-module structure  is given by the tensor product of 
the right $U^L_\zeta(\Gp_J^-)$-modules $M$ and $\BC_{-\mu}$.

We define a functor 
\begin{equation}
\label{eq:to}
\widetilde{\omega}_J^*:\Mod_{\Lambda^J}(A_{J,\zeta})\to
\widetilde{\Mod}(\CO_{\CP_{J,\zeta}})
\end{equation}
by
$
\widetilde{\omega}_J^*K=\CO_\zeta(G)\otimes_{A_{J,\zeta}}K$.
The left $\CO_\zeta(G)$-module structure of $\CO_\zeta(G)\otimes_{A_{J,\zeta}}K$ is given by the left multiplication on the first factor.
The right $U^L_\zeta(\Gp_J^-)$-module structure
of $\CO_\zeta(G)\otimes_{A_{J,\zeta}}K$
is given by the tensor product of two 
right $U^L_\zeta(\Gp_J^-)$-modules, 
where
$K$ is regarded as a right $U^L_\zeta(\Gp_J^-)$-module by
\[
ky=\chi_{J,\lambda}(y)k
\qquad(k\in K(\lambda),\; y\in U^L_\zeta(\Gp_J^-)).
\]
We write
\begin{gather*}
\widetilde{\Mod}(\CO_{\CB_{\zeta}})=
\widetilde{\Mod}(\CO_{\CP_{\emptyset,\zeta}}),
\\
\widetilde{\omega}_*=\widetilde{\omega}_{\emptyset*}:
\widetilde{\Mod}(\CO_{\CB_{\zeta}})
\to\Mod_{\Lambda}(A_{\zeta}),
\qquad
\widetilde{\omega}^*=\widetilde{\omega}_\emptyset^*:
\Mod_{\Lambda}(A_{\zeta})\to
\widetilde{\Mod}(\CO_{\CB_{\zeta}}).
\end{gather*}

The following result for $J=\emptyset$ is due to Backelin and Kremnizer \cite{BK1}.
As in \cite{BK1} it is a consequence of  the general theory due to Artin and Zhang \cite{AZ}.
Here, we only give a sketch of the proof.
\begin{proposition}
\label{prop:1equiv2}
The functor $\omega_J^*\circ\widetilde{\omega}_{J*}$ gives an equivalence 
\begin{equation}
\label{eq:1equiv2}
\widetilde{\Mod}(\CO_{\CP_{J,\zeta}})\cong
\Mod(\CO_{\CP_{J,\zeta}})
\end{equation}
of abelian categories compatible with the shift functors \eqref{eq:shift2}, \eqref{eq:tshift}.
Moreover, identifying 
$\widetilde{\Mod}(\CO_{\CP_{J,\zeta}})$
with 
$\Mod(\CO_{\CP_{J,\zeta}})$
via \eqref{eq:1equiv2}
we have
\begin{equation}
\label{eq:1equiv-a2}
\widetilde{\omega}_J^*
=\omega_J^*,
\qquad
\widetilde{\omega}_{J*}
=\omega_{J*}.
\end{equation}
\end{proposition}
\begin{proof}
By an obvious generalization of \cite[Theorem 4.5]{AZ} we can show that
$\omega_J^*\circ\widetilde{\omega}_{J*}$
gives an equivalence which is compatible with the shift functors.
Moreover,  we have
\begin{equation}
\label{eq:1equiv-b2}
\omega_{J*}\circ
\omega_J^*\circ\widetilde{\omega}_{J*}
=
\widetilde{\omega}_{J*}.
\end{equation}
Here, in order to show that the assumption of 
\cite[Theorem 4.5]{AZ} is satisfied, we need some properties of the derived induction functors for quantized enveloping algebras.
More specifically, we use the Kempf type vanishing theorem \cite[Corollary 5.7]{APW} and the vanishing of cohomologies in sufficiently large degrees
\cite[Theorem 5.8]{APW}.
Then we obtain $\widetilde{\omega}_{J*}
=\omega_{J*}$ by \eqref{eq:1equiv-b2}.
The remaining 
$\widetilde{\omega}_J^*
=\omega_J^*$ 
follows from the fact that 
$\widetilde{\omega}_J^*$, 
$\omega_J^*$
are the left adjoint functors to
$\widetilde{\omega}_{J*}$, 
$\omega_{J*}$ respectively.
\end{proof}
\begin{corollary}
\label{cor:oo}
We have
\[
\omega_{J*}\omega_J^*A_{J,\zeta}
\cong
A_{J,\zeta}.
\]
\end{corollary}

\begin{proof}
By Proposition \ref{prop:1equiv2} it is sufficient to show
$\widetilde{\omega}_{J*}\widetilde{\omega}_J^*A_{J,\zeta}
\cong
A_{J,\zeta}$.
This is obvious from definition.
\end{proof}

\subsection{Equivariant $\CO$-modules}
Similarly to \cite[Section 4]{T2} we define  
$\Mod_{\Lambda^J}^{\eq}(A_{J,\zeta})$ to be the category consisting of 
$N\in\Mod_{\Lambda^J}(A_{J,\zeta})$ 
equipped with a left integrable $U^L_\zeta(\Gg)$-module structure satisfying 
\[
U^L_\zeta(\Gg) N(\lambda)\subset N(\lambda)
\qquad(\lambda\in\Lambda^J),
\]
\[
u(\varphi n)=\sum_{(u)}(u_{(0)}\varphi)\cdot(u_{(1)}n)
\qquad
(u\in U^L_\zeta(\Gg),\;
\varphi\in A_{J,\zeta},\;
n\in N), 
\]
and set
\begin{equation}
\Mod^{\eq}(\CO_{\CP_{J,\zeta}})
=
\Mod_{\Lambda^J}^{\eq}(A_{J,\zeta})/
\Mod_{\Lambda^J}^{\eq}(A_{J,\zeta})\cap
\Tor_{\Lambda^{J+}}(A_{J,\zeta}).
\end{equation}

\begin{remark}
In the case $\zeta=1$ the category
$\Mod^{\eq}(\CO_{\CP_{J,1}})$ is naturally identified with the category consisting of $G$-equivariant quasi-coherent 
$\CO_{\CP_{J}}$-modules.
\end{remark}
Similarly to \cite[Section 4]{T2} we define  
$\widetilde{\Mod}^{\eq}(\CO_{\CP_{J,\zeta}})$ to be the category consisting of 
 an object $M$ of 
$\widetilde{\Mod}(\CO_{\CP_{J,\zeta}})$ 
equipped with a left integrable $U^L_\zeta(\Gg)$-module structure satisfying
\[
(um)y=u(my)
\qquad(m\in M,\; u\in U^L_\zeta(\Gg),\; y\in U^L_\zeta(\Gp_J^-)),
\]
\[
u(\varphi m)=\sum_{(u)}(u_{(0)}\varphi)\cdot(u_{(1)}m)
\qquad (u\in U^L_\zeta(\Gg),\; \varphi\in\CO_\zeta(G),\; m\in M).
\]
\begin{remark}
In the case $\zeta=1$ the category
$\widetilde{\Mod}^{\eq}(\CO_{\CP_{J,1}})$ is naturally identified with the category consisting of quasi-coherent 
$\CO_{G}$-modules equivariant under the left action of $G$ via the left multiplication and the right action of $P_J^-$ via the right multiplication.
\end{remark}

Similarly to \cite[Section 4]{T2} we have the following equivalences of abelian categories:
\begin{equation}
\label{eq:equivariant}
\Mod_\inte(U^L_\zeta(\Gp_J^-))
\cong
\widetilde{\Mod}^{\eq}(\CO_{\CP_{J,\zeta}})
\cong
{\Mod}^{\eq}(\CO_{\CP_{J,\zeta}}).
\end{equation}
Here, the first equivalence  is given by associating 
$M\in\widetilde{\Mod}^{\eq}(\CO_{\CP_{J,\zeta}})$
to $(M^{U^L_\zeta(\Gg)})^{[l]}$,
where $M^{U^L_\zeta(\Gg)}$ denotes the right $U^L_\zeta(\Gp_J^-)$-submodule
\[
\{m\in M\mid um=\varepsilon(u)m\;\;(u\in U^L_\zeta(\Gg))\}
\]
of $M$.
The opposite correspondence is given by associating $K\in\Mod_\inte(U^L_\zeta(\Gp_J^-))$ to 
$\CO_\zeta(G)\otimes K\in
\widetilde{\Mod}^{\eq}(\CO_{\CP_{J,\zeta}})$.
Here, the left $\CO_\zeta(G)$-module structure and the left $U^L_\zeta(\Gg)$-module structure of $\CO_\zeta(G)\otimes K$ are induced by those on the first factor $\CO_\zeta(G)$, and 
the right $U^L_\zeta(\Gp_J^-)$-module structure is given by
\[
(\varphi\otimes k)y=\sum_{(y)}\varphi y_{(0)}\otimes (Sy_{(1)})\cdot k
\qquad
(\varphi\in \CO_\zeta(G),\;
k\in K, \;
y\in U^L_\zeta(\Gp_J^-)).
\]
The second equivalence in \eqref{eq:equivariant} is induced by
\eqref{eq:1equiv2}.
We write
\begin{equation}
\label{eq:equivariant2}
\CF_J:
\Mod_\inte(U^L_\zeta(\Gp_J^-))
\to
{\Mod}^{\eq}(\CO_{\CP_{J,\zeta}})
\end{equation}
for the composite of \eqref{eq:equivariant}.

\subsection{Inverse and direct images}
For 
$
J_1\subset J_2\subset I
$ 
we define functors 
\begin{align*}
({\pi}^{J_2J_1}_\zeta)^*:
{\Mod}(\CO_{\CP_{J_2,\zeta}})
&\to
{\Mod}(\CO_{\CP_{J_1,\zeta}}),
\\
({\pi}^{J_2J_1}_\zeta)_*:
{\Mod}(\CO_{\CP_{J_1,\zeta}})
&\to
{\Mod}(\CO_{\CP_{J_2,\zeta}})
\\
(\widetilde{\pi}^{J_2J_1}_\zeta)^*:
\widetilde{\Mod}(\CO_{\CP_{J_2,\zeta}})
&\to
\widetilde{\Mod}(\CO_{\CP_{J_1,\zeta}}),
\\
(\widetilde{\pi}^{J_2J_1}_\zeta)_*:
\widetilde{\Mod}(\CO_{\CP_{J_1,\zeta}})
&\to
\widetilde{\Mod}(\CO_{\CP_{J_2,\zeta}})
\end{align*}
in the following.

The definition of 
\begin{equation}
\label{eq:invt}
(\widetilde{\pi}^{J_2J_1}_\zeta)^*:
\widetilde{\Mod}(\CO_{\CP_{J_2,\zeta}})
\to
\widetilde{\Mod}(\CO_{\CP_{J_1,\zeta}})
\end{equation}
is simple.
For 
$M\in \widetilde{\Mod}(\CO_{\CP_{J_2,\zeta}})$
we have
$(\widetilde{\pi}^{J_2J_1}_\zeta)^*M=M$
as a left 
$\CO_\zeta(G)$-module and 
$(\widetilde{\pi}^{J_2J_1}_\zeta)^*M=
\Res^{rJ_1J_2}M$
as a right
$U^L_\zeta(\Gp_{J_1}^-)$-module.
This is obviously an exact functor.
By definition we have the following commutative diagram:
\begin{equation}
\label{eq:inv-diag}
\xymatrix@C=70pt{
\Mod_{\Lambda^{J_2}}(A_{J_2,\zeta})
\ar[r]^{A_{J_1,\zeta}\otimes_{A_{J_2,\zeta}}(\bullet)}
\ar[d]_{\tomega_{J_2}^*}
&
{\Mod}_{\Lambda^{J_1}}(A_{J_1,\zeta})
\ar[d]^{\tomega_{J_1}^*}
\\
\widetilde{\Mod}(\CO_{\CP_{J_2,\zeta}})
\ar[r]_{(\widetilde{\pi}^{J_2J_1}_\zeta)^*}
&
\widetilde{\Mod}(\CO_{\CP_{J_1,\zeta}}).
}
\end{equation}

We next give the definition of the functor 
\begin{equation}
\label{eq:dirt}
(\widetilde{\pi}^{J_2J_1}_\zeta)_*:
\widetilde{\Mod}(\CO_{\CP_{J_1,\zeta}})
\to
\widetilde{\Mod}(\CO_{\CP_{J_2,\zeta}}).
\end{equation}
Let 
$M\in \widetilde{\Mod}(\CO_{\CP_{J_1,\zeta}})$.
As a right
$U^L_\zeta(\Gp_{J_2}^-)$-module 
we have
$(\widetilde{\pi}^{J_2J_1}_\zeta)_*M=\Ind^{rJ_2J_1}M$.
The left $\CO_\zeta(G)$-module structure of 
$\Ind^{rJ_2J_1}M$ is given as follows.
Let $
\alpha:\CO_\zeta(G)\otimes M\to M
$ be the left $\CO_\zeta(G)$-module structure of $M$.
Since $\alpha$ is a homomorphism of right
$U^L_\zeta(\Gp_{J_1}^-)$-modules 
we have a homomorphism 
\[
\Ind^{^rJ_2J_1}\alpha:
\Ind^{^rJ_2J_1}(\CO_\zeta(G)\otimes M)\to \Ind^{^rJ_2J_1}M
\]
of right
$U^L_\zeta(\Gp_{J_2}^-)$-modules.
Note that 
$\CO_\zeta(G)$ is 
an integrable right $U^L_\zeta(\Gp_{J_2}^-)$-module.
Hence by the standard property of the induction functor called the tensor identity
we have
$
\Ind^{rJ_2J_1}(\CO_\zeta(G)\otimes M)
\cong
\CO_\zeta(G)\otimes \Ind^{rJ_2J_1}M
$.
Hence we obtain a homomorphism 
\[
\tilde{\alpha}:
\CO_\zeta(G)\otimes \Ind^{rJ_2J_1}M\to \Ind^{rJ_2J_1}M
\]
of right 
$U^L_\zeta(\Gp_{J_2}^-)$-modules.
This gives the desired left $\CO_\zeta(G)$-module structure of 
$\Ind^{rJ_2J_1}M$.

Define a functor
\[
F^{J_2J_1}_\zeta:
\Mod_{\Lambda^{J_1}}(A_{J_1,\zeta})
\to
{\Mod}_{\Lambda^{J_2}}(A_{J_2,\zeta})
\]
by
\[
F^{J_2J_1}_\zeta(M)
=
\bigoplus_{\lambda\in\Lambda^{J_2}}
M(\lambda).
\]
\begin{lemma}
\label{lem:dir}
The following diagram is commutative:
%
\begin{equation}
\label{eq:dir-diag}
\xymatrix@C=50pt@R=30pt{
\Mod_{\Lambda^{J_1}}(A_{J_1,\zeta})
\ar[r]^{F^{J_2J_1}_\zeta}
&
{\Mod}_{\Lambda^{J_2}}(A_{J_2,\zeta})
\\
\widetilde{\Mod}(\CO_{\CP_{J_1,\zeta}})
\ar[r]_{(\widetilde{\pi}^{J_2J_1}_\zeta)_*}
\ar[u]^{\tomega_{J_1*}}
&
\widetilde{\Mod}(\CO_{\CP_{J_2,\zeta}})
\ar[u]_{\tomega_{J_2*}}.
}
\end{equation}
\end{lemma}
\begin{proof}
For 
$M\in \widetilde{\Mod}(\CO_{\CP_{J_1,\zeta}})$, 
$\lambda\in\Lambda^{J_2}$ we have
\begin{align*}
(\tomega_{J_2*}(\widetilde{\pi}^{J_2J_1}_\zeta)_*M)(\lambda)
\cong&
\Hom_{\Mod^r_\inte(U^L_\zeta(\Gp_{J_2}^-))}
(\BC_\lambda,(\widetilde{\pi}^{J_2J_1}_\zeta)_*M)
\\
\cong&
\Hom_{\Mod^r_\inte(U^L_\zeta(\Gp_{J_2}^-))}
(\BC_\lambda,\Ind^{J_2J_1}(M))
\\
\cong&
\Hom_{\Mod^r_\inte(U^L_\zeta(\Gp_{J_1}^-))}
(\BC_\lambda,M)
\\
\cong&
(\tomega_{J_1*}M)(\lambda)
\end{align*}
by the Frobenius reciprocity.
\end{proof}

\begin{lemma}
\label{lem:inv-dir-tadjoint}
The functor $(\widetilde{\pi}^{J_2J_1}_\zeta)_*$ is right adjoint to 
$(\widetilde{\pi}^{J_2J_1}_\zeta)^*$.
\end{lemma}
\begin{proof}
For 
$M\in \widetilde{\Mod}(\CO_{\CP_{J_1,\zeta}})$, 
$N\in \widetilde{\Mod}(\CO_{\CP_{J_2,\zeta}})$
we have
\begin{align*}
\Hom((\widetilde{\pi}^{J_2J_1}_\zeta)^*N,M)
=&\Hom_{\widetilde{\Mod}(\CO_{\CP_{J_1,\zeta}})}
(N,M)
\\
=&
\Hom_{\Mod^r_\inte(U^L_\zeta(\Gp_{J_1}^-))}(N,M)
\cap
\Hom_{\Mod(\CO_\zeta(G))}(N,M),
\\
\Hom(N,(\widetilde{\pi}^{J_2J_1}_\zeta)_*M)
=&
\Hom_{\Mod^r_\inte(U^L_\zeta(\Gp_{J_2}^-))}(N,\Ind^{J_2J_1}M)
\cap
\Hom_{\Mod(\CO_\zeta(G))}(N,\Ind^{J_2J_1}M)
\\
\cong&
\Hom_{\Mod^r_\inte(U^L_\zeta(\Gp_{J_1}^-))}(N,M)
\cap
\Hom_{\Mod(\CO_\zeta(G))}(N,M).
\end{align*}
\end{proof}

We define functors 
\begin{align}
\label{eq:inv}
({\pi}^{J_2J_1}_\zeta)^*:
{\Mod}(\CO_{\CP_{J_2,\zeta}})
\to&
{\Mod}(\CO_{\CP_{J_1,\zeta}}),
\\
\label{eq:dir}
({\pi}^{J_2J_1}_\zeta)_*:
{\Mod}(\CO_{\CP_{J_1,\zeta}})
\to&
{\Mod}(\CO_{\CP_{J_2,\zeta}})
\end{align}
as the functors corresponding to \eqref{eq:invt}, 
\eqref{eq:dirt} under the category equivalence 
in Proposition \ref{prop:1equiv2}.

By \eqref{eq:inv-diag} and Proposition \ref{prop:1equiv2} we have the following.
\begin{lemma}
\label{lem:inv2}
The functor 
\[
A_{J_1,\zeta}\otimes_{A_{J_2,\zeta}}(\bullet):
\Mod_{\Lambda^{J_2}}(A_{J_2,\zeta})\to\Mod_{\Lambda^{J_1}}(A_{J_1,\zeta})
\qquad
(K\mapsto
A_{J_1,\zeta}\otimes_{A_{J_2,\zeta}}K)
\]
induces \eqref{eq:inv}.
Namely we have the following commutative diagram:
%
\[
\xymatrix@C=70pt@R=30pt{
\Mod_{\Lambda^{J_2}}(A_{J_2,\zeta})
\ar[r]^{A_{J_1,\zeta}\otimes_{A_{J_2,\zeta}}(\bullet)}
\ar[d]_{\omega_{J_2}^*}
&
{\Mod}_{\Lambda^{J_1}}(A_{J_1,\zeta})
\ar[d]^{\omega_{J_1}^*}
\\
{\Mod}(\CO_{\CP_{J_2,\zeta}})
\ar[r]_{({\pi}^{J_2J_1}_\zeta)^*}
&
{\Mod}(\CO_{\CP_{J_1,\zeta}}).
}
\]
\end{lemma}
\begin{lemma}
\label{lem:dir2}
We have
\[
(\pi^{J_2J_1}_\zeta)_*
=
\omega_{J_2}^*\circ F^{J_2J_1}_\zeta\circ\omega_{J_1*}.
\]
\end{lemma}
\begin{proof}
By \eqref{eq:dir-diag}
and 
Proposition \ref{prop:1equiv2} we have
\[
\omega_{J_2}^*\circ F^{J_2J_1}_\zeta\circ\omega_{J_1*}
=
\omega_{J_2}^*\circ \omega_{J_2*}\circ
(\pi^{J_2J_1}_\zeta)_*
=(\pi^{J_2J_1}_\zeta)_*.
\]
\end{proof}

By Lemma \ref{lem:inv-dir-tadjoint} we have the following.
\begin{lemma}
\label{lem:inv-dir-adjoint}
The functor $({\pi}^{J_2J_1}_\zeta)_*$ is right adjoint to 
$({\pi}^{J_2J_1}_\zeta)^*$.
\end{lemma}

By Lemma \ref{lem:inv2} and Lemma \ref{lem:dir2} the functors 
\eqref{eq:inv}, \eqref{eq:dir}  induce
\begin{align}
({\pi}^{J_2J_1}_\zeta)^*:
{\modu}(\CO_{\CP_{J_2,\zeta}})
&\to
{\modu}(\CO_{\CP_{J_1,\zeta}}),
\\
({\pi}^{J_2J_1}_\zeta)_*:
{\modu}(\CO_{\CP_{J_1,\zeta}})
&\to
{\modu}(\CO_{\CP_{J_2,\zeta}}).
\end{align}

It is easily also seen that the functors \eqref{eq:inv}, \eqref{eq:dir} induce
\begin{align}
\label{eq:inv-eq}
({\pi}^{J_2J_1}_\zeta)^*:
{\Mod}^{\eq}(\CO_{\CP_{J_2,\zeta}})
\to&
{\Mod}^{\eq}(\CO_{\CP_{J_1,\zeta}}),
\\
\label{eq:dir-eq}
({\pi}^{J_2J_1}_\zeta)_*:
{\Mod}^{\eq}(\CO_{\CP_{J_1,\zeta}})
\to&
{\Mod}^{\eq}(\CO_{\CP_{J_2,\zeta}}).
\end{align}
Similarly to \cite[Lemma 4.9]{T2} we have the following.
\begin{lemma}
\label{lem:IndGamma}
We have the following commutative diagram of functors:
\[
\xymatrix@C=50pt@R=30pt{
\Mod_\inte(U^L_\zeta(\Gp_{J_1}^-))
\ar[r]^{\Ind^{J_2J_1}}
\ar[d]_{\CF_{J_1}}
&
\Mod_\inte(U^L_\zeta(\Gp_{J_2}^-))
\ar[d]^{\CF_{J_2}}
\\
\Mod^{\eq}(\CO_{\CP_{J_1,\zeta}})
\ar[r]_{({\pi}^{J_2J_1}_\zeta)_*}
&
\Mod^{\eq}(\CO_{\CP_{J_2,\zeta}}).
}
\]
\end{lemma}
\section{{$D$-modules on the quantized flag manifolds}}
\subsection{$D$-modules}
\label{subsec:Dmod}
For $\varphi\in A_\BA$, $u\in U_\BA(\Gg)$, $\lambda\in\Lambda$ we define
$
\ell_\varphi, \deru_u, \sigma_\lambda
\in\End_\BA(A_\BA)
$
by
\[
\ell_\varphi(\psi)=\varphi\psi,
\quad
\deru_u(\psi)=u\psi,
\quad
\sigma_\lambda(\psi)=\psi k_\lambda
\qquad(\psi\in A_\BA).
\]
Here, $A_\BA$ is regarded as a $(U_\BA(\Gg),U_\BA(\Gh))$-bimodule via the embedding
$U_\BA(\Gg)\hookrightarrow U^L_\BA(\Gg)$, 
$U_\BA(\Gh)\hookrightarrow U^L_\BA(\Gh)$.
We define an $\BA$-subalgebra $D_\BA$ of $\End_\BA(A_\BA)$ by
\[
D_\BA=
\langle
\ell_\varphi, \deru_u, \sigma_{2\lambda}
\mid
\varphi\in A_\BA, u\in U_\BA(\Gg), \lambda\in\Lambda
\rangle
\subset 
\End_\BA(A_\BA).
\]
Then $D_\BA$ turns out to be a $\Lambda$-graded $\BA$-algebra by the 
$\Lambda$-grading 
$
D_\BA=\bigoplus_{\lambda\in\Lambda}D_\BA(\lambda)
$
given by
\[
D_\BA(\lambda)
=
\{P\in D_\BA\mid
P(A_\BA(\mu))\subset A_\BA(\lambda+\mu)
\quad(\mu\in\Lambda)\}.
\]
We define a $\Lambda$-graded $\BC$-algebra $D_\zeta$ by
\[
D_\zeta=\BC\otimes_\BA D_\BA,
\]
where $\BA\to\BC$ is given by $q^{1/|\Lambda/Q|}\mapsto\zeta^{1/|\Lambda/Q|}$.
It is easily seen that in the $\BC$-algebra $D_\zeta$ we have
\begin{equation}
\label{eq:sigmaHC}
z\in Z_\Har(U_\zeta(\Gg)),
\quad
\Xi(z)=\sum_{\lambda\in\Lambda}a_\lambda k_{2\lambda}
\;
\Longrightarrow
\;
\deru_z=\sum_{\lambda\in\Lambda}a_\lambda\sigma_{2\lambda}.
\end{equation}

Regarding  $A_\zeta$  as a $\Lambda$-graded $\BC$-subalgebra of $D_\zeta$ by
$A_\zeta\to D_\zeta$ ($\varphi\mapsto\ell_\varphi$)
we define a category 
$\Mod(\DD_{\CB_\zeta})$ by
\begin{equation}
\Mod(\DD_{\CB_\zeta})
=
\Mod_\Lambda(D_\zeta)
/
(\Mod_\Lambda(D_\zeta)\cap\Tor_{\Lambda^+}(A_\zeta)).
\end{equation}
The natural exact functor
\[
\omega^*:
\Mod_{\Lambda}(D_\zeta)
\to
\Mod(\DD_{\CB_\zeta})
\]
admits a right adjoint 
\[
\omega_{*}:
\Mod(\DD_{\CB_\zeta})
\to
\Mod_{\Lambda}(D_\zeta),
\]
which is left exact.
For $M\in\Mod_\Lambda(D_\zeta)$ the zero-th part 
$M(0)$ turns out to be a $U_\zeta(\Gg)$-module via
$U_\zeta(\Gg)\to D_\zeta(0)$ ($u\mapsto\deru_u$),
and hence
we obtain a left exact functor
\begin{equation}
\label{eq:Gamma1}
\Gamma:\Mod(\DD_{\CB_\zeta})
\to
\Mod(U_\zeta(\Gg))
\qquad
(M\mapsto(\omega_{*}M)(0)).
\end{equation}

For $t\in H$ we define full subcategories
$\Mod(\DD_{\CB_\zeta,t})$ and
$\Mod_t(\DD_{\CB_\zeta})$ of 
$\Mod(\DD_{\CB_\zeta})$ 
as follows.
Let
$\Mod_{\Lambda}(D_\zeta;t)$ be the full subcategory of $\Mod_{\Lambda}(D_\zeta)$ consisting of 
$M\in \Mod_{\Lambda}(D_\zeta)$ satisfying
\[
\sigma_{2\lambda}|_{M(\mu)}
=\theta_\lambda(t)
\zeta^{2(\lambda,\mu)}
\id
\qquad(\mu\in\Lambda).
\]
We also denote by
$\Mod_{\Lambda,t}(D_\zeta)$ the full subcategory of $\Mod_{\Lambda}(D_\zeta)$ consisting of 
$M\in \Mod_{\Lambda}(D_\zeta)$ 
such that for any $m\in M(\mu)$ with $\mu\in\Lambda$ there exists some $n$ such that 
\[
(\sigma_{2\lambda}-
\theta_\lambda(t)
\zeta^{2(\lambda,\mu)})^nm=0.
\]
Then we define 
$\Mod(\DD_{\CB_\zeta,t})$ and
$\Mod_t(\DD_{\CB_\zeta})$ 
by
\begin{align}
\Mod(\DD_{\CB_\zeta,t})
=&
\Mod_{\Lambda}(D_\zeta;t)
/
(\Mod_{\Lambda}(D_\zeta;t)\cap\Tor_{\Lambda^+}(A_\zeta)),
\\
\Mod_t(\DD_{\CB_\zeta})
=&
\Mod_{\Lambda,t}(D_\zeta)
/
(\Mod_{\Lambda,t}(D_\zeta)\cap\Tor_{\Lambda^+}(A_\zeta)).
\end{align}
Note that  $\Mod(\DD_{\CB_\zeta,t})$ is an analogue of the category of quasi-coherent $\DD_{\CB,\nu}$-modules, where $\DD_{\CB,\nu}$ is the ring of twisted differential operators on the ordinary flag manifold $\CB$ corresponding to the parameter $\nu\in\Gh^*$.
By restricting $\omega^*$ and $\omega_*$ we obtain 
natural exact functors
\[
\omega_t^*:
\Mod_{\Lambda}(D_\zeta;t)
\to
\Mod(\DD_{\CB_\zeta,t}),
\qquad
{}_t\omega^*:
\Mod_{\Lambda,t}(D_\zeta)
\to
\Mod_t(\DD_{\CB_\zeta})
\]
and their right adjoints
\[
\omega_{t*}:
\Mod(\DD_{\CB_\zeta,t})
\to
\Mod_{\Lambda}(D_\zeta;t),
\qquad
{}_t\omega_{*}:
\Mod_t(\DD_{\CB_\zeta})
\to
\Mod_{\Lambda,t}(D_\zeta),
\]
which are left exact.
Moreover, \eqref{eq:Gamma1} induces 
\begin{align}
\label{eq:Gamma2}
&\Gamma:\Mod(\DD_{\CB_\zeta,t})
\to
\Mod(U_\zeta(\Gg)_{[t]})
\qquad
(M\mapsto(\omega_{t*}M)(0))
\\
\label{eq:Gamma22}
&\Gamma:\Mod_t(\DD_{\CB_\zeta})
\to
\Mod_{[t]}(U_\zeta(\Gg))
\qquad
(M\mapsto({}_t\omega_{*}M)(0))
\end{align}
by \eqref{eq:sigmaHC}.
Here, $\Mod_{[t]}(U_\zeta(\Gg))$  is the full subcategory of $\Mod(U_\zeta(\Gg))$ 
consisting of $M\in \Mod(U_\zeta(\Gg))$ such that for any $m\in M$ we have $\Ker(\xi_{[t]})^nm=0$ for some $n$.

We define $\modu(\DD_{\CB_\zeta})$, 
$\modu(\DD_{\CB_\zeta,t})$, 
$\modu_t(\DD_{\CB_\zeta})$ 
similarly to 
$\Mod(\DD_{\CB_\zeta})$, 
$\Mod(\DD_{\CB_\zeta,t})$, 
$\Mod_t(\DD_{\CB_\zeta})$ 
respectively
using $\modu_\Lambda(D_\zeta)$ instead of 
$\Mod_\Lambda(D_\zeta)$.
They are full subcategories of 
$\Mod(\DD_{\CB_\zeta})$, 
$\Mod(\DD_{\CB_\zeta,t})$,
$\Mod_t(\DD_{\CB_\zeta})$
respectively, and 
\eqref{eq:Gamma1}, \eqref{eq:Gamma2}, 
\eqref{eq:Gamma22}
 induce
\begin{align}
\label{eq:Gamma1A}
&\Gamma:\modu(\DD_{\CB_\zeta})
\to
\modu(U_\zeta(\Gg)),
\\
\label{eq:Gamma2A}
&\Gamma:\modu(\DD_{\CB_\zeta,t})
\to
\modu(U_\zeta(\Gg)_{[t]}),
\\
\label{eq:Gamma22A}
&\Gamma:\modu_t(\DD_{\CB_\zeta})
\to
\modu_{[t]}(U_\zeta(\Gg)),
\end{align}
where $\modu_{[t]}(U_\zeta(\Gg))=\Mod_{[t]}(U_\zeta(\Gg))\cap \modu(U_\zeta(\Gg))$.
\begin{remark}
In our previous papers \cite{T1} and \cite{T2}, we used, 
instead of the categories
$\Mod(\DD_{\CB_\zeta})$ and 
$\Mod(\DD_{\CB_\zeta,t})$ ($t\in H$)
described above,
the categories
$\Mod(\DD^1_{\CB_\zeta})$, 
$\Mod(\DD^1_{\CB_\zeta,t})$
defined starting from
\[
D^1_\BA=
\langle
\ell_\varphi, \deru_u, \sigma_{\lambda}
\mid
\varphi\in A_\BA, u\in U_\BA(\Gg), \lambda\in\Lambda
\rangle
\subset \End_\BA(A_\BA)
\]
which is slightly larger than 
$D_\BA$.
In fact we have 
$\Mod(\DD^1_{\CB_\zeta,t})
\cong
\Mod(\DD_{\CB_\zeta,t^2})$ 
by \cite[Proposition 6.1, Remark 6.5]{TR}.
\end{remark}

For $\lambda\in\Lambda$ we define $t_\lambda\in H$ by
\begin{equation}
\label{eq:t}
\theta_\mu(t_\lambda)=\zeta^{2(\lambda,\mu)}
\qquad
(\mu\in\Lambda).
\end{equation}
The functor \eqref{eq:shift2} induces an exact functor
\begin{equation}
(\bullet)[\lambda]:
\Mod(\DD_{\CB_\zeta,t})\to
\Mod(\DD_{\CB_\zeta,tt_\lambda})
\qquad(\lambda\in\Lambda).
\end{equation}
Note
\begin{equation}
w\circ t_\lambda=t_{w(\lambda+\rho)-\rho}
\qquad(w\in W, \lambda\in\Lambda).
\end{equation}
We say that $t\in H$ is regular if $|W\circ t|=|W|$.

In \cite{T0} we proved in the case where $\zeta$ is transcendental that 
\eqref{eq:Gamma2A} gives an equivalence of abelian categories if $t=t_\lambda$ with
$\lambda\in \Lambda^+$.
This is an analogue of the Beilinson-Bernstein correspondence.
For general $\zeta\in\BC^\times$ we have the following conjecture.
\begin{conjecture}
\label{conj:BB2}
The derived functors
\begin{align}
\label{eq:RGamma2A}
R\Gamma:D^b(\modu(\DD_{\CB_\zeta,t}))
\to
D^b(\modu(U_\zeta(\Gg)_{[t]}))
\\
\label{eq:RGamma22A}
R\Gamma:D^b(\modu_t(\DD_{\CB_\zeta}))
\to
D^b(\modu_{[t]}(U_\zeta(\Gg)))
\end{align}
of \eqref{eq:Gamma2A}, \eqref{eq:Gamma22A}
between bounded derived categories give 
equivalences of triangulated categories
if $t\in H$ is regular.
\end{conjecture}

In the case $\zeta$ is not a root of $1$, 
Conjecture \ref{conj:BB2} is regarded as an
analogue of a known result for twisted $D$-modules on the ordinary flag manifold $\CB$ over $\BC$ (see \cite{BB2}, \cite{KT}).
However, the proof in the ordinary case uses a certain integral transform, which is sometimes called the Radon transform, and is not directly applied to our case.
It is an interesting problem to define the Radon transforms for quantized situation.
Conjecture \ref{conj:BB2} in the case $\zeta$ is a root of $1$ will be discussed in Section \ref{sec:ED}.

\subsection{$\CU_{\CB_\zeta}$-modules}
\label{sec:EV}
For $V\in\modu_\inte(U_\zeta^L(\Gg))$ we set
$E_V=V\otimes A_\zeta$
We regard it as a right $A_\zeta$-module via the right multiplication on the second factor. 
We also regard it as a left $U_\zeta^L(\Gg)$-module by 
\[
u(v\otimes\varphi)=\sum_{(u)}u_{(0)}v\otimes \deru_{u_{(1)}}\varphi
\qquad
(u\in U_\zeta^L(\Gg),\;
v\in V,\; 
\varphi\in A_\zeta).
\]
We will identify $E_V$ with $A_\zeta\otimes V$ through the isomorphism
$R^{A_\zeta,V}:A_\zeta\otimes V\to V\otimes A_\zeta$ of $U_\zeta^L(\Gg)$-modules (see \eqref{eq:R}).
We define a left $A_\zeta$-module structure of $E_V$ by the left multiplication on the first factor of $E_V=A_\zeta\otimes V$.
Then $E_V$ turns out to be an $A_\zeta$-bimodule.
Moreover, by 
\[
R^{A_\zeta,V}(1\otimes v)=v\otimes 1
\qquad (v\in V)
\]
we can regard $V$ as a $U_\zeta^L(\Gg)$-submodule of $E_V$ by the embedding 
\[
V\hookrightarrow E_V
\qquad
(v\mapsto v\otimes 1\in V\otimes A_\zeta=E_V).
\]
It is easily seen that the functor 
$
E_V\otimes_{A_\zeta}(\bullet):
\Mod_\Lambda(A_\zeta)
\to
\Mod_\Lambda(A_\zeta)
$
induces an exact functor 
\begin{equation}
\label{eq:funcV}
V\otimes(\bullet):
\Mod(\CO_{\CB_\zeta})
\to
\Mod(\CO_{\CB_\zeta}).
\end{equation}
We refer the readers to \cite[4.3]{T0} for the properties of $E_V$ mentioned above.

\begin{proposition}
\label{prop:tensVA}
Assume that  $V\in\modu_\inte(U^L_\zeta(\Gg))$ is equipped with a $U^L_\zeta(\Gb^-)$-stable filtration 
\begin{equation}
\label{eq:filt1A}
V=V_1\supset V_2\supset\cdots\supset V_m\supset V_{m+1}=\{0\}
\end{equation}
such that $V_j/V_{j+1}$ is a one-dimensional $U^L_\zeta(\Gb^-)$-module with character $\xi_j\in\Lambda$.
Then for $\CM\in\Mod(\CO_{\CB_\zeta})$ we have
 a functorial filtration 
\[
V\otimes \CM=
\CM_1\supset \CM_2\supset\cdots\supset 
\CM_m\supset 
\CM_{m+1}=\{0\}
\]
of $V\otimes\CM$ 
such that 
$\CM_j/\CM_{j+1}
\cong
\CM[\xi_j]
$
for $j=1,\dots, m$.
\end{proposition}
\begin{proof}
A proof of this result over $\BF$ is given in \cite[4.4]{T0}.
Here, we give a different proof using the equivalence
\eqref{eq:1equiv2} for $J=\emptyset$.

It is easily seen that the functor
\[
V\otimes(\bullet):
\widetilde{\Mod}(\CO_{\CB_\zeta})
\to
\widetilde{\Mod}(\CO_{\CB_\zeta})
\]
corresponding to \eqref{eq:funcV} is given by
$
M\mapsto V\otimes M
$.
Here, the right $U^L_\zeta(\Gb^-)$-module structure of $V\otimes M$ is given by the  
right $U^L_\zeta(\Gb^-)$-module structure of  $M$ (the action of $U^L_\zeta(\Gb^-)$ on $V$ is trivial), and 
the left $\CO_\zeta(G)$-module structure of $V\otimes M$ is given by the composite of 
\[
\CO_\zeta(G)\otimes V\otimes M
\xrightarrow{R^{\CO_\zeta(G),V}\otimes 1}
V\otimes\CO_\zeta(G)\otimes M
\xrightarrow{1\otimes \alpha}
V\otimes M,
\]
where 
$\alpha:\CO_\zeta(G)\otimes M\to M$ is the left $\CO_\zeta(G)$-module structure of $M$.

Let us define $(V\otimes M)^\dagger\in\widetilde{\Mod}(\CO_{\CB_\zeta})$ as follows.
As a vector space we have $(V\otimes M)^\dagger=V\otimes M$.
The right $U_\zeta^L(\Gb^-)$-module structure 
of $(V\otimes M)^\dagger$ 
is given by 
\[
(v\otimes m)y=
\sum_{(y)}
(Sy_{(0)})v\otimes my_{(1)}
\qquad
(v\in V,\; m\in M,\; y\in  U^L_\zeta(\Gb^-)),
\]
and the left $\CO_\zeta(G)$-module structure of $(V\otimes M)^\dagger$ is given by the composite of 
\[
\CO_\zeta(G)\otimes V\otimes M
\xrightarrow{\sum_{\gamma\in Q^+}R'_\gamma\otimes1}
V\otimes \CO_\zeta(G)\otimes M
\xrightarrow{1\otimes \alpha} V\otimes M,
\]
where 
$R'_\gamma:\CO_\zeta(G)\otimes V\to V\otimes \CO_\zeta(G)$
is given by
\[
\CR_{\zeta,\gamma}=\sum_pr_p\otimes r'_p
\;\;\Longrightarrow\;\;
R'_\gamma(\varphi\otimes v)
=
\zeta^{-(\lambda-\gamma,\mu)}
\sum_pr'_pv\otimes\varphi r_p
\]
for $v\in V_\mu$ and $\varphi\in\CO_\zeta(G)$ such that 
$\varphi h=\chi_\lambda(h)\varphi$ for any $h\in U_\zeta^L(\Gh)$.

One can check that  the linear map 
$
V\otimes M\to V\otimes M
$ defined by 
\[
v\otimes m\mapsto\sum_iv_i\otimes \varphi_i m
\quad\text{if}\quad 
uv=\sum_{i}\langle\varphi_i,u\rangle v_i \quad(u\in U^L_\zeta(\Gg), \; v\in V)
\]
gives an isomorphism $V\otimes M\to(V\otimes M)^\dagger$ in 
$\widetilde{\Mod}(\CO_{\CB_\zeta})$.

It remains to show that $V_j\otimes M$ is a subobject of $(V\otimes M)^\dagger\in \widetilde{\Mod}(\CO_{\CB_\zeta})$ and 
we have 
$(V_j\otimes M)/(V_{j+1}\otimes M)\cong M[\xi_j]$. 
This is easily seen from the definition of $(V\otimes M)^\dagger$.
\end{proof}
\begin{remark}
Let $\GM_1$ be a quasi-coherent $\CO_G$-module equivariant with respect to the action of $B^-$ on $G$ given by the left multiplication.
It corresponds to an object $M_1$ of $\widetilde{\Mod}(\CO_\CB)$ (see Remark \ref{rem:CC}).
Let $V_1$ be a finite-dimensional $G$-module.
Then $V\otimes M$
(resp.\ $(V\otimes M)^\dagger$) in the proof of Proposition \ref{prop:tensVA} is 
an analogue of the object of $\widetilde{\Mod}(\CO_\CB)$ corresponding to $V_1\otimes_\BC \GM_1$ equipped with right
$B^-$action given by 
\[
(v\otimes m)b=v\otimes mb
\qquad
(\text{resp.\ }
(v\otimes m)b=vb^{-1}\otimes mb)
\]
for $v\in V_1$, $m\in\GM_1$, $b\in B^-$.
\end{remark}

We regard $A_\zeta\otimes U_\zeta(\Gg)$ as a $\BC$-algebra via the smash product of the algebra $A_\zeta$ and the  Hopf algebra $U_\zeta(\Gg)$ acting on it. Namely, we define the multiplication of 
 $A_\zeta\otimes U_\zeta(\Gg)$ by
\[
(\varphi\otimes u)(\varphi'\otimes u')
=\sum_{(u)}
\varphi\cdot\deru_{u_{(0)}}(\varphi')\otimes
{u}_{(1)}u'
\qquad
(\varphi, \varphi'\in A_\zeta,\; u,u'\in U_\zeta(\Gg)).
\]
We will identify $A_\zeta$ and $U_\zeta(\Gg)$ as subalgebras of $A_\zeta\otimes U_\zeta(\Gg)$ by the embeddings
$\varphi\mapsto\varphi\otimes1$, $u\mapsto 1\otimes u$ 
for $\varphi\in A_\zeta$, $u\in U_\zeta(\Gg)$.
We regard $A_\zeta\otimes U_\zeta(\Gg)$ as a  $\Lambda$-graded $\BC$-algebra by 
$(A_\zeta\otimes U_\zeta(\Gg))(\lambda)=
A_\zeta(\lambda)\otimes U_\zeta(\Gg)$.
We have a canonical homomorphism
\begin{equation}
\label{eq:ED}
A_\zeta\otimes U_\zeta(\Gg)\to D_\zeta
\qquad
(\varphi\otimes u\mapsto\ell_\varphi\deru_u)
\end{equation}
of 
$\Lambda$-graded algebras.
We set 
\begin{equation}
\label{eq:UBZ}
\Mod(\CU_{\CB_\zeta})=
\Mod_\Lambda(A_\zeta\otimes U_\zeta(\Gg))/
(\Mod_\Lambda(A_\zeta\otimes U_\zeta(\Gg))
\cap\Tor_{\Lambda^+}(A_\zeta)).
\end{equation}
Then \eqref{eq:funcV} induces an exact functor 
\begin{equation}
\label{eq:FV2}
V\otimes(\bullet):
\Mod(\CU_{\CB_\zeta})
\to
\Mod(\CU_{\CB_\zeta})
\end{equation}
for $V\in\modu_\inte(U_\zeta^L(\Gg))$.
Here, for $M\in\Mod_\Lambda(A_\zeta\otimes U_\zeta(\Gg))$ we regard $V\otimes M$ as a $U_\zeta(\Gg)$-module by the tensor product of two $U_\zeta(\Gg)$-modules.

By Proposition \ref{prop:tensVA} we have the following.
\begin{proposition}
\label{prop:tensVA2}
Let $V$ be as in Proposition \ref{prop:tensVA}.
Let $\CM\in\Mod(\DD_{\CB_\zeta,t})$, and regard it as an object of $\Mod(\CU_{\CB_\zeta})$ via \eqref{eq:ED}.
Then $V\otimes\CM\in \Mod(\CU_{\CB_\zeta})$
has 
 a functorial filtration 
\[
V\otimes \CM=
\CM_1\supset \CM_2\supset\cdots\supset 
\CM_m\supset 
\CM_{m+1}=\{0\}
\]
such that 
$\CM_j/\CM_{j+1}
\cong
\CM[\xi_j]
\in
\Mod(\DD_{\CB_\zeta,tt_{\xi_j}})$
for $j=1,\dots, m$.
\end{proposition}

\section{Representation theory at roots of unity}
\label{sec:Rep}
In the rest of this paper we assume that $\zeta\in\BC^\times$ is a primitive $\ell$-th root of unity,
where $\ell$ is an integer satisfying the following conditions (a1), (a2), (a3):
\begin{itemize}
\item[(a1)]
$\ell>1$ is odd,
\item[(a2)]
$\ell$ is prime to $|\Lambda/Q|$,
\item[(a3)]
$\ell$ is prime to 3 if $\Delta$ is of type $G_2$.
\end{itemize}
We take $\zeta^{1/|\Lambda/Q|}$ to be a primitive $\ell$-th root of unity (see (a2)).
\subsection{Center}
For $\alpha\in\Delta^+$ we define the positive and the negative root vectors $e_\alpha, f_\alpha\in U_\zeta(\Gg)$ using Lusztig's braid group action (see \cite{Lbook}).
Then the elements
$k_{\ell\lambda}$ ($\lambda\in\Lambda$),
$e_\alpha^\ell, \; f_\alpha^\ell$ ($\alpha\in\Delta^+$)
belong to the center $Z(U_\zeta(\Gg))$
of $U_\zeta(\Gg)$.
We denote by $Z_\Fr(U_\zeta(\Gg))$ the subalgebra of 
$Z(U_\zeta(\Gg))$ generated by them.
\begin{remark}
The elements $e_\alpha$, $f_\alpha$ actually depends on the choice of a reduced expression of the longest element of the Weyl group $W$; however, the subalgebra $Z_\Fr(U_\zeta(\Gg))$ is independent of the  choice.
\end{remark}

Define a closed subgroup $K$ of $G\times G$ by
\[
K=\{
(xs, ys^{-1})
\mid
x\in N^+, y\in N^-, s\in H\}.
\]
\begin{proposition}[\cite{DP}, \cite{Gav}]
\begin{itemize}
\item[(i)]
The algebra $Z_\Fr(U_\zeta(\Gg))$ is naturally isomorphic to the coordinate algebra $\CO(K)$ of $K$.
\item[(ii)]
The algebra 
$Z_\Fr(U_\zeta(\Gg))\cap Z_\Har(U_\zeta(\Gg))$ is naturally isomorphic to 
$\CO(H)^W=\CO(H/W)$.
\item[(iii)]
We have an isomorphism
\[
Z(U_\zeta(\Gg))\cong
Z_\Fr(U_\zeta(\Gg))
\otimes_{
Z_\Fr(U_\zeta(\Gg))\cap Z_\Har(U_\zeta(\Gg))
}
Z_\Har(U_\zeta(\Gg))
\quad(z_1z_2\leftrightarrow z_1\otimes z_2)
\]
of algebras.
\end{itemize}
\end{proposition}
Moreover, 
under the identification
\begin{gather*}
Z_\Fr(U_\zeta(\Gg))\cong\CO(K),
\quad
Z_\Har(U_\zeta(\Gg))\cong\CO(H/W\circ),
\\
Z_\Fr(U_\zeta(\Gg))\cap Z_\Har(U_\zeta(\Gg))
\cong
\CO(H/W),
\end{gather*}
the inclusions
\begin{gather*}
Z_\Fr(U_\zeta(\Gg))\cap Z_\Har(U_\zeta(\Gg))
\to
Z_\Har(U_\zeta(\Gg)),
\\
Z_\Fr(U_\zeta(\Gg))\cap Z_\Har(U_\zeta(\Gg))
\to
Z_\Fr(U_\zeta(\Gg))
\end{gather*}
correspond to
$
H/W\circ\to H/W
$ ($[t]\mapsto[t^\ell]$),
and 
\[
\varkappa:K\to H/W
\]
which is the composite of 
\[
\eta:K\to G
\qquad((x_1,x_2)\mapsto x_1x_2^{-1})
\]
with the Steinberg map
$\St:G\to H/W$ respectively.
Hence $Z(U_\zeta(\Gg))$ is isomorphic to the coordinate algebra $\CO(\CX)$ of the affine algebraic variety
$
\CX=K\times_{H/W}(H/W\circ)$.
For $g\in G$ we denote its Jordan decomposition by $g=g_sg_u$.
By definition we have
\begin{equation}
\CX
=
\{(k,[t])\in 
K\times(H/W\circ)
\mid
t^\ell\in\Ad(G)(\eta(k)_s)\}.
\end{equation}

For $k\in K$ we denote by
\begin{equation}
\xi^k:Z_\Fr(U_\zeta(\Gg))\to\BC
\end{equation}
the corresponding character of $Z_\Fr(U_\zeta(\Gg))$.
For 
$(k,[t])\in \CX$ 
we have a character 
\begin{equation}
\xi^k_{[t]}:Z(U_\zeta(\Gg))\to\BC
\end{equation}
of the total center $Z(U_\zeta(\Gg))$
such that
$\xi^k_{[t]}|_{Z_\Har(U_\zeta(\Gg))}=\xi_{[t]}$, and 
$\xi^k_{[t]}|_{Z_\Fr(U_\zeta(\Gg))}=\xi^k$.
We set
\begin{align}
U_\zeta(\Gg)^k=&
U_\zeta(\Gg)\otimes_{Z_\Fr(U_\zeta(\Gg))}\BC
&(k\in K),
\\
U_\zeta(\Gg)^k_{[t]}=&
U_\zeta(\Gg)\otimes_{Z(U_\zeta(\Gg))}\BC
&((k,[t])\in 
\CX)
\end{align}
with respect to $\xi^k$, $\xi_{[t]}^k$ respectively.
Recall also that we have
\begin{align}
U_\zeta(\Gg)_{[t]}=&
U_\zeta(\Gg)\otimes_{Z_\Har(U_\zeta(\Gg))}\BC
&([t]\in H/W\circ)
\end{align}
with respect to $\xi_{[t]}$.
We define $\wU_\zeta(\Gg)^k$
(resp.\
$\wU_\zeta(\Gg)_{[t]}$, 
$\wU_\zeta(\Gg)^k_{[t]}$) 
to be the completion of $U_\zeta(\Gg)$ at 
the maximal ideal
$\Ker(\xi^k)$
(resp.\ 
$\Ker(\xi_{[t]})$, 
$\Ker(\xi^k_{[t]})$) 
of $Z_\Fr(U_\zeta(\Gg))$
(resp.\
$Z_\Har(U_\zeta(\Gg))$, 
$Z(U_\zeta(\Gg))$).

A version of Schur's lemma tells us that
for any irreducible $U_\zeta(\Gg)$-module $M$ there exists some 
$(k,[t])\in \CX$ 
such that $z|_M=\xi^k_{[t]}(z)\id$ for any $z\in Z(U_\zeta(\Gg))$.
It is known by \cite{DKP} that
for $k, k'\in K$ 
we have
$U_\zeta(\Gg)^k\cong U_\zeta(\Gg)^{k'}$
if $\eta(k)$ is conjugate to $\eta(k')$.

We say that $\tilde{t}\in H$ is unramified if 
we have
$\theta_\alpha(\tilde{t})=\zeta^{-2(\rho,\alpha)}$ 
for any $\alpha\in\Delta$ satisfying $\theta_\alpha(\tilde{t}^\ell)=1$.
We denote by $H_\ur$ the set of unramified elements of $H$.

For $k\in K$ we give a description of the set $\CX^k$ consisting of 
the Harish-Chandra central characters
$\xi_{[t]}:Z_\Har(U_\zeta(\Gg))\to\BC$ which is compatible with the Frobenius central character $\xi^{{k}}:Z_\Fr(U_\zeta(\Gg))\to\BC$.
Fix
$h\in H\cap\Ad(G)(\eta(k)_s)$.
Then we have 
$H\cap\Ad(G)(\eta(k)_s)=W(h)$.
By the above argument we have 
\[
\CX^k\cong H(W(h))/W\circ
\cong
H(h)/W_h\circ,
\]
where
\begin{gather*}
H(W(h))
=
\{t\in H
\mid
t^\ell\in W(h)\},
\quad
H(h)=\{t\in H\mid t^\ell=h\},
\\
W_h=\{w\in W\mid w({h})={h}\}.
\end{gather*}
It is known that $W_h$ is the Weyl group of the root system
\[
\Delta_h=
\{\alpha\in\Delta\mid
\theta_\alpha(h)=1\}.
\]
We denote by $H_\ell$ the set of $d\in H$ satisfying $d^\ell=1$.
By our assumption on $\ell$ we have
\begin{equation}
\label{eq:Hl}
H_\ell=\{t_\lambda\mid\lambda\in\Lambda\}
\cong\Lambda/\ell\Lambda
\end{equation}
(see \eqref{eq:t} for the notation).
Set
\begin{align*}
H_\ur(h)=&H(h)\cap H_{\ur}
=\{
\tilde{t}\in H(h)\mid
\theta_\alpha(\tilde{t})=\zeta^{-2(\rho,\alpha)}
\;\;(\alpha\in{\Delta}_h)\}.
\end{align*}
Then we have
\[
H(h)=\bigsqcup_{\tilde{t}\in H_\ur(h)}
\tilde{t}H_\ell.
\]
It is easily seen that for $\tilde{t}\in H_\ur(h)$, 
$d\in H_\ell$, and $w\in{W}_h$
we have
$w\circ(\tilde{t}d)=\tilde{t}(w(d))$.

Hence we obtain
\begin{align*}
\CX^k\cong&
H(h)/{W}_h\circ
=
\bigsqcup_{\tilde{t}\in H_\ur(h)}
(\tilde{t}H_\ell)/{W}_h\circ,
\\
(\tilde{t}H_\ell)/{W}_h\circ
\cong&
H_\ell/{W}_h
\cong
{\Lambda}/({W}_h\ltimes(\ell{\Lambda}))
\qquad
(\tilde{t}\in H_\ur(h)).
\end{align*}

\subsection{Tensor product with integrable highest weight modules}
\label{subsec:TP}
Let $k\in K$, and take 
$h\in H$ such that $\Ad(G)(\eta(k)_s)\cap H=W(h)$.
We denote by $\modu^k(U_\zeta(\Gg))$ the category of finitely generated $U_\zeta(\Gg)$-modules on which $z-\xi^k(z)$ for $z\in Z_\Fr(U_\zeta(\Gg))$ acts locally nilpotently.
For $[t]\in H(W(h))/W\circ$ 
we set 
$\modu^k(U_\zeta(\Gg)_{[t]})
=\modu^k(U_\zeta(\Gg))
\cap
\modu(U_\zeta(\Gg)_{[t]})$.
We also denote by 
$\modu^k_{[t]}(U_\zeta(\Gg))$ the category of 
 finitely generated $U_\zeta(\Gg)$-modules on which $z-\xi_{[t]}^k(z)$ for $z\in Z(U_\zeta(\Gg))$ acts locally nilpotently.

Assume $M\in\modu^k(U_\zeta(\Gg))$.
Since $U_\zeta(\Gg)^k$ is finite-dimensional, 
$M$ is a finite-dimensional $U_\zeta(\Gg)$-module.
In particular, it has a finite composition series.
Let $N$ be a composition factor of $M$.
By Schur's lemma there exists some $[t]\in 
H(W(h))/W\circ$ such that any $z\in Z_\Har(U_\zeta(\Gg))$ 
acts on $N$ by the scalar multiplication of $\xi_{[t]}(z)$.
It follows that  we have the direct sum decomposition 
\begin{equation}
\label{eq:Hd}
\modu^k(U_\zeta(\Gg))
=\bigoplus_{[t]\in H(W(h))/W\circ}
\modu^k_{[t]}(U_\zeta(\Gg)).
\end{equation}

If $V$ is a $U^L_\zeta(\Gg)$-module, 
we will regard it as a $U_\zeta(\Gg)$-module 
via the natural homomorphism $U_\zeta(\Gg)\to U^L_\zeta(\Gg)$ induced by 
$U_\BA(\Gg)\subset U^L_\BA(\Gg)$.
For a $U_\zeta(\Gg)$-module $M$ we regard 
$V\otimes M$ as a $U_\zeta(\Gg)$-module through the comultiplication of $U_\zeta(\Gg)$.
We see easily the following.
\begin{lemma}
For $V\in\modu_\inte(U^L_\zeta(\Gg))$ and  $M\in\modu^k(U_\zeta(\Gg))$ we have
$V\otimes M\in\modu^k(U_\zeta(\Gg))$.
\end{lemma}
Let $V\in\modu_\inte(U^L_\zeta(\Gg))$.
Fix 
$z\in Z_\Har(U_\zeta(
\Gg))$.
For $\nu\in \Lambda$ we define 
$a_\nu\in {}^eU_\zeta(\Gh)$ by
\[
\Xi(z)
=\sum_{\mu\in \Lambda}
c_\mu k_{2\mu}\in {}^eU_\zeta(\Gh)^{W\circ}
\;\Longrightarrow\;
a_\nu=
\sum_{\mu\in \Lambda}c_\mu\zeta^{2(\mu,\nu)} k_{2\mu}\in {}^eU_\zeta(\Gh),
\]
and set
\[
F_z(x)=
\prod_{\nu\in\wt(V)}(x-a_{\nu})
\in {}^eU_\zeta(\Gh)[x].
\]
Here, $\wt(V)$ denotes the multi-set consisting of weights of $V$, where each weight $\nu$ of $V$ is counted $\dim V_\nu$-times.
We see easily that
$F_z(x)\in {}^eU_\zeta(\Gh)^{W\circ}[x]$,
and hence we can write
\begin{equation}
\label{eq:Fz}
F_z(x)=x^n+\sum_{k=0}^{n-1}\Xi(z_k)x^k
\qquad(z_k\in Z_\Har(U_\zeta(\Gg))),
\end{equation}
where $n=\dim V$.

Define
\[
\delta:
Z_\Har(U_\zeta(\Gg))\to 
\End(V)\otimes U_\zeta(\Gg)
\]
to be the composite of 
\[
Z_\Har(U_\zeta(\Gg))
\subset 
U_\zeta(\Gg)
\xrightarrow{\Delta}
U_\zeta(\Gg)\otimes U_\zeta(\Gg)
\xrightarrow{\sigma_V\otimes1}
\End(V)\otimes U_\zeta(\Gg),
\]
where $\sigma_V:
U_\zeta(\Gg)
\to
\End(V)
$ is the corresponding representation.

We have the following analogue of Kostant's result for the ordinary enveloping algebras.
\begin{proposition}
\label{prop:Kostant}
Let $V\in \modu_\inte(U^L_\zeta(\Gg))$.
We assume that there exist a finite-dimensional integrable $U_\BF(\Gg)$-module $V_\BF$ and its $U_\BA^L(\Gg)$-stable $\BA$-form $V_\BA$ such that $V=\BC\otimes_\BA V_\BA$ with respect to the specialization $q\mapsto \zeta$.
Let $z\in Z_\Har(U_\zeta(\Gg))$ and define $z_k\in Z_\Har(U_\zeta(\Gg))$ by \eqref{eq:Fz}.
Then we have
\[
\delta(z)^n+\sum_{k=0}^{n-1}
(1\otimes z_k)
\delta(z)^k
=0
\]
in $\End(V)\otimes U_\zeta(\Gg)$.
\end{proposition}
\begin{proof}
Set $\tilde{\BF}=\BC(q^{1/|\Lambda/Q})$, 
$\tilde{\BA}=\BC[q^{\pm1/|\Lambda/Q}]$.
The proof of the  corresponding statement  
over $\tilde{\BF}$, where $U_\zeta(\Gg)$, $V$, $Z_\Har(U_\zeta(\Gg))$ are replaced by 
$\tilde{\BF}\otimes_\BF U_\BF(\Gg)$, 
$\tilde{\BF}\otimes_\BF V_\BF$, $Z(\tilde{\BF}\otimes_\BF U_\BF(\Gg))$ 
 respectively,
is similar to that for the ordinary enveloping algebras given in 
\cite{Kos}, \cite{BG}.
Our assertion follows from this 
since $U_\zeta(\Gg)$ is a specialization of 
the $\tilde{\BA}$-form $\tilde{\BA}\otimes_\BA U_\BA(\Gg)$ of 
$\tilde{\BF}\otimes_\BF U_\BF(\Gg)$ and since
$Z(\tilde{\BA}\otimes_\BA U_\BA(\Gg))\to Z_\Har(U_\zeta(\Gg))$ is surjective.
\end{proof}

It follows easily from Proposition \ref{prop:Kostant} the following.
\begin{proposition}
\label{prop:8ccX}
Let $V\in \modu_\inte(U^L_\zeta(\Gg))$ be as in Proposition \ref{prop:Kostant}.
Assume $M\in\modu^k_{{[t]}}(U_\zeta(\Gg))$ for 
$[t]\in H(W(h))/W\circ$.
Then for
$z\in Z_\Har(U_\zeta(\Gg))$
we have
\[
\prod_{\nu\in\wt(V)}(z-\xi_{[tt_\nu]}(z))(V\otimes M)=0.
\]
\end{proposition}

In general for $V\in\modu_\inte(U^L_\zeta(\Gg))$ 
we regard
$V^*=\Hom_\BC(V,\BC)$
as an  object of $\modu_\inte(U^L_\zeta(\Gg))$ by 
\[
\langle uf,v\rangle=
\langle f, (Su)v\rangle
\qquad
(u\in U^L_\zeta(\Gg), \;
f\in V^*,\; 
v\in V).
\]
For $\lambda\in\Lambda^+$ 
we call 
$\nabla_\zeta(\lambda)=\Delta_\zeta(-w_0\lambda)^*\in
\modu_\inte(U^L_\zeta(\Gg))$ the dual Weyl module with highest weight $\lambda$.
Here, $w_0$ is the longest element of $W$.
There exists a non-zero homomorphism 
$\Delta_\zeta(\lambda)
\to
\nabla_\zeta(\lambda)
$ of $U^L_\zeta(\Gg)$-modules, which is unique up to constant multiple, and 
its image $L_\zeta(\lambda)$
is an irreducible $U^L_\zeta(\Gg)$-module.
Moreover, any irreducible integrable $U^L_\zeta(\Gg)$-module is isomorphic to $L_\zeta(\lambda)$ for some $\lambda\in\Lambda^+$.
\begin{corollary}
\label{cor:8ZX}
Let $V\in \modu_\inte(U^L_\zeta(\Gg))$, and let
$L_\zeta(\lambda_1)$, \dots, $L_\zeta(\lambda_r)$ for $\lambda_j\in\Lambda^+$ be its composition factors.
Let $t\in H(W(h))$, and define a subset $\CA$ of $H/W\circ$
(without multiplicity) by
\[
\CA=\bigcup_{j=1}^r\{[tt_\nu]\mid\nu\in\wt(\Delta_\zeta(\lambda_j))\}.
\]
Then for 
$
M\in\modu^k_{{[t]}}(U_\zeta(\Gg))
$
we have
$
V\otimes M\in\bigoplus_{[t']\in\CA}
\modu^k_{[t']}(U_\zeta(\Gg))$.
\end{corollary}

By \eqref{eq:Hd} we have
\begin{equation}
\label{eq:Hd2}
\modu(\wU_\zeta(\Gg)^k)
=\bigoplus_{[t]\in H(W(h))/W\circ}
\modu(\wU_\zeta(\Gg)^k_{[t]}),
\end{equation}
and by Corollary \ref{cor:8ZX} we easily obtain the following.
\begin{corollary}
Let $V\in \modu_\inte(U^L_\zeta(\Gg))$, 
$t\in H(W(h))$, $\CA\subset H/W\circ$ be as in 
Corollary \ref{cor:8ZX}.
Then for 
$
M\in\modu(\wU_\zeta(\Gg)^k_{{[t]}})
$
we have
$
V\otimes M\in\bigoplus_{[t']\in\CA}
\modu(\wU_\zeta(\Gg)^k_{[t']})$.
\end{corollary}

\section{Back to the ordinary flag manifold}
\label{sec:back}
The aim of this section is to relate the quantized partial flag manifold $\CP_{J, \zeta}$ with the ordinary partial flag manifold $\CP_J$ using the Frobenius morphism
$\Fr:\CP_{J,\zeta}\to\CP_J$, which is a morphism in the category of non-commutative schemes.
\subsection{$\GO$-modules} 
In the arguments below we identify $\Mod(\CO_{\CP_{J,1}})$
(resp.\ $\modu(\CO_{\CP_{J,1}})$)
with the category
$\Mod(\CO_{\CP_J})$
(resp.\ $\modu(\CO_{\CP_J})$)
consisting of quasi-coherent (resp.\ coherent)
$\CO_{\CP_J}$-modules, 
where $\CP_J=P_J^-\backslash G$ is the ordinary partial flag manifold 
(see Remark \ref{rem:P1}).
Note that for $\lambda\in\Lambda^J$ 
\begin{equation}
\CO_{\CP_J}[\lambda]:=
\omega_J^*(A_1[\lambda])\in\modu(\CO_{\CP_{J,1}})=\modu(\CO_{\CP_J})
\end{equation}
is the invertible $\CO_{\CP_J}$-module corresponding to $\lambda\in\Lambda^J$.

Let 
$
\Fr:U^L_\zeta(\Gg)\to U(\Gg)
$
be Lusztig's Frobenius homomorphism (see \cite{Lbook}).
It induces an embedding 
$
\CO(G)\hookrightarrow\CO_\zeta(G)
$
of Hopf algebras.
From this we obtain an embedding
$
A_1\hookrightarrow A_\zeta
$
of $\BC$-algebras.
Note that we have
$A_1(\lambda)\hookrightarrow A_\zeta(\ell\lambda)$ for $\lambda\in\Lambda$.
Define a 
$\Lambda^J$-graded algebra $A_{J,\zeta}^{(\ell)}
=\bigoplus_{\lambda\in\Lambda^J}
A_{J,\zeta}^{(\ell)}(\lambda)
$ by
\[
A_{J,\zeta}^{(\ell)}(\lambda)
=
A_{J,\zeta}(\ell\lambda)
\qquad
(\lambda\in\Lambda^J).
\]
Then $A_{J,1}
=\bigoplus_{\lambda\in\Lambda^J}
A_{1}(\lambda)$ is a $\Lambda^J$-graded central subalgebra of 
$A_{J,\zeta}^{(\ell)}$.
Moreover, $A_{J,\zeta}^{(\ell)}$ is finitely generated as an $A_{J,1}$-module.
Applying
$\omega_{J}^*:\modu_{\Lambda^J}(A_{J,1})\to
\modu(\CO_{\CP_{J,1}})$
to $A_{J,\zeta}^{(\ell)}$ regarded as an object of $\modu_{\Lambda^J}(A_{J,1})$, 
we set
\[
\GO_J=
\Fr_*\CO_{\CP_{J,\zeta}}:=
\omega_{J}^*A_{J,\zeta}^{(\ell)}
\in\modu(\CO_{\CP_{J,1}})
=\modu(\CO_{\CP_{J}}).
\]
Note that  the coherent $\CO_{\CP_J}$-module 
$\GO_J$ is naturally equipped with an 
$\CO_{\CP_{J}}$-algebra structure through the multiplication of $A_{J,\zeta}^{(\ell)}$.
Moreover, 
the category
\begin{gather*}
\Mod_{\Lambda^J}(A_{J,\zeta}^{(\ell)})/
(\Mod_{\Lambda^J}(A_{J,\zeta}^{(\ell)})
\cap
\Tor_{\Lambda^{J+}}(A_{J,1}))
\\
(\text{resp.}
\;
\modu_{\Lambda^J}(A_{J,\zeta}^{(\ell)})/
(\modu_{\Lambda^J}(A_{J,\zeta}^{(\ell)})
\cap
\Tor_{\Lambda^{J+}}(A_{J,1}))
)
\end{gather*}
is naturally identified with the category
$\Mod(\GO_J)$
(resp.\
$\modu(\GO_J)$)
of quasi-coherent
(resp.\ coherent)
$\GO_J$-modules.

Let 
\[
\overline{\omega}_{J}^*:
\Mod_{\Lambda^{J}}(A_{J,\zeta}^{(\ell)})
\to
\Mod(\GO_J)
\]
be the canonical functor, and let
\[
\overline{\omega}_{J*}:
\Mod(\GO_J)
\to
\Mod_{\Lambda^{J}}(A_{J,\zeta}^{(\ell)})
\]
be its right adjoint.

We write 
$\GO=\GO_\emptyset\in\modu(\CO_\CB)$ and
\[
\overline{\omega}^*=\overline{\omega}_\emptyset^*:
\Mod_{\Lambda}(A_{\zeta}^{(\ell)})
\to
\Mod(\GO),
\qquad
\overline{\omega}_{*}:=\overline{\omega}_{\emptyset*}:
\Mod(\GO)
\to
\Mod_{\Lambda^{J}}(A_{\zeta}^{(\ell)})
\]

Let
\[
(\bullet)^{(\ell)}:
\Mod_{\Lambda^J}(A_{J,\zeta})
\to
\Mod_{\Lambda^J}(A_{J,\zeta}^{(\ell)})
\]
be the functor given by
$M\mapsto
\bigoplus_{\lambda\in\Lambda^J}M(\ell\lambda)
$, and define a functor
\begin{equation}
\label{eq:Fr}
\Fr_*:
\Mod(\CO_{\CP_{J,\zeta}})
\to
\Mod(\GO_J)
\end{equation}
by the commutative diagram:
%
\[
\xymatrix@C=40pt@R=30pt{
\Mod_{\Lambda^{J}}(A_{J,\zeta})
\ar[r]^{(\bullet)^{(\ell)}}
\ar[d]_{\omega_J^*}
&
\Mod_{\Lambda^{J}}(A_{J,\zeta}^{(\ell)})
\ar[d]^{\overline{\omega}_J^*}
\\
\Mod(\CO_{\CP_{J,\zeta}})
\ar[r]_{\Fr_*}^\cong
&
\Mod(\GO_J).
}
\]

Similarly to \cite[Lemma 3.9]{T1}
we have the following.
\begin{lemma}
\label{lem:Fr-dir}
\begin{itemize}
\item[(i)]
The functor 
\eqref{eq:Fr} gives category equivalences:
\[
\Mod(\CO_{\CP_{J,\zeta}})
\cong
\Mod(\GO_J),
\qquad
\modu(\CO_{\CP_{J,\zeta}})
\cong
\modu(\GO_J).
\]
\item[(ii)]
We have the following commutative diagram:
%
\[
\xymatrix@C=30pt@R=25pt{
\Mod(\CO_{\CP_{J,\zeta}})
\ar[r]^{\omega_{J*}}
\ar[d]_{\Fr_*}^\cong
&
\Mod_{\Lambda^{J}}(A_{J,\zeta})
\ar[d]^{(\bullet)^{(\ell)}}
\\
\Mod(\GO_J)
\ar[r]_{\overline{\omega}_{J*}}
&
\Mod_{\Lambda^{J}}(A_{J,\zeta}^{(\ell)}).
}
\]
\end{itemize}
\end{lemma}

For $\lambda\in\Lambda^J$ we set
\begin{equation}
\GO_J[\lambda]=\Fr_*(\omega_J^*(A_{J,\zeta}[\lambda]))
=\overline\omega_J^*((A_{J,\zeta}[\lambda])^{(\ell)})
\in\Mod(\GO_J).
\end{equation}
Then 
$\GO_J[\lambda]$ is naturally an $\GO_J$-bimodule.
Moreover, 
the two $\CO_{\CP_J}$-module structures of $\GO_J[\lambda]$ obtained by restricting the left and the right $\GO_J$-module structures coincide.
We have
\begin{equation}
\GO_J[\ell\lambda]\cong\CO_{\CP_J}[\lambda]\otimes_{\CO_{\CP_J}}\GO_J
\qquad(\lambda\in\Lambda^J),
\end{equation}
\begin{equation}
\GO_J[\lambda]\otimes_{\GO_J}\GO_J[\mu]
\cong
\GO_J[\lambda+\mu]
\qquad(\lambda, \mu\in\Lambda).
\end{equation}
Define 
\begin{equation}
(\bullet)[\lambda]:
\Mod(\GO_J)\to\Mod(\GO_J)
\qquad(\GM\mapsto\GM[\lambda])
\end{equation}
to be the functor induced by \eqref{eq:shift2}.
Then we have
\begin{equation}
\GM[\lambda]\cong\GO_J[\lambda]\otimes_{\GO_J}\GM
\qquad(\lambda\in\Lambda^J, \GM\in\Mod(\GO_J)).
\end{equation}
The natural $U^L_\zeta(\Gg)$-module structure of $A_{J,\zeta}$ 
induces a canonical $\BC$-algebra homomorphism
\begin{equation}
\label{eq:UO}
U^L_\zeta(\Gg)\to\End_\BC(\GO_J[\lambda])
\qquad(\lambda\in\Lambda^J).
\end{equation}

Now assume $J_1\subset J_2\subset I$.
Recall that we have natural morphism
$\pi^{J_2J_1}:\CP_{J_1}\to\CP_{J_2}$ of ordinary algebraic varieties.
By Lemma \ref{lem:dir2} and Lemma \ref{lem:Fr-dir} we have the following.
\begin{lemma}
\label{lem:piFr}
The following diagram is commutative:
%
\[
\xymatrix@C=50pt@R=25pt{
\Mod(\CO_{\CP_{J_1,\zeta}})
\ar[r]^{({\pi}^{J_2J_1}_\zeta)_*}
\ar[d]_{\Fr_*}^\cong
&
\Mod(\CO_{\CP_{J_2,\zeta}})
\ar[d]^{\Fr_*}_\cong
\\
\Mod(\GO_{J_1})
\ar[r]_{({\pi}^{J_2J_1})_*}
&
\Mod(\GO_{J_2}).
}
\]
\end{lemma}
\begin{lemma}
$
(\pi^{J_2J_1})_*(\GO_{J_1})
=
\GO_{J_2}$.
\end{lemma}
\begin{proof}
For $J\subset I$ we have
\[
\GO_{J}
=\overline{\omega}_{J}^*A_{J,\zeta}^{(\ell)}
=\Fr_*(\omega_J^*A_{J,\zeta}).
\]
Hence it is sufficient to show 
$(\pi^{J_2J_1}_\zeta)_*(\omega_{J_1}^*A_{J_1,\zeta})=
\omega_{J_2}^*A_{J_2,\zeta}$ 
by Lemma \ref{lem:piFr}.
By Corollary \ref{cor:oo} and Lemma \ref{lem:dir2} we have
\[
(\pi^{J_2J_1}_\zeta)_*(\omega_{J_1}^*A_{J_1,\zeta})
=\omega_{J_2}^*F_\zeta^{J_2J_1}\omega_{J_1*}
\omega_{J_1}^*A_{J_1,\zeta}
=\omega_{J_2}^*F_\zeta^{J_2J_1}A_{J_1,\zeta}
=\omega_{J_2}^*A_{J_2,\zeta}.
\]
\end{proof}
Hence by
\begin{align*}
\Hom(\GO_{J_2},\GO_{J_2})
\cong&
\Hom(\GO_{J_2},
(\pi^{J_2J_1})_*(\GO_{J_1}))
\cong
\Hom((\pi^{J_2J_1})^{-1}\GO_{J_2},
\GO_{J_1})
\end{align*}
we obtain a homomorphism
\[
(\pi^{J_2J_1})^{-1}\GO_{J_2}
\to
\GO_{J_1}
\]
of sheaf of rings.
\begin{proposition}
The following diagram is commutative:
\[
\xymatrix@C=140pt@R=25pt{
\Mod(\CO_{\CP_{J_2,\zeta}})
\ar[r]^{({\pi}^{J_2J_1}_\zeta)^*}
\ar[d]_{\Fr_*}^\cong
&
\Mod(\CO_{\CP_{J_1,\zeta}})
\ar[d]^{\Fr_*}_\cong
\\
\Mod(\GO_{J_2})
\ar[r]_{\GO_{J_1}
\otimes_{(\pi^{J_2J_1})^{-1}\GO_{J_2}}(\pi^{J_2J_1})^{-1}(\bullet)}
&
\Mod(\GO_{J_1}).
}
\]
\end{proposition}
\begin{proof}
Since the vertical functors $\Fr_*$ are equivalences,
it is sufficient to show that the functor 
$\GO_{J_1}
\otimes_{(\pi^{J_2J_1})^{-1}\GO_{J_2}}(\pi^{J_2J_1})^{-1}(\bullet)$
is left adjoint to 
\[
(\pi^{J_2J_1})_*:
\Mod(\GO_{J_1})
\to
\Mod(\GO_{J_2}).
\]
For 
$M\in\Mod(\GO_{J_1})$, 
$N\in\Mod(\GO_{J_2})$
we have
\begin{align*}
&
\Hom_{\GO_{J_2}}
(N,(\pi^{J_2J_1})_*M)
\cong
\Hom_{(\pi^{J_2J_1})^{-1}\GO_{J_2}}
((\pi^{J_2J_1})^{-1}N,M)
\\
\cong&
\Hom_{\GO_{J_1}}
(\GO_{J_1}
\otimes_{(\pi^{J_2J_1})^{-1}\GO_{J_2}}(\pi^{J_2J_1})^{-1}N,M).
\end{align*}
\end{proof}

\subsection{$\GD$-modules}
We set
\begin{equation}
\GD=\Fr_*\DD_{\CB_\zeta}
:=
\overline{\omega}^*(D_\zeta^{(\ell)})
\in\Mod(\GO).
\end{equation}
It is a sheaf of $\BC$-algebras on $\CB$ equipped with
$\BC$-algebra homomorphisms
\begin{align}
\label{eq:GD1}
&U_\zeta(\Gg)\to \GD
&(u\mapsto\deru_u),
\\
\label{eq:GD2}
&
\GO\to \GD
&(\varphi\mapsto\ell_\varphi),
\\
\label{eq:GD3}
&
\BC[\Lambda]\to \GD
&(e(\mu)\mapsto\sigma_{2\mu}).
\end{align}
As a $\BC$-algebra we have
\[
\GD=
\langle
\deru_u,\,
\ell_\varphi,\,
\sigma_{2\mu}\mid
u\in U_\zeta(\Gg),\;
\varphi\in\GO,\;
\mu\in\Lambda
\rangle
\]
locally on $\CB$.
For $\lambda\in\Lambda$ we have a natural $\GD$-module structure of $\GO[\lambda]$.
For $\mu\in\Lambda$ we have
$\sigma_{2\mu}|_{\GO(\lambda)}=\zeta^{2(\mu,\lambda)}\id$.
For $\varphi\in\GO$
the action of $\ell_\varphi$ on $\GO[\lambda]$ is given by the left $\GO$-module structure of $\GO[\lambda]$.
For $u\in U_\zeta(\Gg)$ 
the action of $\deru_u$ 
on $\GO[\lambda]$
is given by the composite of 
$U_\zeta(\Gg)\to U^L_\zeta(\Gg)$ with \eqref{eq:UO}.

In $\GD$ we have
\begin{align}
\label{eq:Drel1}
&\deru_u\ell_\varphi=
\sum_{(u)}
\ell_{u_{(0)}\varphi}\deru_{u_{(1)}}
&(u\in U_\zeta(\Gg),\;\varphi\in\GO),
\\
\label{eq:Drel2}
&\ell_\varphi\sigma_{2\mu}=
\sigma_{2\mu}\ell_\varphi
&(\varphi\in\GO,\;\mu\in\Lambda),
\\
\label{eq:Drel3}
&\deru_u\sigma_{2\mu}=
\sigma_{2\mu}\deru_u
&(u\in U,\;\mu\in\Lambda).
\end{align}
In particular, $\sigma_{2\mu}$ belongs to the center of $\GD$.
It is also easily seen that $\deru_z$ for $z\in Z_\Fr(U_\zeta(\Gg))$ belongs to the center of $\GD$ (see the proof of \cite[Lemma 5.1]{T1}).
For $t\in  H$ we set
\begin{equation}
\GD_t:=
\GD/\sum_{\mu\in\Lambda}\GD(\sigma_{2\mu}-\theta_\mu(t)).
\end{equation}
We denote by $\modu_t(\GD)$ the category of coherent $\GD$-module $\GM$ such that 
for any $\mu\in\Lambda$ and 
any section $m$ of $\GM$ we have
$(\sigma_{2\mu}-\theta_\mu(t))^nm=0$ 
for some $n$.
We have
\begin{alignat}{2}
\label{eq:DB1}
\modu(\DD_{\CB_\zeta})&\cong
\modu(\GD),
\\
\label{eq:DB2}
\modu(\DD_{\CB_\zeta,t})&\cong
\modu(\GD_t),
\\
\label{eq:DB3}
\modu_t(\DD_{\CB_\zeta})&\cong
\modu_t(\GD).
\end{alignat}

For $J\subset I$
we set
\begin{equation}
\GD_J=\pi^{J}_*\GD.
\end{equation}
It is a sheaf of $\BC$-algebras on $\CP_J$
 equipped with
$\BC$-algebra homomorphisms
\begin{align}
&U_\zeta(\Gg)\to \GD_J
&(u\mapsto\deru_u),
\\
&
\GO_J\to \GD_J
&(\varphi\mapsto\ell_\varphi),
\\
&
\BC[\Lambda]\to \GD_J
&(e(\mu)\mapsto\sigma_{2\mu})
\end{align}
satisfying the relations similar to \eqref{eq:Drel1}, 
\eqref{eq:Drel2}, \eqref{eq:Drel3}.
For $t\in  H$ we set
\[
\GD_{J,t}:=
\GD_J/\sum_{\mu\in\Lambda}\GD_J(\sigma_{2\mu}-\theta_\mu(t)).
\]

\begin{theorem}
\label{thm:GammaD}
Let $t\in H$.
Assume that $\ell$ is a power of a prime number and that the order of $t^\ell$ is finite and prime to $\ell$.
Then we have
\[
R\Gamma(\CP_J,\GD_{J,t})\cong U_\zeta(\Gg)_{[t]}.
\]
\end{theorem}
\begin{remark}

We conjecture that Theorem \ref{thm:GammaD} holds without assuming that $\ell$ is a power of a prime 
(see Conjecture \ref{conj:BB2}).
In fact we can show Theorem \ref{thm:GammaD} without assuming that $\ell$ is a power of a prime in the case $G$ is of type $A$
(see \cite{TK}).
\end{remark}
We first recall the outline of the proof of
Theorem \ref{thm:GammaD} in the case $J=\emptyset$ given in \cite{T3}.

The adjoint action of $U_\BF(\Gg)$ on $U_\BF(\Gg)$ given by
\[
\ad(u)(v)=\sum_{(u)}u_{(0)}\cdot v\cdot Su_{(1)}
\qquad
(u,v\in U_\BF(\Gg))
\]
induces the adjoint action of $U^L_\zeta(\Gg)$ on $U_\zeta(\Gg)$.
We set 
\[
{}^fU_\zeta(\Gg)=
\{u\in
U_\zeta(\Gg)\mid
\dim \ad(U^L_\zeta(\Gg))(u)<\infty\}.
\]
It is a subalgebra of $U_\zeta(\Gg)$.
Define a subalgebra ${}^fD_\zeta$  of $D_\zeta$ 
by
\[
{}^fD_\zeta=
\langle 
\ell_\varphi, \deru_u, \sigma_{2\lambda}
\mid
\varphi\in A_\zeta, u\in {}^fU_\zeta(\Gg), \lambda\in\Lambda
\rangle.
\]
We also  set
\[
{}^fD_{\zeta,t}
=
{}^fD_\zeta/
\sum_{\lambda\in\Lambda}
{}^fD_\zeta(\sigma_{2\lambda}-\theta_\lambda(t)).
\]
Define objects
${}^f\DD_{\CB_\zeta}$ and ${}^f\DD_{\CB_\zeta,t}$ ($t\in H$) of 
$\Mod(\CO_{\CB_\zeta})$ by
\[
{}^f\DD_{\CB_\zeta}=\omega^*({}^fD_\zeta),
\qquad
{}^f\DD_{\CB_\zeta,t}=\omega^*({}^fD_{\zeta,t}).
\]
We also define objects ${}^f\GD$ and ${}^f\GD_t$ of $\Mod(\GO)$ by
\[
{}^f\GD=\Fr_*{}^f\DD_{\CB_\zeta},
\qquad
{}^f\GD_t=\Fr_*{}^f\DD_{\CB_\zeta,t}.
\]
Then ${}^f\GD$ is the subalgebra of $\GD$ given by
\[
{}^f\GD=
\langle 
\ell_\varphi, \deru_u, \sigma_{2\lambda}
\mid
\varphi\in\GO, u\in {}^fU_\zeta(\Gg), \lambda\in\Lambda
\rangle,
\]
and we have
\[
{}^f\GD_t=
{}^f\GD/
\sum_{\lambda\in\Lambda}{}^f\GD(\sigma_{2\lambda}-\theta_\lambda(t))
\qquad(t\in H).
\]
Moreover, we have
\[
\GD={}^f\GD\otimes_{{}^fU_\zeta(\Gg)}U_\zeta(\Gg),
\qquad
\GD_t={}^f\GD_t\otimes_{{}^fU_\zeta(\Gg)}U_\zeta(\Gg).
\]
Since $U_\zeta(\Gg)$ is a flat 
${}^fU_\zeta(\Gg)$-module (both as a left module and a right module), 
we have
\begin{align*}
R\Gamma(\CB,\GD)=&
R\Gamma(\CB,{}^f\GD)
\otimes_{{}^fU_\zeta(\Gg)}U_\zeta(\Gg)
=
R\Gamma(\CB_\zeta,{}^f\DD_{\CB_\zeta})
\otimes_{{}^fU_\zeta(\Gg)}U_\zeta(\Gg)
,
\\
R\Gamma(\CB,\GD_t)=&
R\Gamma(\CB,{}^f\GD_t)
\otimes_{{}^fU_\zeta(\Gg)}U_\zeta(\Gg)
=
R\Gamma(\CB_\zeta,{}^f\DD_{\CB_\zeta,t})
\otimes_{{}^fU_\zeta(\Gg)}U_\zeta(\Gg).
\qquad
\end{align*}
Here, $R\Gamma(\CB_\zeta,\bullet)$ is the derived functor of the global section functor
\[
\Gamma(\CB_\zeta,\bullet):
\Mod(\CO_{\CB_\zeta})\to\Mod(\BC)
\qquad(M\mapsto(\omega_*M)(0)).
\]

We denote by ${}^eU_\zeta(\Gg)$ the subalgebra of $U_\zeta(\Gg)$ generated by 
$k_{2\lambda}$ for $\lambda\in\Lambda$ and  $e_i$, $Sf_i$ for $i\in I$.
Set
\begin{equation}
V_\zeta
={}^eU_\zeta(\Gg)/
\sum_{i\in I}(Sf_i)\cdot
{}^eU_\zeta(\Gg).
\end{equation}
The adjoint action of $U^L_\zeta(\Gg)$ on $U_\zeta(\Gg)$ induces a  $U^L_\zeta(\Gb^-)$-module structure of  $V_\zeta$ so that $V_\zeta\in\Mod_\inte(U^L_\zeta(\Gb^-))$.
We also regard $V_\zeta$ as an $\CO(H)$-module by 
\[
\overline{u}\theta_\lambda=\overline{uk_{2\lambda}}
\qquad
(u\in {}^eU_\zeta(\Gg),\; \lambda\in\Lambda).
\]
Then $V_\zeta$ is a $(U^L_\zeta(\Gb^-),\CO(H))$-bimodule.
For $t\in H$ we set 
\[
V_{\zeta,t}=V_\zeta\otimes_{\CO(H)}\BC_t,
\]
where $\BC_t$ is the one-dimensional $\CO(H)$-module given by
$t:\CO(H)\to\BC$.
By \cite{T2}
${}^f\DD_{\CB_\zeta}$ and ${}^f\DD_{\CB_\zeta,t}$ are objects of 
$\Mod^{\eq}(\CO_{\CB_\zeta})$, and the corresponding objects of 
$\Mod_\inte(U^L_\zeta(\Gb^-))$ are $V_\zeta$ and $V_{\zeta,t}$ 
respectively (see \eqref{eq:equivariant}).
We have the natural morphisms
\begin{align}
\label{eq:conj-Ind1}
{}^fU_\zeta(\Gg)\otimes_{Z_\Har(U_\zeta(\Gg))}\CO(H)
\to
R\Ind^{I\emptyset}(V_\zeta),
\\
\label{eq:conj-Ind2}
{}^fU_\zeta(\Gg)\otimes_{Z_\Har(U_\zeta(\Gg))}\BC_{[t]}
\to
R\Ind^{I\emptyset}(V_{\zeta,t}),
\end{align}
where, $\BC_{[t]}$ is the one-dimensional $Z_\Har(U_\zeta(\Gg))$-module given by $\xi_{[t]}$.
We see by Lemma \ref{lem:IndGamma} that 
if \eqref{eq:conj-Ind1} is an isomorphism, 
we have
\begin{equation}
R\Gamma(\CB,\GD)\cong
U_\zeta(\Gg)\otimes_{Z_\Har(U_\zeta(\Gg))}\CO(H), 
\end{equation}
and if \eqref{eq:conj-Ind2} is an isomorphism, 
we have
\begin{equation}
R\Gamma(\CB,\GD_t)\cong
U_\zeta(\Gg)_{[t]}.
\end{equation}

The quantized coordinate algebra $\CO_\zeta(G)$ is denoted as
$\CO_\zeta(G)_\ad$ when it is regarded as a left $U^L_\zeta(\Gg)$-module via the adjoint action:
\[
\ad(u)(\varphi)
=\sum_{(u)}u_{(0)}\cdot\varphi\cdot(S^{-1}u_{(1)})
\qquad
(u\in U^L_\zeta(\Gg),\; \varphi\in \CO_\zeta(G)).
\]
A similar convention is also applied to $\CO_\zeta(B^-)$ and $\CO_\zeta(P_J^-)$.
One can construct using the Drinfeld paring the natural isomorphisms
\begin{equation}
\label{eq:Disom}
V_\zeta \cong \CO_\zeta(B^-)_\ad,
\qquad
{}^fU_\zeta(\Gg)\cong \CO_\zeta(G)_\ad
\end{equation}
in $\Mod_\inte(U^L_\zeta(\Gb^-))$ and 
$\Mod_\inte(U^L_\zeta(\Gg))$ respectively
(see \cite{Cal}, \cite{T3}).
Then under the identification \eqref{eq:Disom} the morphisms 
\eqref{eq:conj-Ind1}, \eqref{eq:conj-Ind2} become canonical morphisms
\begin{align}
\label{eq:conj-Ind1a}
\CO_\zeta(G)_\ad\otimes_{\CO(H/W\circ)}\CO(H)
\to
R\Ind^{I\emptyset}(\CO_\zeta(B^-)_\ad),
\\
\label{eq:conj-Ind2a}
\CO_\zeta(G)_\ad\otimes_{\CO(H/W\circ)}\BC_{[t]}
\to
R\Ind^{I\emptyset}(\CO_\zeta(B^-)_{\ad,t}),
\end{align}
where $\CO_\zeta(B^-)_{\ad,t}=\CO_\zeta(B^-)_{\ad}\otimes_{\CO(H)}\BC_t$.

Denote by $\CO(H)_t$ the localization of $\CO(H)$ at the maximal ideal $\Ker(t:\CO(H)\to\BC)$.
In \cite{T3} we established the isomorphism
\begin{equation}
\label{eq:conj-Ind1b}
\CO_\zeta(G)_\ad\otimes_{\CO(H/W\circ)}\CO(H)_t
\cong
R\Ind^{I\emptyset}(\CO_\zeta(B^-)_\ad)
\otimes_{\CO(H)}\CO(H)_t
\end{equation}
using reduction to the case $q=1$
(the assumption on $\ell$ and $t\in H$ in Theorem \ref{thm:GammaD} is only used here).
By \cite[Proposition 4.5]{T3} this implies 
\begin{equation}
\label{eq:conj-Ind2b}
\CO_\zeta(G)_\ad\otimes_{\CO(H/W\circ)}\BC_{[t]}
\cong
R\Ind^{I\emptyset}(\CO_\zeta(B^-)_{\ad,t}).
\end{equation}
Hence Theorem \ref{thm:GammaD} for $J=\emptyset$ follows from the above argument.

Now let us give a proof of Theorem \ref{thm:GammaD} for general $J$.
By Lemma \ref{lem:piFr}, Lemma \ref{lem:IndGamma}
we have 
\begin{align*}
R\pi^{J}_*({}^f\GD)
\cong&
\Fr_*(R\pi^{J}_{\zeta*}({}^f\DD_{\CB_\zeta}))
\cong
\Fr_*\CF_J(R\Ind^{J\emptyset}(\CO_\zeta(B^-)_\ad)),
\end{align*}
and
\begin{align*}
R\pi^{J}_*({}^f\GD_t)
\cong&
\Fr_*(R\pi^{J}_{\zeta*}({}^f\DD_{\CB_\zeta,t}))
\cong
\Fr_*\CF_J(R\Ind^{J\emptyset}(\CO_\zeta(B^-)_{\ad,t})).
\end{align*}
Assume we could show 
\begin{equation}
\label{eq:IndJ-van}
R^i\Ind^{J\emptyset}(\CO_\zeta(B^-)_{\ad})\otimes_{\CO(H)}\CO(H)_t=0
\qquad(i\ne0).
\end{equation}
Then by \cite[Proposition 4.5]{T3} we have
\[
R\Ind^{J\emptyset}(\CO_\zeta(B^-)_{\ad,t})
\cong
\Ind^{J\emptyset}(\CO_\zeta(B^-)_{\ad})\otimes_{\CO(H)}\BC_t,
\]
and hence
\begin{align*}
R\pi^{J}_*({}^f\GD_t)
\cong&
\Fr_*\CF_J(\Ind^{J\emptyset}
(\CO_\zeta(B^-)_\ad))\otimes_{\CO(H)}\BC_t
\cong
\pi^{J}_*({}^f\GD)\otimes_{\CO(H)}\BC_t.
\end{align*}
Since $U_\zeta(\Gg)$ is a flat ${}^fU_\zeta(\Gg)$-module, we have
\begin{align*}
R\pi^{J}_*(\GD_t)
\cong&
R\pi^{J}_*({}^f\GD_t)\otimes_{{}^fU_\zeta(\Gg)}U_\zeta(\Gg)
\cong
(\pi^{J}_*({}^f\GD)\otimes_{\CO(H)}\BC_t)
\otimes_{{}^fU_\zeta(\Gg)}U_\zeta(\Gg)
\\
\cong&
\pi^{J}_*(\GD)\otimes_{\CO(H)}\BC_t
\cong\GD_{J,t}.
\end{align*}
Hence we obtain
\[
R\Gamma(\CP_J,\GD_{J,t})
\cong
R\Gamma(\CP_J,R\pi^{J}_*(\GD_t))
\cong
R\Gamma(\CB,\GD_t)
\cong
U_\zeta(\Gg)_{[t]}
\]
by Theorem \ref{thm:GammaD} for $J=\emptyset$.

It remains to show \eqref{eq:IndJ-van}.
Define 
$U^L_\zeta(\Gl_J)$, $U^L_\zeta(\Gb_J^-)$ and $O_\zeta(B_J^-)$ in a obvious way.
We can similarly define the induction functor 
\[
\Ind':\Mod_\inte(U^L_\zeta(\Gb_J^-))
\to
\Mod_\inte(U^L_\zeta(\Gl_J)).
\]
By a standard fact concerning induction functors we have
\[
\For(\Ind^{J\emptyset}(M))
\cong\Ind'(\For(M))
\]
for any $U^L_\zeta(\Gb^-)$-module $M$.
Here, $\For$ denotes the forgetful functor.
We can show that  there exists an integrable $U^L_\zeta(\Gp_J)$-module $N$ such that 
\[
\For(\CO_\zeta(B^-)_\ad)
\cong
\CO_\zeta(B_J^-)_\ad\otimes N
\]
in $\Mod_\inte(U^L_\zeta(\Gb_J^-))$.
Hence we obtain
\begin{align*}
&\For(R^i\Ind^{J\emptyset}(\CO_\zeta(B^-)_{\ad}))\otimes_{\CO(H)}\CO(H)_t
\\
\cong&
R^i\Ind'(\For(\CO_\zeta(B^-)_{\ad}))\otimes_{\CO(H)}\CO(H)_t
\\
\cong&
R^i\Ind'(\CO_\zeta(B_J^-)_\ad\otimes N)\otimes_{\CO(H)}\CO(H)_t
\\
\cong&
(R^i\Ind'(\CO_\zeta(B_J^-)_\ad)\otimes_{\CO(H)}\CO(H)_t)\otimes N.
\end{align*}
Here, the last isomorphism is a consequence of the tensor identity for the induction functor. Therefore, it is sufficient to show 
\[
R^i\Ind'(\CO_\zeta(B_J^-)_\ad)\otimes_{\CO(H)}\CO(H)_t
\qquad(i\ne0).
\]
This is proved similarly to \eqref{eq:conj-Ind2b} 
by replacing $G$ and $B^-$ with $L_J$ and $B_J^-$ respectively.
The proof of Theorem \ref{thm:GammaD} is complete.

\subsection{Localization on $\CV$}
\label{subsec:locV}
We set
\begin{equation}
\CY=K\times_{H/W} H
\end{equation}
with respect to $\varkappa:K\to H/W$ and $H\to H/W$ ($t\mapsto [t^\ell]$).
We also set
\begin{equation}
\CV=
\{(B^-g, k,t)\in\CB\times K\times H
\mid
g\eta(k)g^{-1}\in t^{\ell}N^-\},
\end{equation}
and consider the projections
\begin{equation}
p^\CV_\CB:\CV\to \CB
\qquad
((B^-g, k,t)\mapsto B^-g),
\end{equation}
\begin{equation}
p^\CV_H:\CV\to H
\qquad
((B^-g, k,t)\mapsto t),
\end{equation}
\begin{equation}
p^\CV_{\CY}:\CV\to \CY
\qquad
((B^-g, k,t)\mapsto(k,t)).
\end{equation}
We define a subalgebra 
$\GZ$ of $\GD$ by
\[
\GZ=
\langle
\ell_\varphi,
\deru_u,
\sigma_{2\mu}
\mid
\varphi\in\CO_\CB,\;
u\in Z_\Fr(U_\zeta(\Gg)),\;
\mu\in \Lambda
\rangle
\subset
\GD.
\]
It is contained in the center of $\GD$.
By \cite{T1} we have
\begin{equation}
\label{eq:Z}
\GZ
\cong
(p^\CV_\CB)_{*}\CO_\CV.
\end{equation}
Note that $p^\CV_\CB$ is an affine morphism.
Hence localizing $\GD$ on $\CV$ we obtain an 
$\CO_{\CV}$-algebra 
$\sD$ 
such that
\begin{equation}
\Mod(\GD)
\cong
\Mod(\sD),
\qquad
\modu(\GD)
\cong
\modu(\sD).
\end{equation}
It is  easily seen that $\sD$ is a coherent $\CO_\CV$-module.
We have a natural $\BC$-algebra homomorphism
\begin{equation}
\label{eq:UtD}
U_\zeta(\Gg)\to\sD
\qquad(u\mapsto\deru_u).
\end{equation}

We define subvarieties 
$\CV_t$ for $t\in H$, and 
$\CV_t^k$ for $(k,t)\in \CY$ by
\begin{equation}
\CV_t=(p^\CV_H)^{-1}(t),
\qquad
\CV_t^k=
(p^\CV_{\CY})^{-1}(k,t).
\end{equation}
We denote by $\wCV_t$
(resp.\ $\wCV^k_t$) the formal neighborhood of $\CV_t$
(resp.\ $\CV^k_t$) in $\CV$.
\begin{remark}
Let $Y$ be a closed subvariety of a variety $X$.
In this paper we consider the formal neighborhood $\widehat{Y}$  of $Y$ in $X$ only when $Y$ is a base change of an affine closed embedding $Z'\to Z$ for a projective morphism $X\to Z$.
Hence following \cite{BM} we understand $\widehat{Y}$ to be the ordinary scheme $X\times_Z\Spec(\widehat{\CO(Z)})$, where
$\widehat{\CO(Z)}$ denotes the completion of $\CO(Z)$ at the ideal defining $Z'$
(see \cite[Section 0.1]{BM}).
\end{remark}

For $t\in H$ we define an $\CO_{\CV_t}$-algebra 
$\sD_t$ 
and an $\CO_{\wCV_t}$-algebra 
$\wsD_t$ 
by
\[
\sD_{t}
=
\sD\otimes_{\CO_\CV}\CO_{\CV_t},
\qquad
\wsD_{t}
=
\sD\otimes_{\CO_\CV}\CO_{\wCV_t}.
\]
Then \eqref{eq:UtD} induces $\BC$-algebra homomorphisms
\begin{equation}
\label{eq:UtD2}
U_\zeta(\Gg)_{[t]}\to\sD_t,
\qquad
\wU_\zeta(\Gg)_{[t]}\to\wsD_t.
\end{equation}
Since $\sigma_{2\mu}-\theta_\mu(t)$ for 
$\mu\in\Lambda$ generate the defining ideal of $\CV_t$ in $\CV$ we have
\begin{align}
\label{eq:DEt}
\modu(\GD_t)
\cong
\modu(\sD_{t}).
\end{align}
We denote by $\modu_t(\sD)$ the category of coherent $\sD$-modules supported on $\CV_t$.
Then we have
\begin{equation}
\modu_t(\GD)\cong\modu_t(\sD).
\end{equation}

For $(k,t)\in\CY$ we set
\[
\sD^k_{t}
=
\sD\otimes_{\CO_\CV}\CO_{\CV^k_t},
\qquad
\wsD^k_{t}
=
\sD\otimes_{\CO_\CV}\CO_{\wCV^k_t}.
\]
They are 
an $\CO_{\CV^k_t}$-algebra
and
an $\CO_{\wCV^k_t}$-algebra respectively.
We denote by 
$\modu^k_t(\sD)$
the category of coherent $\sD$-modules
supported on $\CV^k_t$.
Then we have
\[
\modu(\sD^k_t)
\subset
\modu^k_t(\sD)
\subset
\modu(\wsD^k_t).
\]

Set
\begin{equation}
\CB_t^k=
\{B^-g\in\CB
\mid
g\eta(k)g^{-1}\in t^\ell N^-\}
\quad((k,t)\in \CY).
\end{equation}
Note that 
we have a natural isomorphism $\CV^k_t\cong\CB^k_t$ for $(k,t)\in \CY$.
Since the variety $\CB_t^k$ depends only on $(k,t^\ell)$, 
we have a natural identification
\begin{equation}
\label{eq:Vkt}
\CV_{td}^k\cong
\CB^k_t
\qquad((k,t)\in \CY,\;d\in H_\ell).
\end{equation}

Regard $\CB^k_t$ as a subvariety of 
\begin{equation}
\tilde{G}'=\{(B^-g,k')\in \CB\times K\mid
g\eta(k')g^{-1}\in B^-\}.
\end{equation}
Since the natural morphism $\CV\to \tilde{G}'$ is an \'{e}tale morphism, we have
\begin{equation}
\label{eq:wVkt}
\wCV_{td}^k\cong
\wCB^k_t
\qquad(d\in H_\ell),
\end{equation}
where
$\wCB^k_t$ is the formal neighborhood of 
$\CB^k_t$ in $G'$.
We will sometimes regard $\sD^k_{td}$ and $\wsD^k_{td}$ as an 
$\CO_{\CB^k_{t}}$-algebra and an $\CO_{\wCB^k_{t}}$-algebra  respectively.

\subsection{Localization on $\CV_J$}
\label{subsec:locVJ}
Let $J\subset I$.
We set
\begin{equation}
\CV_J=
\{(P_J^-g, k,t)\in\CP_J\times K\times H
\mid
g\eta(k)g^{-1}\in\St_J^{-1}([t^{\ell}])N_J^-\},
\end{equation}
where $\mathrm{St}_J:L_J\to H/W_J$ is the Steinberg map for $L_J$.
We have a natural morphism 
\[
\spi^{J}:\CV\to\CV_J
\qquad((B^-g,k,t)\mapsto(P_J^-g,k,t)).
\]

Define $a:B^-\to H$ and
$a_J:P_J^-\to L_J$ using the natural isomorphisms
$B^-/N^-\cong H$ and $P_J^-/N_J^-\cong L_J$ respectively.
Set
\begin{align*}
\tilde{G}=&\{(B^-g,x)\in\CB\times G\mid
gxg^{-1}\in B^-\},
\\
\tilde{G}_J=&
\{(P_J^-g,x)\in\CP_J\times G\mid
gxg^{-1}\in P_J^-\}.
\end{align*}
Then we have a natural morphism
\[
\varrho_J:
\tilde{G}\to\tilde{G}_J\times_{H/W_J}H
\qquad
((B^-g,x)\mapsto(P_J^-g,x,a(gxg^{-1}))),
\]
where $\tilde{G}_J\to H/W_J$ is given by 
$(P_J^-g,x)\mapsto \St_J(a_J(gxg^{-1}))$, and
$H\to H/W_J$ is given by $t\mapsto[t]=Wt$.
We have the following Cartesian diagram:
\begin{equation}
\label{eq:Cart}
\xymatrix@C=80pt{
\CV
\ar[r]^{\spi^{J}}
\ar[d]
&
\CV_J
\ar[d]
\\
\tilde{G}
\ar[r]_{\varrho_J}
&
\tilde{G}_J\times_{H/W_J}H,
}
\end{equation}
where the vertical arrows are induced by $\eta:K\to G$.
\begin{lemma}
\label{lem:spi}
We have
$
R\spi^{J}_*\CO_\CV=\CO_{\CV_J}
$.
\end{lemma}
\begin{proof}
In the Cartesian diagram \eqref{eq:Cart} the vertical arrows are smooth morphisms.
Hence it is sufficient to show 
\begin{equation}
\label{eq:rhoJ}
R\varrho_{J*}\CO_{\tilde{G}}=
\CO_{\tilde{G}_J\times_{H/W_J}H}.
\end{equation}

In the case $J=I$ $\varrho_J$ is given by
\[
\varrho_I:\tilde{G}\to G\times_{H/W}H
\qquad
((B^-g,x)\mapsto(x,a(gxg^{-1}))),
\]
and \eqref{eq:rhoJ} is well-known (see, for example, \cite[(8.13)]{T3}).

Let $J$ be arbitrary.
Note that $\CP_J$ is covered by open subsets of the form $U^{g_0}=P_J^-\backslash P_J^-N_J^+g_0$ for $g_0\in G$.
Let 
\[
\varrho_J^{g_0}:(p\circ\varrho_J)^{-1}(U^{g_0})\to
p^{-1}(U^{g_0})
\]
be the restriction 
of $\varrho_J$, where $p:\tilde{G}_J\times_{H/W_J}H\to \CP_J$ is the natural projection.
Then it is sufficient to show 
\[
R\varrho^{g_0}_{J*}\CO_{(p\circ\varrho_J)^{-1}(U^{g_0})}=
\CO_{p^{-1}(U^{g_0})}
\]
for any $g_0\in G$.
It is easily shown that we have a commutative diagram
\[
\xymatrix@C=80pt{
(p\circ\varrho_J)^{-1}(U^{g_0})
\ar[r]^{\varrho_J^{g_0}}
\ar[d]_{\cong}
&
p^{-1}(U^{g_0})
\ar[d]^{\cong}
\\
\tilde{L}_J\times (N_J^+\times N_J^-)
\ar[r]_{\varrho'_J\times\Id}
&
(L_J\times_{H/W_J}H)\times (N_J^+\times N_J^-),
}
\]
where
\[
\tilde{L}_J=\{(B_J^-g,x)\in B_J^-\backslash L_J\times L_J\mid
gxg^{-1}\in B_J^-\},
\]
and $\varrho'_J$ is defined similarly to $\varrho_I$.
Hence the assertion is a consequence of 
\[
R\varrho'_{J*}\CO_{\tilde{L}_J}=
\CO_{\tilde{L}_J\times_{H/W_J}H},
\]
which is shown similarly to \eqref{eq:rhoJ} for $J=I$ by replacing $G$ with $L_J$.
\end{proof}
Let
\[
p^{\CV_J}_{\CP_J}:{\CV_J}\to \CP_J, \quad
p^{\CV_J}_H:{\CV_J}\to H,\quad
p^{\CV_J}_{\CY}:{\CV_J}\to \CY
\]
be the projections.

We set
\[
\GZ_J=\pi^{J}_*\GZ.
\]
It is a central subalgebra of $\GD_J=\pi^{J}_*\GD$.
We have
\begin{equation}
\GZ_J
\cong
(p^{\CV_J}_{\CP_J})_{*}\CO_{\CV_J}.
\end{equation}
In fact 
\[
\GZ_J=\pi^{J}_*\GZ
\cong
\pi^{J}_*(p^{\CV}_{\CB})_{*}\CO_{\CV}
\cong
(p^{\CV_J}_{\CP_J})_{*}
\spi^{J}_*\CO_{\CV}
\cong
(p^{\CV_J}_{\CP_J})_{*}\CO_{\CV_J}
\]
by \eqref{eq:Z} and Lemma \ref{lem:spi}.
Hence 
localizing $\GD_J$ on $\CV_J$ we obtain an 
$\CO_{\CV_J}$-algebra $\sD_J$.
By definition we have
\[
\sD_J=\spi^{J}_*(\sD).
\]
We have a natural $\BC$-algebra homomorphism 
\begin{equation}
\label{eq:UtDJ}
U_\zeta(\Gg)\to\sD_J\qquad(u\mapsto \deru_u).
\end{equation}

Set
\[
\CV_{J,t}=(p^{\CV_J}_H)^{-1}(t)
\quad(t\in H),
\qquad
\CV_{J,t}^k=
(p^{\CV_J}_{\CY})^{-1}(k,t)
\qquad((k,t)\in \CY),
\]
and let 
$\wCV_{J,t}$ (resp.\ 
$\wCV_{J,t}^k$) be the formal neighbourhood of 
$\CV_{J,t}$ (resp.\ $\CV_{J,t}^k$) in $\CV_J$.
Define an $\CO_{\CV_{J,t}}$-algebra $\sD_{J,t}$,
an $\CO_{\wCV_{J,t}}$-algebra $\wsD_{J,t}$, 
an $\CO_{\CV^k_{J,t}}$-algebra $\sD^k_{J,t}$,
and an $\CO_{\wCV^k_{J,t}}$-algebra $\wsD^k_{J,t}$ 
by
\begin{gather*}
\sD_{J,t}
=
\sD_J\otimes_{\CO_{\CV_J}}\CO_{\CV_{J,t}},
\qquad
\wsD_{J,t}
=
\sD_J\otimes_{\CO_{\CV_J}}\CO_{\wCV_{J,t}}
\qquad(t\in H),
\\
\sD^k_{J,t}
=
\sD_J\otimes_{\CO_{\CV_J}}\CO_{\CV^k_{J,t}},
\qquad
\wsD^k_{J,t}
=
\sD_J\otimes_{\CO_{\CV_J}}\CO_{\wCV^k_{J,t}}
\qquad((k,t)\in\CY)
\end{gather*}
respectively.
By \eqref{eq:UtDJ} we have natural $\BC$-algebra homomorphisms
\begin{align}
\label{eq:UtD2J}
&U_\zeta(\Gg)_{[t]}\to\sD_{J,t},
\qquad
\wU_\zeta(\Gg)_{[t]}\to\wsD_{J,t},\qquad(t\in H),
\\
\label{eq:UtD2Ja}
&U_\zeta(\Gg)^k_{[t]}\to\sD_{J,t}^k,
\qquad
\wU_\zeta(\Gg)^k_{[t]}\to\wsD^k_{J,t},\qquad((k,t)\in \CY).
\end{align}

For $t\in H$ (resp.\ $(k,t)\in\CY$) 
we denote by $\modu_t(\sD_J)$ (resp.\ $\modu_t^k(\sD_J)$) 
the category of coherent $\sD_J$-modules
supported on $\CV_{J,t}$ (resp.\ $\CV_{J,t}^k$).
Then we have
\begin{gather*}
\modu(\sD_{J,t})\subset
\modu_t(\sD_{J})\subset
\modu(\wsD_{J,t})
\qquad(t\in H)
\\
\modu(\sD_{J,t}^k)\subset
\modu^k_t(\sD_{J})\subset
\modu(\wsD_{J,t}^k)
\qquad((k,t)\in\CY).
\end{gather*}

For $(k,t)\in \CY$ set
\begin{equation}
\CP_{J,t}^k=
\{P_J^-g\in\CP_J
\mid
g\eta(k)g^{-1}\in \St_J^{-1}(t^\ell)N_J^-\},
\end{equation}
and 
let $\wCP^k_{J,t}$ be the formal neighborhood of 
$\CP^k_{J,t}$ naturally embedded in the smooth subvariety
\begin{equation}
\tilde{G}'_J=\{(P^-g,k')\in \CP_J\times K\mid
g\eta(k')g^{-1}\in P_J^-\}
\end{equation}
of $\CP_J\times K$.

Then we have natural identifications
\begin{equation}
\label{eq:Vkt2}
\CV_{J,td}^k\cong
\CP^k_{J,t}, 
\qquad
\wCV_{J,td}^k\cong
\wCP^k_{J,t}
\qquad((k,t)\in \CY,\;d\in H_\ell).
\end{equation}
We will sometimes regard $\sD^k_{J,td}$ and $\wsD^k_{J,td}$ as an 
$\CO_{\CP^k_{J,t}}$-algebra and an $\CO_{\wCP^k_{J,t}}$-algebra  respectively.

\subsection{Azumaya property}
\label{subsec:Azumaya}
Let us recall the main results of \cite{T1}.
In \cite[Theorem 6.1] {T1} it is proved using a result of \cite{BrG} that $\sD$ is an Azumaya algebra of rank $\ell^{2|\Delta^+|}$ on $\CV$
if the following conditions are satisfied
\begin{itemize}
\item[(c1)]
$\ell$ is prime to 3 if $G$ is of type $F_4$, $E_6$, $E_7$, $E_8$;
\item[(c2)]
$\ell$ is prime to 5 if $G$ is of type $E_8$.
\end{itemize}
Namely, $\sD$ is a locally free $\CO_\CV$-module of rank $\ell^{2|\Delta^+|}$, and 
for any $v\in\CV$ the fiber $\sD(v)$ is isomorphic to the matrix algebra $M_{\ell^{|\Delta^+|}}(\BC)$.
Moreover, it is also shown that 
$\sD_t^k$ and $\wsD_t^k$ are split Azumaya algebras for any $(k,t)\in \CY$ 
under the same conditions (c1), (c2) (see \cite[Theorem 6.10, Remark 6.13]{T1}).
Namely, 
there exists a locally free $\CO_{\CV^k_t}$-module $\GK$ 
(resp.\ 
a locally free $\CO_{\wCV^k_t}$-module $\wGK$) 
of rank $\ell^{|\Delta^+|}$ such that 
$\sD_t^k\cong\CEnd_{\CO_{\CV^k_t}}(\GK)$
(resp.\
$\wsD_t^k\cong\CEnd_{\CO_{\wCV^k_t}}(\wGK)$).

We now discuss the choice of the splitting bundles $\GK$ 
and $\wGK$ of $\sD_t^k$ and $\wsD_t^k$ for $(k,t)\in\CY$ respectively.
Note that we have a natural algebra homomorphism
\begin{equation}
\label{eq:BrG}
\CO_\CV\otimes_{Z(U_\zeta(\Gg))}U_\zeta(\Gg)\to
\sD
\end{equation}
given by
$Z(U_\zeta(\Gg))\cong
\CO(\CX)\to \CO_\CV$
corresponding to
\[
\CV\to \CX
\qquad((B^-g, k, t)\mapsto (k,[t])).
\]
Set 
\[
\CX_\ur=K\times_{H/W}(H_\ur/W\circ)\subset \CX,
\quad
\CY_\ur=K\times_{H/W}H_\ur\subset \CY,
\quad
\CV_{\ur}=(p^{\CV}_H)^{-1}(H_\ur).
\]
They are open subsets of $\CX$, $\CY$ and $\CV$ respectively.
By \cite[Theorem 2.5, Theorem 4.9]{BrG} together with 
\cite[Theorem 4.3]{BrGoo} 
we have the following.
\begin{proposition}
\label{prop:BrG}
The algebra 
$\CO_{\CX_\ur}\otimes_{Z(U_\zeta(\Gg))}U_\zeta(\Gg)$ 
is an Azumaya algebra on $\CX_\ur$ of rank $\ell^{2|\Delta^+|}$.
\end{proposition}

Hence
the restriction
$\CO_{\CV_\ur}
\otimes_{Z(U_\zeta(\Gg))}
U_\zeta(\Gg)$
of $\CO_{\CV}
\otimes_{Z(U_\zeta(\Gg))}
U_\zeta(\Gg)$ to $\CV_\ur$ is also an Azumaya algebra of rank $\ell^{2|\Delta^+|}$.
Since $\sD$ is locally generated as an $\CO_\CV$-module by $\ell^{2|\Delta^+|}$-sections, wee see that \eqref{eq:BrG} is an isomorphism on $\CV_\ur$ (\cite[Section 6]{T1}).
Namely, we have
\begin{equation}
\label{eq:isom-ur}
\sD|_{\CV_{\ur}}\cong
\CO_{\CV_\ur}
\otimes_{Z(U_\zeta(\Gg))}
U_\zeta(\Gg).
\end{equation}

Let $(k,t)\in\CY_\ur$.
By \eqref{eq:isom-ur} we have
\begin{equation}
\sD^k_{{t}}\cong 
\CO_{\CV^k_{{t}}}\otimes_\BC U_\zeta(\Gg)^k_{[{t}]},
\qquad
\wsD^k_{{t}}\cong
\CO_{\wCV^k_{{t}}} 
\otimes_{\widehat{\CO}(\CX)^k_{[{t}]}}
\wU_\zeta(\Gg)^k_{[{t}]},
\end{equation}
where $\widehat{\CO}(\CX)^k_{[{t}]}$ denotes the completion of 
${\CO}(\CX)$ at the maximal ideal 
corresponding to the point $(k,[{t}])\in \CX_\ur$.
Since 
$U_\zeta(\Gg)^k_{[{t}]}$ is isomorphic to the matrix algebra $M_{\ell^{|\Delta^+|}}(\BC)$,  there exists a unique irreducible $U_\zeta(\Gg)^k_{[{t}]}$-module $K^k_{[{t}]}$.
It follows that 
$\sD^k_{{t}}$
is a split Azumaya algebra with 
splitting bundle
\[
\GK^k_{{t}}=
\CO_{\CB^k_{{t}}}
\otimes_\BC
K^k_{[{t}]},
\]
where we identify $\CV^k_t$ with $\CB^k_t$ via \eqref{eq:Vkt}.
Moreover, since 
$\widehat{\CO}(\CX)^k_{[{t}]}$ is a complete local ring, there exists a 
$\wU_\zeta(\Gg)^k_{[{t}]}$-module 
$\wK^k_{[{t}]}$ which is free over 
$\widehat{\CO}(\CX)^k_{[{t}]}$ 
such that 
$\BC\otimes_{\widehat{\CO}(\CX)^k_{[{t}]}}\wK^k_{[{t}]}
\cong
{K}^k_{[{t}]}$
and 
$\wU_\zeta(\Gg)^k_{[{t}]}\cong
\End_{\wU_\zeta(\Gg)^k_{[{t}]}}(\wK^k_{[{t}]})$.
Hence
$\wsD^k_{{t}}$
is a split Azumaya algebra with 
splitting bundle
\[
\wGK^k_{{t}}=
\CO_{\wCB^k_{{t}}}
\otimes_{\wCO(\CX)^k_{[{t}]}}
\wK^k_{[{t}]}.
\]

By \cite{T1} (see \cite[Proposition 6.12]{T1} and its proof) 
we have the following.
\begin{theorem}
\label{thm:Azumaya}
Let
$(k,\tilde{t})\in\CY_\ur$ and 
$\lambda\in\Lambda$.
Then
$\sD^k_{\tilde{t}t_\lambda}$
and 
$\wsD^k_{\tilde{t}t_\lambda}$
are  split Azumaya algebras with splitting bundles
\[
\GK^{k}_{\tilde{t}}[\lambda]=
\GO[\lambda]\otimes_{\GO}
\GK^{k}_{\tilde{t}},
\qquad
\wGK^{k}_{\tilde{t}}[\lambda]=
\GO[\lambda]\otimes_{\GO}
\wGK^{k}_{\tilde{t}}
\]
respectively,
where
$\CV^k_{\tilde{t}t_\lambda}$ and $\wCV^k_{\tilde{t}t_\lambda}$ 
are identified with $\CB^k_{\tilde{t}}$ and $\wCB^k_{\tilde{t}}$ 
respectively, and 
the $\GO$-module structures of $\GK^{k}_{\tilde{t}}$ 
and 
$\wGK^{k}_{\tilde{t}}$ are 
given by $\GO\to\sD$.
\end{theorem}

\begin{remark}
If (c1), (c2) are satisfied, then for any $t\in H$
there exists $\tilde{t}\in H_\ur$ and $d\in H_\ell$ 
such that $t=\tilde{t}d$.
Hence 
$\sD^k_{t}$
and 
$\wsD^k_{t}$
are split Azumaya algebras for any $(k,t)\in\CY$.
For $(k,t)\in\CY$ such that $\eta(k)$ is unipotent  
we can always find $\tilde{t}\in H_{\ur}$ and $\lambda\in\Lambda$ such that $t=\tilde{t}t_\lambda$ without assuming (c1), (c2)
(see the proof of
\cite[Lemma 6.6]{T1}).
Hence we do not need to assume (c1), (c2) in dealing with the case where
$\eta(k)$ is unipotent.
\end{remark}
\begin{remark}
The splitting bundle
$\GK^{k}_{\tilde{t}}[\lambda]$ 
for 
$\sD^k_{\tilde{t}t_\lambda}$ 
depends on the choice of $\lambda$.
We have 
$t_\lambda=t_{\lambda+\ell\nu}$ for 
$\lambda, \nu\in \Lambda$ and 
\[
\GK^{k}_{\tilde{t}}[\lambda+\ell\nu]
=
\GO[\ell\nu]
\otimes_{\GO}
\GK^{k}_{\tilde{t}}[\lambda]
\cong
\CO_{\CB}[\nu]\otimes_{\CO_{\CB}}
\GK^{k}_{\tilde{t}}[\lambda].
\]
\end{remark}
For $J\subset I$ 
set 
$
\CV_{J,\ur}=(p^{\CV_J}_H)^{-1}(H_\ur)
$,
and let
$\spi^{J}_\ur:\CV_{\ur}\to
\CV_{J,\ur}$ be the restriction
of $\spi^{J}$.

\begin{proposition}
\label{prop:Azumaya-ur}
\begin{itemize}
\item[(i)]
The restriction of 
$\sD_J$ to 
$\CV_{J,\ur}$ 
is an Azumaya algebra.
\item[(ii)]
Assume $(k,\tilde{t})\in \CY_\ur$.
Then 
$\sD_{J,\tilde{t}}^k$ and 
$\wsD_{J,\tilde{t}}^k$ are  split Azumaya algebras with 
splitting bundles
$\GK_{J,\tilde{t}}^k=
\CO_{\CP_{J,t}^k}\otimes_\BC K^{k}_{[\tilde{t}]}$ 
and 
$\wGK_{J,\tilde{t}}^k=
\CO_{\wCP_{J,t}^k}\otimes_{\wCO(\CX)^k_{[t]}} \wK^{k}_{[\tilde{t}]}$ 
respectively.
\end{itemize}
\end{proposition}
\begin{proof}
By \eqref{eq:isom-ur} and Lemma \ref{lem:spi} we have
\begin{align*}
\sD_J|_{\CV_{J,\ur}}
\cong&
(\spi^{J}_\ur)_*
(\sD|_{\CV_{\ur}})
\cong
(\spi^{J}_\ur)_*
(\CO_{\CV_\ur}
\otimes_{Z(U_\zeta(\Gg))}
U_\zeta(\Gg))
\\
\cong&
(\spi^{J}_\ur)_*\CO_{\CV_\ur}
\otimes_{Z(U_\zeta(\Gg))}
U_\zeta(\Gg)
\\
\cong&
\CO_{\CV_{J,\ur}}
\otimes_{Z(U_\zeta(\Gg))}
U_\zeta(\Gg),
\end{align*}
and hence the assertion is proved similarly to the case $J=I$ using Proposition \ref{prop:BrG}.
\end{proof}

Similarly to Theorem \ref{thm:Azumaya} we obtain 
from Proposition \ref{prop:Azumaya-ur} 
the following.
\begin{theorem}
\label{thm:Azumaya-J}
Let $(k,\tilde{t})\in \CY_\ur$ and let 
$\mu\in\Lambda^J$.
Then 
$\sD_{J,\tilde{t}t_\mu}^k$ and $\wsD_{J,\tilde{t}t_\mu}^k$
are split Azumaya algebras with 
splitting bundles
$\GK_{J,\tilde{t}}^k[\mu]
=\GO_J[\mu]\otimes_{\GO_J}\GK_{J,\tilde{t}}^k
$
and
$\wGK_{J,\tilde{t}}^k[\mu]
=\GO_J[\mu]\otimes_{\GO_J}\wGK_{J,\tilde{t}}^k
$
respectively.
\end{theorem}

\begin{theorem}
\label{thm:Morita}
For $(k,\tilde{t})\in \CY_\ur$
and $\mu\in \Lambda^J$
we have
\begin{align}
\label{eq:Morita1}
\modu(\sD^{k}_{J,\tilde{t}t_\mu})
\cong&
\modu(\CO_{\CP^{k}_{J,\tilde{t}}})
\qquad&
(\GK^{k}_{J,\tilde{t}}[\mu]
\otimes_{\CO_{\CP^{k}_{J,\tilde{t}}}}
\CS\longleftrightarrow\CS),
\\
\label{eq:Morita2}
\modu(\wsD^{k}_{J,\tilde{t}t_\mu})
\cong&
\modu(\CO_{\wCP^{k}_{J,\tilde{t}}})
\qquad&
(\wGK^{k}_{J,\tilde{t}}[\mu]
\otimes_{\CO_{\wCP^{k}_{J,\tilde{t}}}}
\CS\longleftrightarrow\CS),
\\
\label{eq:Morita3}
\modu^{k}_{\tilde{t}t_\mu}(\sD_J)
\cong&
\modu_{\CP^k_{J,\tilde{t}}}(\CO_{\widetilde{G}_J'}),
\end{align}
where 
$\modu_{\CP^k_{J,t}}(\CO_{\widetilde{G}_J'})$ denotes 
the category of coherent $\CO_{\widetilde{G}_J'}$-modules
supported on $\CP^k_{J,t}$.
\end{theorem}
\begin{proof}
The assertions \eqref{eq:Morita1} and \eqref{eq:Morita2} are consequences of Theorem \ref{thm:Azumaya-J} and the Morita correspondence.
Then  \eqref{eq:Morita3} follows from \eqref{eq:Morita2}.
\end{proof}

\section{{Equivalence of derived categories}}
\label{sec:ED}
\subsection{Choice of $\dot{k}\in K$}
As noted in Section \ref{sec:Rep}, 
$U_\zeta(\Gg)^k$ for $k\in K$ depends only on the conjugacy class of $G$ containing $\eta(k)$.
Hence in the representation theory of $U_\zeta(\Gg)$
we can select 
for each conjugacy class $O$ of $G$ any single $k$ satisfying $\eta(k)\in O$ and restrict ourselves to the investigation of $U_\zeta(\Gg)^k$-modules.
It is well-known that any conjugacy class of $G$ contains an element $g$ such that $g_s\in H$, $g_u\in N^+$ and there exists a maximal torus of the centralizer $Z_G(g)$ contained in $H$ (see \cite[22.3]{H}).
In the rest of this paper we fix 
$\dot{k}\in K$, $\dot{g}\in G$, 
$\dot{h}\in H$, $\dot{x}\in N^+$
such that
\begin{itemize}
\item[(k1)]
$\eta(\dot{k})=\dot{g}=\dot{h}\dot{x}=\dot{x}\dot{h}$,
\item[(k2)]
there exists a maximal torus of 
$Z_G(\dot{g})$ contained in $H$,
\end{itemize}
and 
restrict ourselves to the case where the Frobenius central character is given by $\xi^{\dot{k}}$.
Note that we can write $\dot{k}=(\dot{x}\dot{h}^{1/2},\dot{h}^{-1/2})$, where $\dot{h}^{1/2}\in H$ satisfies $(\dot{h}^{1/2})^2=\dot{h}$.
On the other hand
by the theory of parabolic induction 
(see \cite{DK2})
we are further reduced to the case where $\dot{h}$ is exceptional.
Here, a semisimple element $h$ of $G$ is called exceptional if 
the semisimple rank of $Z_{G}(h)$ is the same as that of $G$.
Hence we further assume
\begin{itemize}
\item[(k3)]
$\dot{h}$ is exceptional
\end{itemize}
in the following.
We set $\dot{G}=Z_G(\dot{h})$.
It is connected since $G$ is chosen to be simply-connected.
Hence $\dot{G}$ is a connected semisimple algebraic group with maximal torus $H$.
We denote by $\dot{\Delta}\subset\Delta$ the root system of $\dot{G}$.
We set $\dot{\Delta}^+=\Delta^+\cap \dot{\Delta}$, 
$\dot{Q}=\sum_{\alpha\in\dot{\Delta}}\BZ\alpha$.
We denote by $\dot{W}$ the Weyl group of $\dot{\Delta}$.
\subsection{Equivalences}
In \cite{T3} we have proved Conjecture \ref{conj:BB2}
in certain cases.
In particular, it is true 
when
the regular element $t\in H$ satisfies $t\in H(\dot{h})$ for $\dot{h}$ as above, and $\ell$ satisfies the conditions:
\begin{itemize}
\item[(k4)]
$\ell$ is prime to $|Q/\dot{Q}|$,
\item[(p)]
$\ell$ is a power of a prime number
\end{itemize}
in addition to the conditions (a1), (a2), (a3).
In the rest of this paper we also assume the conditions (k4), (p).

Recall  that $H(\dot{h})$ consists of elements of the form $\tilde{t}d$ for $\tilde{t}\in H_\ur(\dot{h})$ and $d\in H_\ell$.
We fix $\dot{t}\in H_\ur(\dot{h})$ in the following.

By \eqref{eq:UtDJ}, \eqref{eq:UtD2J} we have left 
exact functors
\begin{align}
\label{eq:GammaX1}
\Gamma:&
\modu(\sD_{J,\dot{t}d}),
\to
\modu(U_\zeta(\Gg)_{[\dot{t}d]}),
\\
\label{eq:GammaX2}
\Gamma:&
\modu_{\dot{t}d}(\sD_J)
\to
\modu_{[\dot{t}d]}(U_\zeta(\Gg)),
\\
\label{eq:GammaX3}
\Gamma:&
\modu(\wsD_{J,\dot{t}d})
\to
\modu(\wU_\zeta(\Gg)_{[\dot{t}d]}).
\end{align}
We have also right exact functors
\begin{align}
\label{eq:LocX1}
\Loc_{J,\dot{t}d}:&
\modu(U_\zeta(\Gg)_{[\dot{t}d]})
\to
\modu(\sD_{J,\dot{t}d}),
\\
\label{eq:LocX2}
{}_{\dot{t}d}\Loc_J:&
\modu_{[\dot{t}d]}(U_\zeta(\Gg))
\to
\modu_{\dot{t}d}(\sD_J),
\\
\label{eq:LocX3}
\widehat{\Loc}_{J,\dot{t}d}:&
\modu(\wU_\zeta(\Gg)_{[\dot{t}d]})
\to
\modu(\wsD_{J,\dot{t}d})
\end{align}
in the opposite directions.
Here, \eqref{eq:LocX1} and \eqref{eq:LocX3} are defined by
\[
{\Loc}_{J,\dot{t}d}(M)
=
\sD_{J,\dot{t}d}\otimes_{U_\zeta(\Gg)_{[\dot{t}d]}}M,
\qquad
\widehat{\Loc}_{J,\dot{t}d}(M)=
\wsD_{J,\dot{t}d}\otimes_{\wU_\zeta(\Gg)_{[\dot{t}d]}}M
\]
respectively.
For $M\in\modu_{[\dot{t}d]}(U_\zeta(\Gg))$ we have
\[
\sD_J\otimes_{U_\zeta(\Gg)}M
\in
\bigoplus_{t'\in W\circ(\dot{t}d)}
\modu_{t'}(\sD_J),
\]
and ${}_{\dot{t}d}\Loc_J(M)$ is the component of 
$\sD_J\otimes_{U_\zeta(\Gg)}M$ belonging to
$\modu_{\dot{t}d}(\sD_J)$.

We say that $t\in H$ is $J$-regular if it satisfies
\[
w\in W,\;\;
w\circ t=t\;\Longrightarrow\; w\in W_J.
\]
Note that $t$ is $\emptyset$-regular if and only if it is regular in the sense of \ref{subsec:Dmod}.
\begin{theorem}
\label{thm:BB1}
Let $d=t_\lambda$ for $\lambda\in\Lambda^J$, and assume that $\dot{t}d$ is $J$-regular.
Then the right derived functors of 
\eqref{eq:GammaX1}, 
\eqref{eq:GammaX2}, 
\eqref{eq:GammaX3}
and the left derived functors of 
\eqref{eq:LocX1}, 
\eqref{eq:LocX2}, 
\eqref{eq:LocX3}
induce 
equivalences
\begin{align}
\label{eq:BB1a}
&D^b(\modu(U_\zeta(\Gg)_{[\dot{t}d]}))
\cong
D^b(\modu(\sD_{J,\dot{t}d})),
\\
\label{eq:BB1b}
&D^b(\modu_{[\dot{t}d]}(U_\zeta(\Gg)))
\cong
D^b(\modu_{\dot{t}d}(\sD_J)),
\\
\label{eq:BB1c}
&D^b(\modu(\wU_\zeta(\Gg)_{[\dot{t}d]}))
\cong
D^b(\modu(\wsD_{J,\dot{t}d}))
\end{align}
of triangulated categories.
\end{theorem}
\begin{proof}
Assume $J=I$.
Then the first assertion \eqref{eq:BB1a} is already proved in \cite[Theorem 9.4]{T3} in view of 
\eqref{eq:DB2} and \eqref{eq:DEt}.
The main part in the argument of \cite{T3}, which is used to establish \eqref{eq:BB1a} for $J=I$, is to show 
$R\Gamma(\sD_t)\cong U_\zeta(\Gg)_{[t]}$ for $t\in H(\dot{h})$ (see Theorem \ref{thm:GammaD}).
The proof of  \eqref{eq:BB1c} for $J=I$ is similar to that of 
\eqref{eq:BB1a}.
Here, we need 
$R\Gamma(\wsD_t)\cong \wU_\zeta(\Gg)_{[t]}$ instead of  $R\Gamma(\sD_t)\cong U_\zeta(\Gg)_{[t]}$.
This also holds since we actually proved in \cite{T3}  a stronger result
\[
R\Gamma(\jmath_t^*(\sD))
\cong 
U_\zeta(\Gg)\otimes_{\CO(H/W\circ)}{\CO(H/W\circ)_{[t]}}.
\]
Here $\CO(H)_t$ and $\CO(H/W\circ)_{[t]}$ for $t\in H(\dot{h})$ denote the localizations of 
$\CO(H)$ and $\CO(H/W\circ)$
at the maximal ideals
corresponding to the points $t\in H$ and $[t]\in H/W\circ$ respectively, 
and 
\[
\jmath_t:\CV\times_H\Spec(\CO(H)_t)\to \CV
\]
is the natural morphism (see \cite[Theorem 8.1]{T3}).
The assertion \eqref{eq:BB1b} for $J=I$ follows  from \eqref{eq:BB1c}  for $J=I$ since both sides of 
\eqref{eq:BB1b} are subcategories of those of 
\eqref{eq:BB1c}.

The proof for general $J$ is similar.
Details are omitted.
\end{proof}

\begin{remark}
If we can show Theorem \ref{thm:GammaD} without assuming the condition (p), then Theorem \ref{thm:BB1} 
holds without assuming it.
If this is the case, the arguments in the rest of this paper also work without assuming (p).
\end{remark}
For $d\in H_\ell$ we denote by 
$\modu^{\dot{k}}(\sD_{J,\dot{t}d})$ 
(resp.\ 
$\modu^{\dot{k}}_{\dot{t}d}(\sD_J)$)
the category consisting of coherent $\sD_{J,\dot{t}d}$-modules 
(resp.\ 
$\sD_J$-modules)
supported on $\CV^k_{td}$.
Taking into account of the action of $Z_\Fr(U_\zeta(\Gg))$ we also obtain the following.
\begin{theorem}
\label{thm:BB3}
Let $d=t_\lambda$ for $\lambda\in\Lambda^J$, and assume that $\dot{t}d$ is $J$-regular.
Then we have equivalences
\begin{align}
\label{eq:BB3a}
&D^b(\modu^{\dot{k}}(U_\zeta(\Gg)_{[\dot{t}d]}))
\cong
D^b(\modu^{\dot{k}}(\sD_{J,\dot{t}d})),
\\
\label{eq:BB3b}
&D^b(\modu^{\dot{k}}_{[\dot{t}d]}(U_\zeta(\Gg)))
\cong
D^b(\modu^{\dot{k}}_{\dot{t}d}(\sD_J)),
\\
\label{eq:BB3c}
&D^b(\modu(\wU_\zeta(\Gg)^{\dot{k}}_{[\dot{t}d]}))
\cong
D^b(\modu(\wsD^{\dot{k}}_{J,\dot{t}d}))
\end{align}
of triangulated categories.
\end{theorem}

\section{Torus action}
\subsection{Torus $C$}
For $\lambda\in\Lambda$ we define $s_\lambda\in H$ by 
$\theta_\mu(s_\lambda)=\zeta^{(\lambda,\mu)}$ for $\mu\in\Lambda$.
Then we have $s_\lambda^2=t_\lambda$.
Let $\dot{s}\in H$ be the unique element satisfying 
$\dot{s}^2=\dot{t}$, $\dot{s}^\ell=\dot{h}^{1/2}$.

We set
\begin{equation}
\label{eq:C}
C=Z_H(\dot{g})^0=Z_H(\dot{x})^0.
\end{equation}
By our assumption (k2) $C$ is a maximal torus of 
$Z_G(\dot{g})=Z_{\dot{G}}(\dot{x})$.
Define a subset $\Gamma$ of $\Delta^+$ by 
\[
{\dot{x}}=\exp(\dot{a}),
\qquad
\dot{a}=\sum_{\beta\in\Gamma}\dot{a}_\beta,
\qquad
\dot{a}_\beta\in\Gg_\beta\setminus\{0\},
\]
and set
\begin{align*}
\Lambda'_C=&\Lambda\cap
(\sum_{\beta\in\Gamma}\BQ\beta),
\qquad
\Lambda_C=\Lambda/\Lambda'_C.
\end{align*}
We denote by $\lambda\mapsto\overline{\lambda}$ the canonical map $\Lambda\to \Lambda_C$.
For $\gamma=\overline{\lambda}\in\Lambda_C$ we 
set
$\theta^C_{\gamma}=\theta_\lambda|_C$.
Then the character group of the torus $C$ is identified with 
$\Lambda_C$ by the correspondence $\gamma\leftrightarrow\theta^C_\gamma$.
We set
\begin{equation}
\label{eq:Cell}
C_\ell=\{h\in C\mid h^\ell=1\},
\end{equation}
\begin{equation}
\label{eq:QC}
Q_C=\{\lambda \in Q
\mid
(\lambda,\Lambda'_C)=\{0\}\}.
\end{equation}

We see easily that $s_\lambda\in C_\ell$ for $\lambda\in Q_C$.

\begin{lemma}
The group homomorphism
$Q_C\to C_\ell\;(\lambda\mapsto s_\lambda)$ 
induces an isomorphism
\[
Q_C/\ell Q_C\cong C_\ell
\]
of abelian groups.
\end{lemma}
\begin{proof}
Set
$Q^\vee_C=\{\gamma\in Q^\vee
\mid
(\gamma,\Lambda'_C)=0\}$.
It
is the lattice dual to 
$\Lambda_C$ so that we have $Q^\vee_C/\ell Q^\vee_C\cong C_\ell$.
Hence it is sufficient to show that the natural map 
$Q_C/\ell Q_C\to Q^\vee_C/\ell Q^\vee_C$ 
induced by 
$Q_C\subset Q^\vee_C$ 
is bijective.
This follows from our assumption on $\ell$.
\end{proof}

\subsection{Graded modules}
Note that $U_\zeta(\Gg)$ is equipped with the natural weight space decomposition
$
U_\zeta(\Gg)
=\bigoplus_{\lambda\in Q}
U_\zeta(\Gg)_\lambda
$
such that
$U_\zeta(\Gh)\subset U_\zeta(\Gg)_0$, 
$e_i\in U_\zeta(\Gg)_{\alpha_i}$, 
$f_i\in U_\zeta(\Gg)_{-\alpha_i}$.
We define a $\Lambda_C$-grading 
\begin{equation}
\label{eq:Ugr}
U_\zeta(\Gg)=\bigoplus_{\gamma\in\Lambda_C}
U_\zeta(\Gg)\langle\gamma\rangle
\end{equation}
of the algebra $U_\zeta(\Gg)$ by
$
U_\zeta(\Gg)\langle\gamma\rangle
=\bigoplus_{\lambda\in Q, \overline{\lambda}=\gamma}
U_\zeta(\Gg)_\lambda$ 
for $\gamma\in\Lambda_C$, where the natural map $Q\to\Lambda_C$ is denoted by $\lambda\mapsto\overline{\lambda}$.
We say that a $U_\zeta(\Gg)$-module $M$ is 
$\Lambda_C$-graded if we are given a grading 
\begin{equation}
\label{eq:grz1}
M=\bigoplus_{\gamma\in\Lambda_C}
M\langle\gamma\rangle
\end{equation}
of $M$ such that
\begin{equation}
\label{eq:grz2}
U_\zeta(\Gg)\langle\gamma\rangle 
M\langle\gamma'\rangle
\subset M\langle\gamma+\gamma'\rangle
\qquad(\gamma, \gamma'\in\Lambda_C).
\end{equation}

\begin{lemma}
\label{lem:Uk-gr}
The $\Lambda_C$-grading \eqref{eq:Ugr} of $U_\zeta(\Gg)$ induces 
$\Lambda_C$-gradings of 
$U_\zeta(\Gg)^{\dot{k}}$ and 
$U_\zeta(\Gg)^{\dot{k}}_{[\dot{t}d]}$
for $d\in H_\ell$.
\end{lemma}
\begin{proof}
Note 
$U_\zeta(\Gg)^{\dot{k}}=
U_\zeta(\Gg)/\CI$, where $\CI$ is the ideal of $U_\zeta(\Gg)$ generated by
$e_\beta^{\ell}$, 
$Sf_\beta^\ell-c_\beta$ for $\beta\in\Delta^+$, and
$k_{\ell\lambda}-d_\lambda$
for $\lambda\in\Lambda$.
Here, $c_\beta\in\BC$, $d_\lambda\in \BC^\times$ are determined from $\dot{k}$.
Hence in order to verify the assertion for 
$U_\zeta(\Gg)^{\dot{k}}$
it is sufficient to show that if $c_\beta\ne0$ for $\beta\in\Delta^+$, we have 
$\beta\in\Ker(\Lambda\to\Lambda_C)$.
This is easily seen from the definition of $\xi^{\dot{k}}$.
The assertion for $U_\zeta(\Gg)^{\dot{k}}_{[\dot{t}d]}$ follows from 
$Z_\Har(U_\zeta(\Gg))\subset U_\zeta(\Gg)_0\subset U_\zeta(\Gg)\langle0\rangle$. 
\end{proof}

The $\Lambda_C$-gradings of $U=U_\zeta(\Gg)$, $U_\zeta(\Gg)^{\dot{k}}$, $U_\zeta(\Gg)^{\dot{k}}_{[\dot{t}d]}$ determine left $C$-actions on them given by
\[
s\cdot u=\theta_\gamma^C(s) u
\qquad(s\in C,\; \gamma\in\Lambda_C, \; u\in U\langle\gamma\rangle).
\]
Correspondingly we have coactions
$
U\to U\otimes\CO(C).
$

\begin{lemma}
\label{lem:M-gr0}
For $M\in\modu(U_\zeta(\Gg)^{\dot{k}})$
we have a simultaneous eigenspace decomposition
\[
M=\bigoplus_{\overline{\gamma}\in\Lambda_C/\ell\Lambda_C}
M\langle\langle\overline{\gamma}\rangle\rangle,
\]
\[
M\langle\langle\overline{\gamma}\rangle\rangle
=\{m\in M\mid
k_\nu m=\theta_\nu(\dot{s})\theta^C_\gamma(s_\nu)m
\;\;(\nu\in Q_C)
\}
\qquad(\gamma\in\Lambda_C)
\]
with respect to the action of 
$\{k_\nu\}_{\nu\in Q_C}$ on $M$.
\end{lemma}
\begin{proof}
Since 
$U_\zeta(\Gg)^{\dot{k}}$ is finite-dimensional, 
$M$ is finite-dimensional.
By $M\in\modu(U_\zeta(\Gg)^{\dot{k}})$ we have
$k_\nu^\ell-\theta_\nu(\dot{h}^{1/2})=0$ for any $\nu\in\Lambda$ as an operator on $M$.
In particular, the action of $k_\nu$ on $M$ is diagonalizable.
It follows that 
the action of 
$\{k_\nu\}_{\nu\in\Lambda}$  on $M$ is simultaneously diagonalizable since
$U_\zeta(\Gh)$ is commutative.
Hence we obtain
\[
M=\bigoplus_{h\in H}M(h),
\qquad
M(h)=
\{
m\in M\mid k_\nu m=\theta_\nu(h)m
\;\;(\nu\in\Lambda)
\}.
\]
For 
$m\in M(h)$ we have
\[
\theta_{\nu}(h^\ell)m
=
\theta_{\ell\nu}(h)m
=
k_{\ell\nu}m
=\theta_\nu(\dot{h}^{1/2})m.
\]
for any $\nu\in\Lambda$, and hence 
we have $h^\ell=\dot{h}^{1/2}$ if $M(h)\ne0$.
Therefore, by 
$
\{h\in H\mid h^\ell=\dot{h}^{1/2}\}
=
\{\dot{s}s_\lambda\mid \lambda\in Q\}
$ we obtain 
$
M=\bigoplus_{\lambda\in Q}M(\dot{s}s_\lambda)$, 
which gives the simultaneous eigenspace decomposition of $M$ with respect to the action of 
$\{k_\nu\}_{\nu\in \Lambda}$ on $M$.

We restrict this to the action of 
$\{k_\nu\}_{\nu\in Q_C}$.
Assume $\nu\in Q_C$.
The  we have 
\[
k_\nu|_{M(\dot{s}s_\lambda)}
=\theta_\nu(\dot{s})\theta_\nu(s_\lambda)\id
=\theta_\nu(\dot{s})\theta_\lambda(s_\nu)\id
=\theta_\nu(\dot{s})\theta^C_{\overline\lambda}(s_\nu)\id
\]
for $\lambda\in Q$.
Moreover, for $\lambda, \mu\in Q$ we have
$\theta_\nu(\dot{s})\theta^C_{\overline\lambda}(s_\nu)
=
\theta_\nu(\dot{s})\theta^C_{\overline{\mu}}(s_\nu)
$ for any $\nu\in Q_C$
if and only if 
$\theta^C_{\overline{\lambda-\mu}}|_{C_\ell}=1$.
This condition is equivalent to 
$
\lambda-\mu\in\Ker(Q\to\Lambda_C/\ell\Lambda_C)
$.
It remains to show 
\[
Q/
\Ker(Q\to\Lambda_C/\ell\Lambda_C)
\cong
\Lambda_C/\ell\Lambda_C.
\]
This follows from the surjectivity of 
the composite of 
$Q\hookrightarrow\Lambda\to\Lambda_C
\to\Lambda_C/\ell\Lambda_C$, which is a consequence of (a2).
\end{proof}

We denote by
$\modu(U_\zeta(\Gg)^{\dot{k}};C)$
the category of $\Lambda_C$-graded 
$U_\zeta(\Gg)$-modules 
$M$ contained in 
$\modu(U_\zeta(\Gg)^{\dot{k}})$
such that
\begin{equation}
\label{eq:grz3a}
\nu\in Q_C,\;
\gamma\in\Lambda_C,\;
m\in M\langle\gamma\rangle
\Longrightarrow\;
k_\nu m
=
\theta_\nu(\dot{s})\theta^C_\gamma(s_\nu)m.
\end{equation}
We also denote by
$\modu^{\dot{k}}(U_\zeta(\Gg);C)$
the category of $\Lambda_C$-graded 
$U_\zeta(\Gg)$-modules 
$M$ contained in 
$\modu^{\dot{k}}(U_\zeta(\Gg))$
such that
\begin{equation}
\label{eq:grz3b}
\nu\in Q_C,\;
\gamma\in\Lambda_C,\;
m\in M\langle\gamma\rangle
\Longrightarrow\;
\exists n \;\;\text{s.t.}\;\;
(k_\nu-\theta_\nu(\dot{s})\theta^C_\gamma(s_\nu))^nm=0.
\end{equation}

Let $d\in H_\ell$.
We denote by
$\modu^{\dot{k}}_{[\dot{t}d]}(U_\zeta(\Gg);C)$
the full subcategory of 
$\modu^{\dot{k}}(U_\zeta(\Gg);C)$
consisting of $M\in \modu^{\dot{k}}(U_\zeta(\Gg);C)$
which is contained in $\modu^{\dot{k}}_{[\dot{t}d]}(U_\zeta(\Gg))$
as a $U_\zeta(\Gg)$-module.
By Lemma \ref{lem:Uk-gr} the torus $C$ acts on $\wU_\zeta(\Gg)^{\dot{k}}_{[\dot{t}d]}$.
Correspondingly we have a coaction 
\[
\wU_\zeta(\Gg)^{\dot{k}}_{[\dot{t}d]}\to
\wU_\zeta(\Gg)^{\dot{k}}_{[\dot{t}d]}\hotimes
\CO(C)
\]
where
\[
\wU_\zeta(\Gg)^{\dot{k}}_{[\dot{t}d]}\hotimes
\CO(C)
=\varprojlim_{n}
\left(U_\zeta(\Gg)\otimes \CO(C)\right)/
((\Ker(\xi^{\dot{k}}_{\dot{t}d}))^nU_\zeta(\Gg)\otimes \CO(C)).
\]
For a $\wU_\zeta(\Gg)^{\dot{k}}_{[\dot{t}d]}$-module $M$ we set
\[
M\hotimes
\CO(C)
=\varprojlim_{n}
(M\otimes \CO(C))/
((\Ker(\xi^{\dot{k}}_{\dot{t}d}))^nM\otimes \CO(C)).
\]
We denote by $\modu(\wU_\zeta(\Gg)^{\dot{k}}_{[\dot{t}d]};C)$ the category consisting of finitely generated 
$\wU_\zeta(\Gg)^{\dot{k}}_{[\dot{t}d]}$-module $M$ endowed with a right coaction $M\to M\hotimes\CO(C)$ of $\CO(C)$ on $M$ such that 
for any $\nu\in Q_C$ and $n>0$ the linear endomorphism
$f_{\nu,n}\in\End_\BC(M/(\Ker(\xi^{\dot{k}}_{\dot{t}d}))^nM)$ induced by 
$f_\nu=\sigma(k_\nu)-\theta_\nu(\dot{s})\rho(s_\nu)\in\End_\BC(M)$ is nilpotent.
Here, $\rho(s)$ for $s\in C$  denotes the action of $s$ on $M$ induced by the coaction of  $\CO(C)$ on $M$, and $\sigma(u)$ for $u\in U_\zeta(\Gg)$ denotes the action of $u$ on $M$ coming from the $\wU_\zeta(\Gg)^{\dot{k}}_{[\dot{t}d]}$-module structure.

\subsection{Equivariant $\GD$-modules}
We have natural right actions of $C$ on $\CB$, $K$ and $\CV$ given by
\begin{align}
\label{eq:1ac}
&B^-g\cdot h=B^-gh
\qquad&(B^-g\in\CB),
\\
\label{eq:1ac1}
&(x_1,x_2)\cdot h=(h^{-1}x_1h,h^{-1}x_2h)
\qquad&((x_1,x_2)\in K),
\\
\label{eq:1ac2}
&(B^-g,k,t)\cdot h=
(B^-g\cdot h,k\cdot h,t)
\qquad&((B^-g,k,t)\in\CV)
\end{align}
for $h\in C$.
We will also use the following right actions of $C$ on the same spaces given by the $\ell$-th powers:
\begin{align}
\label{eq:ellac}
&B^-g\star h=B^-gh^{\ell}
\qquad&(B^-g\in\CB),
\\
\label{eq:ellac1}
&(x_1,x_2)\star h=(h^{-\ell}x_1h^\ell,h^{-\ell}x_2h^\ell)
\qquad&((x_1,x_2)\in K),
\\
\label{eq:ellac2}
&(B^-g,k,t)\star h=
(B^-g\star h,k\star h ,t)
\qquad&((B^-g,k,t)\in\CV)
\end{align}
for $h\in C$.

Let $A_\BF=
\bigoplus_{\lambda\in\Lambda}
(A_\BF)_\lambda$ 
be the weight space decomposition of $A_\BF$,  where 
$A_\BF$ is regarded as a $U_\BF(\Gh)$-module by 
restricting the left $U_\BF(\Gg)$-module structure.
We define a $\Lambda_C$-grading 
$
A_\BF=\bigoplus_{\gamma\in\Lambda_C}A_\BF\langle\gamma\rangle$ of $A_\BF$ by
$
A_\BF\langle\gamma\rangle
=
\bigoplus_{\lambda\in\Lambda, \overline{\lambda}=\gamma}
(A_\BF)_\lambda$.
It induces a $\Lambda_C$-grading 
$
A_\BA=\bigoplus_{\gamma\in\Lambda_C}A_\BA\langle\gamma\rangle$
of $A_\BA$ given
by $A_\BA\langle\gamma\rangle=
A_\BF\langle\gamma\rangle\cap A_\BA$.
Moreover,  we have a $\Lambda_C$-grading 
$
D_\BA=\bigoplus_{\gamma\in\Lambda_C}D_\BA\langle\gamma\rangle$
of $D_\BA$ given by
\[
D_\BA\langle\gamma\rangle
=\{P\in D_\BA\mid P(A_\BA\langle\gamma'\rangle)\subset A_\BA\langle\gamma+\gamma'\rangle
\;(\gamma'\in\Lambda_C)\}.
\]
By the specialization $q\mapsto\zeta$ we obtain 
$\Lambda_C$-gradings
\begin{equation}
\label{eq:gradingAD}
A_\zeta=\bigoplus_{\gamma\in\Lambda_C}A_\zeta\langle\gamma\rangle,
\qquad
D_\zeta=\bigoplus_{\gamma\in\Lambda_C}D_\zeta\langle\gamma\rangle
\end{equation}
of $A_\zeta$, $D_\zeta$ respectively.
We have
$\deru_u, \ell_\varphi\in D_\zeta\langle\gamma\rangle$ for 
$u\in U_\zeta(\Gg)\langle\gamma\rangle$,
$\varphi\in A_\zeta\langle\gamma\rangle$,
and 
$\sigma_{2\lambda}\in D_\zeta\langle0\rangle$ for 
$\lambda\in\Lambda$.
Note that the original $\Lambda$-gradings
$
A_\zeta=
\bigoplus_{\lambda\in\Lambda}
A_\zeta(\lambda)$,
$
D_\zeta=\bigoplus_{\lambda\in\Lambda}
D_\zeta(\lambda)$
are compatible with the new $\Lambda_C$-gradings in the sense 
\begin{align*}
A_\zeta=
\bigoplus_{\lambda\in\Lambda, \gamma\in\Lambda_C}
A_\zeta(\lambda)
\cap
A_\zeta\langle\gamma\rangle,
\qquad
D_\zeta=
\bigoplus_{\lambda\in\Lambda, \gamma\in\Lambda_C}
D_\zeta(\lambda)
\cap
D_\zeta\langle\gamma\rangle.
\end{align*}
From \eqref{eq:gradingAD} we obtain left actions of the torus $C$ on the $\BC$-algebras 
$A_\zeta$, $D_\zeta$
given by
\begin{align*}
&P\in D_\zeta\langle\gamma\rangle
\;\Longrightarrow\;
hP=\theta^C_\gamma(h)P
\qquad(h\in C),
\\
&\varphi\in A_\zeta\langle\gamma\rangle
\;\Longrightarrow\;
h\varphi=\theta^C_\gamma(h)\varphi
\qquad(h\in C).
\end{align*}
By restricting them to $A_\zeta^{(\ell)}$, $D_\zeta^{(\ell)}$ we obtain left actions of 
$C$ on $A_\zeta^{(\ell)}$, $D_\zeta^{(\ell)}$.
By further restricting it to $A_1$ through the embedding 
$A_1\hookrightarrow A_\zeta^{(\ell)}$
we obtain a left $C$-action on $A_1$.
Identifying $A_1$ with the projective coordinate algebra of the ordinary flag manifold 
$\CB=B^-\backslash G$ we obtain the corresponding right $C$-action on $\CB$.
This is given by 
\eqref{eq:ellac}.
Define 
\[
p:\CB\times C\to\CB,
\qquad
a:\CB\times C\to\CB
\]
by 
\[
p(x,h)=x,\qquad
a(x,h)=x\star h
\]
respectively.
\begin{proposition}
The left action of $C$ on $A_\zeta^{(\ell)}$
(resp.\ $D_\zeta^{(\ell)}$)
defined above gives a $C$-equivariant $\CO_\CB$-algebra structure of
\[
\GO=\omega^*A_\zeta^{(\ell)}
\qquad
(\text{resp.}\;
\GD=\omega^*D_\zeta^{(\ell)})
\]
with respect to the right $C$-action \eqref{eq:ellac}.
Namely, we are given a homomorphism
\[
F:a^*\GO\to p^*\GO
\qquad
(\text{resp.}\;
F:a^*\GD\to p^*\GD)
\]
of $\CO_{\CB\times C}$-algebras satisfying $F|_{\CB\times\{1\}}=\id$
and the cocycle condition.
\end{proposition}
Similarly, we have the following.
\begin{proposition}
The sheaves $\GD$, 
$\GD^{\dot{k}}$ and 
$\GD^{\dot{k}}_{\dot{t}d}$ ($d\in H_\ell$) are 
$C$-equivariant $\CO_\CB$-algebras
with respect to the right $C$-action \eqref{eq:ellac}.
\end{proposition}

We refer the reader to \cite{Kas} for basics about equivariant coherent sheaves.

The right action \eqref{eq:ellac2} of $C$ on $\CV$  induces the right actions of $C$ on $\CV^{\dot{k}}$ and $\CV^{\dot{k}}_{\dot{t}d}$ for $d\in H_\ell$;
\begin{align}
\label{eq:ellac2a}
(B^-g,\dot{k},t)\star h
=&
(B^-gh^\ell,\dot{k},t)
\qquad&(h\in C, (B^-g,\dot{k},t)\in \CV^{\dot{k}}),
\\
\label{eq:ellac2b}
(B^-g,\dot{k},\dot{t}d)\star h
=&
(B^-gh^\ell,\dot{k},\dot{t}d))
\qquad&(h\in C, (B^-g,\dot{k},\dot{t}d)
\in \CV^{\dot{k}}_{\dot{t}d}).
\end{align}
Moreover, the $C$-equivariant $\CO_\CB$-algebra structure of $\GD$ 
(resp.\ $\GD^{\dot{k}}$, $\GD^{\dot{k}}_{\dot{t}d}$)
is naturally lifted to  
the $C$-equivariant $\CO_{\CV}$-algebra 
(resp.\ $\CO_{\CV^{\dot{k}}}$-algebra, 
$\CO_{\CV^{\dot{k}}_{\dot{t}d}}$-algebra) 
structure of $\sD$ 
(resp.\ $\sD^{\dot{k}}$, 
$\sD^{\dot{k}}_{\dot{t}d}$)
with respect to \eqref{eq:ellac2}
(resp.\ \eqref{eq:ellac2a}, 
\eqref{eq:ellac2b}).

Assume that 
$\GM$ is a $C$-equivariant $\sD$-module.
The $\BC$-algebra homomorphism
$
U_\zeta(\Gg)
\to\sD
$ 
induced by \eqref{eq:GD1} gives an action 
\begin{equation}
\label{eq:deru}
U_\zeta(\Gg)\to\End_{\BC}(\GM)
\qquad(u\mapsto\deru_u)
\end{equation}
of $U_\zeta(\Gg)$ on $\GM$.
On the other hand  since the action of the subgroup $C_\ell$ of $C$ on $\CV$ in  \eqref{eq:ellac2}
is trivial,
we obtain  an action 
\begin{equation}
\label{eq:rho}
C_\ell\to\End_{\BC}(\GM)
\qquad(h\mapsto\rho_h)
\end{equation}
of $C_\ell$ on $\GM$.

We denote by 
$\modu(\sD^{\dot{k}}_{\dot{t}d};\star C)$
the category of 
$C$-equivariant 
coherent
$\sD^{\dot{k}}_{\dot{t}d}$-module $\GM$
such that 
for any $\nu\in Q_C$ and 
any section $m$ of $\GM$ 
we have
$\deru_{k_\nu}m=
\theta_\nu(\dot{s})
\rho_{s_\nu}m$
with respect to
\eqref{eq:deru} and 
\eqref{eq:rho}.

We also denote by 
$\modu^{\dot{k}}_{\dot{t}d}(\sD;\star C)$
the category of 
$C$-equivariant 
coherent
$\sD$-module $\GM$
contained  
in $\modu^{\dot{k}}_{\dot{t}d}(\sD)$
as a $\sD$-module 
and
for any $\nu\in Q_C$ and 
any section $m$ of $\GM$ 
there exists $n$ 
satisfying 
$(\deru_{k_\nu}-
\theta_\nu(\dot{s})
\rho_{s_\nu})^nm=0$
with respect to
\eqref{eq:deru} and 
\eqref{eq:rho}.

Note that  $\dot{t}d$ for $d\in H_\ell$ is regular if and only if $d$ is regular with respect to the action of $\dot{W}$ on $H_\ell$, i.e.
\begin{equation}
\label{eq:regularity}
w\in\dot{W},
\;wd=d
\;
\Longrightarrow
\;
w=1.
\end{equation}

If $d\in H_\ell$ is regular in the sense of \eqref{eq:regularity}, then we have 
\begin{align*}
D^b(\modu_{[{\dot{t}d}]}^{\dot{k}}(\sD))
\cong&
D^b(\modu_{[{\dot{t}d}]}^{\dot{k}}(U_\zeta(\Gg)))
\end{align*}
by Theorem \ref{thm:BB3}.

\begin{proposition}
\label{prop:BBC1}
Assume $d\in H_\ell$ satisfies \eqref{eq:regularity}.
Then \eqref{eq:BB3b}
induces the equivalence
\begin{align*}
D^b(\modu_{[{\dot{t}d}]}^{\dot{k}}(\sD;\star C))
\cong&
D^b(\modu_{[{\dot{t}d}]}^{\dot{k}}(U_\zeta(\Gg);C))
\end{align*}
of triangulated categories.
\end{proposition}
\begin{proof}
We can easily
show that 
the functors
\begin{align*}
\Gamma:
\modu_{\dot{t}d}^{\dot{k}}(\sD)
\to
\modu_{[{\dot{t}d}]}^{\dot{k}}(U_\zeta(\Gg)),
\\
{}_{\dot{t}d}^{\dot{k}}\Loc:
\modu_{[{\dot{t}d}]}^{\dot{k}}(U_\zeta(\Gg))
\to
\modu_{\dot{t}d}^{\dot{k}}(\sD)
\end{align*}
induce
\begin{align*}
\Gamma:
\modu_{\dot{t}d}^{\dot{k}}(\sD;\star C)
\to
\modu_{[{\dot{t}d}]}^{\dot{k}}(U_\zeta(\Gg)
;C),
\\
{}_{\dot{t}d}^{\dot{k}}\Loc:
\modu_{[{\dot{t}d}]}^{\dot{k}}(U_\zeta(\Gg)
;C)
\to
\modu_{\dot{t}d}^{\dot{k}}(\sD;\star C)
\end{align*}
respectively.
Hence the desired result follows from Proposition \ref{prop:D-cat} and 
Lemma \ref{lem:U-cat} below.
\end{proof}

Denote by
$\Mod(\sD;\star C)$
the category of $C$-equivariant quasi-coherent $\sD$-modules.
\begin{proposition}
\label{prop:D-cat}
The abelian category 
$\Mod(\sD;\star C)$
has enough injectives.
Moreover, for any injective object 
$\CI$ of $\Mod(\sD;\star C)$ we have 
\[
R^k\Gamma(\CI)=0\qquad(k>0).
\]
\end{proposition}
\begin{proof}
The first half is a consequence of the general theory of 
Grothendieck categories.
Let us show the latter half.
Since 
$\sD$ is a locally free 
$\CO_{\CV}$-module,
we can show that 
$\CI$ is injective as an object of 
$\Mod(\CO_{\CV};\star C)$ similarly to \cite[Lemma 3.3.6]{Kas}.
Hence the assertion follows from \cite[Lemma 3.3.9]{Kas}.
\end{proof}

Denote by
$\Mod(U_\zeta(\Gg);C)$
the category of $\Lambda_C$-graded $U_\zeta(\Gg)$-modules.
The following is a consequence of the standard arguments on the cohomologies of graded rings (see e.g. \cite{JanM}).
\begin{lemma}
\label{lem:U-cat}
The abelian category 
$\Mod(U_\zeta(\Gg);C)$
has enough projectives.
Moreover, any projective object of 
$\Mod(U_\zeta(\Gg);C)$ is projective as  an object of $\Mod(U_\zeta(\Gg))$.
\end{lemma}

We denote by  $(\CV^{\dot{k}}_{\dot{t}d}\times C)^\wedge$ the formal neighborhood of $\CV^{\dot{k}}_{\dot{t}d}\times C$ in $\CV\times C$.
We define 
\[
p:(\CV^{\dot{k}}_{\dot{t}d}\times C)^\wedge\to
\wCV^{\dot{k}}_{\dot{t}d},
\qquad
a:(\CV^{\dot{k}}_{\dot{t}d}\times C)^\wedge\to
\wCV^{\dot{k}}_{\dot{t}d}
\]
to be the morphisms induced by the projection 
$\CV\times C \to\CV$ and the right $C$-action 
\eqref{eq:ellac2b}
respectively.
Then $\wsD^{\dot{k}}_{\dot{t}d}$ is a $C$-equivariant 
$\CO_{\wCV^{\dot{k}}_{\dot{t}d}}$-algebra in the sense that we are given a homomorphism 
$F:a^*(\wsD^{\dot{k}}_{\dot{t}d})\to 
p^*(\wsD^{\dot{k}}_{\dot{t}d})$ satisfying $F|_{\wCV^{\dot{k}}_{\dot{t}d}\times \{1\}}=\id$ and the cocycle condition.

We denote by $\modu(\wsD^{\dot{k}}_{\dot{t}d};\star C)$ the category of $C$-equivariant coherent  $\wsD^{\dot{k}}_{\dot{t}d}$-modules.
Namely, its object is a coherent $\wsD^{\dot{k}}_{\dot{t}d}$-module $\GM$ equipped with an isomorphism
$F_\GM:a^*\GM\to p^*\GM$ of $p^*(\wsD^{\dot{k}}_{\dot{t}d})$-modules 
satisfying $F_\GM|_{\wCV^{\dot{k}}_{\dot{t}d}\times \{1\}}=\id$
and the cocycle condition.

Similarly to Proposition \ref{prop:BBC1} we have the following

\begin{proposition}
\label{prop:BBC2}
Assume $d\in H_\ell$ satisfies \eqref{eq:regularity}.
Then \eqref{eq:BB3c} 
induces the equivalence
\begin{align*}
D^b(\modu(\wsD_{\dot{t}d}^{\dot{k}};\star C))
\cong &
D^b(\modu(\wU_\zeta(\Gg)_{[{\dot{t}d}]}^{\dot{k}};C))
\end{align*}
of triangulated categories.
\end{proposition}

\subsection{Equivariant Morita equivalence}
We next give a $C$-equivariant analogue of the Morita equivalence \eqref{eq:Morita1}.
We will use the identification
\begin{equation}
\label{eq:VB2}
\CV^{\dot{k}}_{\dot{t}t_\lambda}
\cong
\CB^{\dot{k}}_{\dot{t}}
\qquad(\lambda\in\Lambda),
\end{equation}
and regard 
$\sD_{\dot{t}t_\lambda}^{\dot{k}}$
with an $\CO_{\CB^{\dot{k}}_{\dot{t}}}$-algebra
in the following.
We have the right $C$-action on 
$\CB^{\dot{k}}_{\dot{t}}$ given by 
\begin{align}
\label{eq:CB1}
B^-g\cdot h=B^-gh
\qquad(B^-g\in \CB^{\dot{k}}_{\dot{t}}, \; h\in C).
\end{align}
Under the identification \eqref{eq:VB2} this corresponds to 
the right $C$-action on
$\CV^{\dot{k}}_{\dot{t}t_\lambda}$
 induced by the 
right $C$-action \eqref{eq:1ac2}  on $\CV$
(without $\ell$-th power).
The right $C$-action on $\CB^{\dot{k}}_{\dot{t}}$ given by the $\ell$-th power is denoted by
\begin{equation}
\label{eq:CBell}
\CB^{\dot{k}}_{\dot{t}}\times C\to
\CB^{\dot{k}}_{\dot{t}},
\qquad((x,h)\mapsto x\star h:=x\cdot h^\ell).
\end{equation}

Let us describe the action of  $C$ on the splitting bundle  
$\GK^{\dot{k}}_{\dot{t}}[\lambda]$
of $\sD^{\dot{k}}_{\dot{t}t_\lambda}$ for $\lambda\in\Lambda$.
Recall 
$
\GK^{\dot{k}}_{\dot{t}}
=
\CO_{\CB^{\dot{k}}_{\dot{t}}}
\otimes_\BC
K^{\dot{k}}_{[\dot{t}]}
$, 
where $K^{\dot{k}}_{[\dot{t}]}$ is the unique irreducible module over $U_\zeta(\Gg)^{\dot{k}}_{[\dot{t}]}$.
By \cite{JanM} $K^{\dot{k}}_{[\dot{t}]}$ has a unique (up to shift of grading) $\Lambda_C$-grading so that 
$K_{[\dot{t}]}^{\dot{k}}
\in
\modu(U_\zeta(\Gg)^{\dot{k}}_{[\dot{t}]};C)$.
Hence there exists a unique (up to tensoring with one-dimensional $C$-module)  $C$-equivariant 
$\sD_{\dot{t}}^{\dot{k}}$-module structure of 
$\GK^{\dot{k}}_{\dot{t}}$
such that 
$
\GK^{\dot{k}}_{\dot{t}}
\in
\modu(\sD_{\dot{t}}^{\dot{k}};\star C)
$.
Not that for $\lambda\in\Lambda$ we have
\[
\sD_{\dot{t}t_\lambda}^{\dot{k}}
=
\GO[\lambda]\otimes_\GO
\sD_{\dot{t}}^{\dot{k}}
\otimes_\GO\GO[-\lambda],
\qquad
\GK_{\dot{t}}^{\dot{k}}[\lambda]
=
\GO[\lambda]
\otimes_\GO
\GK_{\dot{t}}^{\dot{k}}.
\]
Since $\GO$ is a $C$-equivariant $\CO_\CB$-algebra and since $\GO[\lambda]$ is a $C$-equivariant $\GO$-module, 
we have
$
\GK_{\dot{t}}^{\dot{k}}[\lambda]
\in
\modu(\sD_{\dot{t}t_\lambda}^{\dot{k}};\star C)$.

We denote by 
$\modu(\CO_{\CB_{\dot{t}}^{\dot{k}}};C)$
the category of $C$-equivariant coherent 
$\CO_{\CB_{\dot{t}}^{\dot{k}}}$-modules with respect to the right action \eqref{eq:CB1} of $C$ on $\CB_{\dot{t}}^{\dot{k}}$.
\begin{proposition}
\label{prop:MorC}
The Morita equivalence \eqref{eq:Morita1}
induces an equivalence
\begin{equation}
\label{eq:MorC}
\modu(\sD_{\dot{t}t_\lambda}^{\dot{k}};\star C)
\cong
\modu(\CO_{\CB_{\dot{t}}^{\dot{k}}};C)
\end{equation}
of categories.
\end{proposition}
\begin{proof}
Note that \eqref{eq:Morita1} is given by 
\begin{align*}
&H:\modu(\sD^{\dot{k}}_{\dot{t}t_\lambda})\to\modu(\CO_{\CB^{\dot{k}}_{\dot{t}}})
\qquad
&(H(\GM)=\HHom_{\sD^{\dot{k}}_{\dot{t}t_\lambda}}
(\GK^{\dot{k}}_{\dot{t}}[\lambda],\GM)),
\\
&E:\modu(\CO_{\CB^{\dot{k}}_{\dot{t}}})\to
\modu(\sD^{\dot{k}}_{\dot{t}t_\lambda})
\qquad
&(E(\CS)=\GK^{\dot{k}}_{\dot{t}}[\lambda]\otimes_{\CO_{\CB^{\dot{k}}_{\dot{t}}}}\CS).
\end{align*}

Assume $\GM\in
\modu(\sD_{\dot{t}t_\lambda}^{\dot{k}};\star C)$.
Then we have a natural $C$-equivariant 
$\CO_{\CB_{\dot{t}}^{\dot{k}}}$-module structure of 
$H(\GM)$
with respect to \eqref{eq:CBell}.
However, since the action of 
$C_\ell$ on
$H(\GM)$ is trivial, 
$H(\GM)$ is in fact a $C/C_\ell$-equivariant 
$\CO_{\CB_{\dot{t}}^{\dot{k}}}$-module.
Hence identifying $C/C_\ell$ with $C$ by $hC_\ell\leftrightarrow h^\ell$, we obtain 
$H(\GM)
\in
\modu(\CO_{\CB_{\dot{t}}^{\dot{k}}};C)$.

Assume 
$\CS\in\modu(\CO_{\CB_{\dot{t}}^{\dot{k}}};C)$.
By lifting the $C$-equivariant 
$\CO_{\CB_{\dot{t}}^{\dot{k}}}$-module structure of  $\CS$ via 
$C\ni h\mapsto h^\ell\in C$ we obtain a new 
$C$-equivariant 
$\CO_{\CB_{\dot{t}}^{\dot{k}}}$-module structure
of $\CS$ 
with respect to \eqref{eq:CBell}.
By this we obtain 
$E(\CS)
\in
\modu(\sD_{\dot{t}t_\lambda}^{\dot{k}};\star C)$.
\end{proof}

We next give a $C$-equivariant analogue of \eqref{eq:Morita2}.
We will use the identification
\begin{equation}
\label{eq:VBII}
\wCV^{\dot{k}}_{\dot{t}t_\lambda}
\cong
\wCB^{\dot{k}}_{\dot{t}}
\qquad(\lambda\in\Lambda)
\end{equation}
(see \eqref{eq:wVkt}), 
and regard 
$\wsD_{\dot{t}t_\lambda}^{\dot{k}}$
as
an $\CO_{\wCB^{\dot{k}}_{\dot{t}}}$-algebra
in the following.
The right $C$-action 
\[
(B^-g,\dot{k})\cdot h
=
(B^-gh,\dot{k}\cdot h)
\qquad
((B^-g,\dot{k})\in\tilde{G}', \; h\in C)
\]
on $\tilde{G}'$ induces a right $C$-action
\begin{equation}
\label{eq:CB2}
(\CB^{\dot{k}}_t\times C)^\wedge
\to
\wCB^{\dot{k}}_t
\end{equation}
on $\wCB^{\dot{k}}_t$, 
where
$(\CB^{\dot{k}}_t\times C)^\wedge$ is the formal neighborhood of $\CB^{\dot{k}}_t\times C$ in $\tilde{G}'\times C$.

We denote by 
$\modu(\CO_{\wCB_{\dot{t}}^{\dot{k}}};C)$
the category of $C$-equivariant coherent 
$\CO_{\wCB_{\dot{t}a}^{\dot{k}}}$-modules with respect to the right $C$-action \eqref{eq:CB2}
on 
$\wCB_{\dot{t}}^{\dot{k}}$.
Similarly to Proposition \ref{prop:MorC} we have the following.
\begin{proposition}
\label{prop:MorC2}
The Morita equivalence \eqref{eq:Morita2}
induces an equivalence
\begin{equation}
\label{eq:MorC2}
\modu(\wsD_{\dot{t}t_\lambda}^{\dot{k}};\star C)
\cong
\modu(\CO_{\wCB_{\dot{t}}^{\dot{k}}};C)
\end{equation}
of categories.
\end{proposition}

Denote by 
$\modu_{\CB_{\dot{t}}^{\dot{k}}}(\CO_{\widetilde{G}'};C)$
the category of $C$-equivariant coherent 
$\CO_{\widetilde{G}'}$-module
supported on 
$\CB_{\dot{t}}^{\dot{k}}$.
By \eqref{eq:Morita3} and Proposition \ref{prop:MorC2} we have the following.
\begin{equation}
\label{eq:MorC2a}
\modu_{\dot{t}t_\lambda}^{\dot{k}}(\sD;\star C)
\cong
\modu_{\CB_{\dot{t}}^{\dot{k}}}(\CO_{\widetilde{G}'};C).
\end{equation}

\section{{Equivalence of abelian categories}}
In the rest of this paper we further pose the condition  
\begin{itemize}
\item[(k4)]
$\dot{h}=1$,
\end{itemize}
which is stronger than (k3).
In particular, we have $\dot{\Delta}=\Delta$, $\dot{W}=W$, and 
$\dot{t}=t_{-\rho}$.
For $\lambda\in\Lambda$ we set $\dot{t}(\lambda)=\dot{t}t_\lambda=t_{\lambda-\rho}$.

\subsection{Alcoves}
Set 
\[
\Lambda_\BR=\BR\otimes_\BZ \Lambda.
\]
For $\lambda\in\Lambda$ we define $\tau_\lambda:
\Lambda_\BR\to
\Lambda_\BR$ by $\tau_\lambda(x)=x+\ell\lambda$ for $x\in
\Lambda_\BR$.
We denote by $W_a$ 
(resp.\ $\tilde{W}_a$) 
the group of affine transformations of $\Lambda_\BR$ generated by $W$ and $\tau_\lambda$ for $\lambda\in Q$
(resp.\ $\Lambda$).
It is called the affine Weyl group (resp.\ extended affine Weyl group).
The action of $W_a$ 
 is generated by the reflections $s_{\alpha^\vee, n}$ 
($\alpha\in\Delta^+$, $n\in\BZ$)
with respect to the hyperplanes ${\GH}_{\alpha^\vee,n}$ 
given by 
\[
{\GH}_{\alpha^\vee,n}
=\{\lambda\in{\Lambda_\BR}
\mid
\langle\alpha^\vee,\lambda\rangle=\ell n\}.
\]
We set $\Lambda_\BR^\reg:=\Lambda_\BR\setminus
\bigcup_{\alpha\in\Delta^+, n\in\BZ}
{\GH}_{\alpha^\vee,n}$.
Note that $d=t_\lambda$ for  $\lambda\in\Lambda$ is regular in the sense of \eqref{eq:regularity} (for $\dot{W}=W$) if and only if 
we have $\lambda\in\Lambda^\reg:=\Lambda\cap\Lambda_\BR^{\reg}$.
The connected components of 
$\Lambda_\BR^\reg$ are called alcoves.
The affine Weyl group $W_a$ acts simply transitively on the set of alcoves.
We call the alcove
\[
A_0=
\{
\mu\in\Lambda_\BR
\mid
\langle\alpha^\vee,\mu\rangle>0
\;(\alpha\in{\Delta}^+), 
\;\;
\langle{\theta}^\vee,\mu\rangle<\ell\}
\]
the fundamental alcove.
Here, $\theta^\vee$ is the highest coroot.

Set $I_a=I\sqcup\{0\}$, and define 
$\GH_i$ for $i\in I_a$ by
$\GH_i=\GH_{\alpha_i^\vee,0}$ ($i\in I$),  
$\GH_0=\GH_{\theta^\vee,1}$.
Then $A_0$ is surrounded by the hyperplanes 
$\GH_i$ ($i\in I_a$).
For $i\in I_a$ we denote by $F_0^i$ the codimension one face of $A_0$ contained in $\GH_i$.
We define $s_i$ ($i\in I_a$) to be the reflection with respect to $\GH_i$.
Then $W$ and $W_a$ are Coxeter groups with the canonical generator systems
$
S=\{s_i\mid i\in I\}$ 
and 
$S_a=\{s_i\mid i\in I_a\}$
respectively.
We denote by $\ell:W_a\to\BZ_{\geqq0}$ the length function.
Set $\Omega=\{w\in\tW_a\mid w(A_0)=A_0\}$.
It is a finite subgroup of $\tW_a$ isomorphic to
$\tW_a/W_a\cong \Lambda/Q$.
Note $\tW_a\cong \Omega\ltimes W_a$.
We extend the length function for $W_a$ to 
$\ell:\tW_a\to\BZ_{\geqq0}$ by setting
$
\ell(w\omega)=\ell(w)
$
for 
$w\in W_a$, $\omega\in\Omega$.

The following fact is crucially used in the theory of translation functors (see
Jantzen \cite[Part II, 7.7]{JanB}).
\begin{lemma}
\label{lem:6J}
Assume that 
$\lambda, \mu\in \Lambda$ are contained in the closure of the same alcove.
Define 
$\nu\in \Lambda^+$ by 
$
\{\nu\}=W(\mu-\lambda)\cap \Lambda^+
$.
Assume that $\xi\in\wt(\Delta_\zeta(\nu))$ satisfies
$
\lambda+\xi\in W_a\mu
$.
Then we have 
$\xi\in W\nu$.
Moreover, we have
\[
w_1\lambda=\lambda,
\qquad
w_1\mu=\lambda+\xi.
\]
for some $w_1\in W_a$.
\end{lemma}

We define the braid group $\tBB_a$ to be the group generated by the elements $b_w$ for $w\in \tW_a$ satisfying the fundamental relations 
\[
b_wb_{w'}=b_{ww'}
\qquad(w, w'\in \tW_a,\quad\ell(ww')=\ell(w)+\ell(w')).
\]
We denote by $\BB$
(resp.\ $\BB_a$) 
the subgroup of $\tBB_a$ generated by $b_w$ for $w\in W$ 
(resp.\ $W_a$).
Then $\BB$ and $\BB_a$ are naturally isomorphic to the braid groups associated to the Coxeter groups $W$ and $W_a$ respectively.
We denote by $\tBB_a^+$ 
(resp.\ $\BB^+_a$) 
the subsemigroup of 
$\tBB_a$ 
(resp.\ $\BB_a$)  generated by $b_w$ for $w\in W_a$ 
(resp.\ $\tW_a$).

\subsection{Exotic sheaves}
We recall the definition of the exotic sheaves.
Let 
\[
\varrho_G:\tilde{G}\to G
\]
be the obvious projection.
We have
\[
\varrho_G^{-1}(\dot{x})
\cong\
\CB^{\dot{x}}
:=
\{
B^-g\in\CB\mid
g\dot{x}g^{-1}\in B^-\}.
\]
We denote by $\wCB^{\dot{x}}$ the formal neighborhood of 
$\varrho_G^{-1}(\dot{x})$ in $\tilde{G}$.
Note that we have 
\[
\CB^{\dot{x}}=\CB^{\dot{k}}_{\dot{t}},
\qquad
\wCB^{\dot{x}}=\wCB^{\dot{k}}_{\dot{t}}
\]
in the notation of Section \ref{subsec:locV}.

Set 
\begin{equation}
\tilde{\Gg}
=
\{
(B^-g,a)\in\CB\times \Gg\mid
\Ad(g)(a)\in \Gb^-\},
\end{equation}
and define a morphism
\begin{equation}
\varrho_\Gg:\tilde{\Gg}\to \Gg
\end{equation}
to be the obvious projection.
The exponential map $\exp:\Gg\to G$
induces a morphism
\[
\widetilde{\exp}:\tilde{\Gg}\to\tilde{G}
\qquad
((B^-g,a)\mapsto(B^-g,\exp(a)))
\]
of complex manifolds.
This is not algebraic, but becomes algebraic after restricting it to the nilpotent cone.
Take $\dot{a}\in \Gn^+$ such that $\exp(\dot{a})=\dot{x}$.
Then we have 
\[
\varrho_\Gg^{-1}(\dot{a})
\cong\
\CB^{\dot{a}}
:=
\{
B^-g\in\CB\mid
\Ad(g)(\dot{a})\in \Gb^-\}.
\]
It is easily seen that $\widetilde{\exp}$ 
induces isomorphisms
\begin{equation}
\label{eq:Bax}
\CB^{\dot{a}}\cong\CB^{\dot{x}},
\qquad
\wCB^{\dot{a}}
\cong
\wCB^{\dot{x}}
\end{equation}
of (algebraic) schemes, 
where
$\wCB^{\dot{a}}$ is the formal neighborhood of 
$\CB^{\dot{a}}$ in $\tilde{\Gg}$.

Let us transfer the definition of exotic sheaves from 
$\wCB^{\dot{a}}$ 
to
$\wCB^{\dot{x}}$. 
Note 
\[
\tilde{G}\times_G\tilde{G}
=\{(B^-g_1,B^-g_2,x)
\in
\CB\times\CB\times G
\mid
g_kxg_k^{-1}\in B^-\;(k=1, 2)\}.
\]
For $i\in I$ we define a subvariety $\Gamma_i$ of 
$\tilde{G}\times_G\tilde{G}$ to be the closure of 
\[
\{
(B^-g,B^-\tilde{s}_ig,x)
\in
\tilde{G}\times_G\tilde{G}
\mid
gxg^{-1}\in B^-\cap \tilde{s}_i^{-1}B^-\tilde{s}_i
\},
\]
where 
$\tilde{s}_i\in N_G(H)\subset G$ 
is a representative of $s_i\in W\cong N_G(H)/H$.
Let $\varrho_{G,i}:\Gamma_i\to G$ be the natural morphism, and set $\Gamma_i^{\dot{x}}=\varrho_{G,i}^{-1}(\dot{x})$.
The projections $p_{i,r}:\Gamma_i\to \tilde{G}$ ($r=1, 2$)
induce projective morphisms
${p}^{\dot{x}}_{i,r}:
\wGamma_i^{\dot{x}}
\to\wCB^{\dot{x}}$
of schemes, where 
$\wGamma_i^{\dot{x}}$ denotes the formal neighborhood of 
$\Gamma_i^{\dot{x}}$ in
$\Gamma_i$.

By \cite[1.3.2]{BM} (see also \cite{Ri}) 
we have a weak action
\begin{equation}
\label{eq:Jb}
\BFJ_b:D^b(\modu(\CO_{{\wCB^{\dot{x}}}}))
\xrightarrow{\;\;\cong\;\;}
D^b(\modu(\CO_{{\wCB^{\dot{x}}}}))
\qquad(b\in\tBB_a)
\end{equation}
of $\tBB_a$ on 
$D^b(\modu(\CO_{\wCB^{\dot{x}}}))$ 
given by 
\begin{align*}
\BFJ_{b_{s_i}}(\CS)=&
R({p}^{\dot{x}}_{i,1})_*(({p}^{\dot{x}}_{i,2})^*\CS)
\qquad&(i\in I),
\\
\BFJ_{b_{\tau_\lambda}}(\CS)=&
\CO_{\CB}[\lambda]\otimes_{\CO_\CB}\CS
\qquad&(\lambda\in \Lambda^+).
\end{align*}
Let 
${\varrho}^{\dot{x}}:{\wCB^{\dot{x}}}\to\widehat{\{\dot{x}\}}$ be the natural morphisms induced by $\varrho_G$, 
where $\widehat{\{\dot{x}\}}$ is the formal neighborhood of $\{\dot{x}\}$ in $G$.
By \cite[1.5.1]{BM} there exists a $t$-structure
$(D^{\leqq0},D^{\geqq0})$ of the triangulated category
$D^b(\modu(\CO_{\wCB^{\dot{x}}}))$ 
given by 
\begin{align*}
{D}^{\leqq0}
=&
\{
\CS
\in
D^b(\modu(\CO_{\wCB^{\dot{x}}}))
\mid
R{\varrho}^{\dot{x}}_{*}({\BFJ}_{b}\CS)\in D^{\leqq0}(\modu(\CO_{\widehat{\{\dot{x}\}}}))
\;\;(b\in\tBB_a^+)
\},
\\
{D}^{\geqq0}
=&
\{
\CS
\in
D^b(\modu(\CO_{\wCB^{\dot{x}}}))
\mid
R{\varrho}^{\dot{x}}_{*}(({\BFJ}_b)^{-1}\CS)\in D^{\geqq0}(\modu(\CO_{\widehat{\{\dot{x}\}}}))
\;\;(b\in\tBB_a^+)
\}.
\end{align*}
We denote by 
$\modu^\ex(\CO_{{\wCB^{\dot{x}}}})$ the heart of this $t$-structure.
We call objects of 
$\modu^\ex(\CO_{{\wCB^{\dot{x}}}})$
exotic sheaves 
on ${\wCB^{\dot{x}}}$.
We define  
$\modu^\ex(\CO_{{\wCB^{\dot{x}}}};C)$
to be the abelian category consisting of objects of 
$D^b(\modu(\CO_{{\wCB^{\dot{x}}}};C))$
which are 
contained in 
$\modu^\ex(\CO_{{\wCB^{\dot{x}}}})$
if we forget the $C$-action.

Recall that for $\lambda\in\Lambda^\reg$ 
we have equivalences 
\begin{align}
\label{eq:de1b}
D^b(\modu(\wU_\zeta(\Gg)^{\dot{k}}
_{[\dot{t}(\lambda)]}))
\cong&
D^b(\modu(\CO_{{\wCB^{\dot{x}}}})),\\
\label{eq:de2b}
D^b(\modu(\wU_\zeta(\Gg)^{\dot{k}}
_{[\dot{t}(\lambda)]};C))
\cong&
D^b(\modu(
\CO_{{\wCB^{\dot{x}}}}
;C))
\end{align}
of triangulated categories
by Theorem \ref{thm:BB3}, Theorem \ref{thm:Morita}, Proposition \ref{prop:BBC2}, Proposition \ref{prop:MorC2}.

The following is the main result of this paper.
\begin{theorem}
\label{thm:ab}
Let $\lambda_0\in A_0\cap \Lambda$.
Then 
\eqref{eq:de1b}, \eqref{eq:de2b}  induce
equivalences
\begin{align}
\label{eq:ab1b}
&
\modu(\wU_\zeta(\Gg)^{\dot{k}}
_{[\dot{t}(\lambda_0)]})
\cong
\modu^\ex(\CO_{{\wCB^{\dot{x}}}}),
\\
&
\label{eq:ab2b}
\modu(\wU_\zeta(\Gg)^{\dot{k}}
_{[\dot{t}(\lambda_0)]};C)
\cong
\modu^\ex(\CO_{{\wCB^{\dot{x}}}};C)
\end{align}
of abelian categories.
\end{theorem}
Note that \eqref{eq:ab2b}
easily follows from \eqref{eq:ab1b}.
The aim of this section is to give a proof of \eqref{eq:ab1b}  following the argument of \cite{BM}.
The point is to show that the wall-crossing functor for $U_\zeta(\Gg)$-modules correspond to 
the action \eqref{eq:Jb} of $\tBB_a$ for coherent sheaves.

\subsection{$\GU$-modules}
We set
\[
\GU=\Fr_*\CU_{\CB_\zeta}
:=
\overline{\omega}^*((A_\zeta\otimes U_\zeta(\Gg))^{(\ell)})
\]
(see \eqref{eq:UBZ} for the notation).
It is a sheaf of $\BC$-algebras on $\CB$.
It contains $\GO$ and $U_\zeta(\Gg)$ as subalgebras, and 
\[
\GO\otimes U_\zeta(\Gg)\to \GU
\qquad(\varphi\otimes u\mapsto \varphi u)
\]
is an isomorphism of sheaves.
Moreover, we have
\[
{u}{\varphi}=
\sum_{(u)}
(\deru_{u_{(0)}}(\varphi)){u}_{(1)}
\qquad
(u\in U_\zeta(\Gg), \varphi\in \GO)
\]
in $\GU$.
Hence $\GU$ is the smash product the algebra $\GO$ and the Hopf algebra $U_\zeta(\Gg)$ acting on it.
We denote by $\GJ^{\dot{k}}$ the ideal of $\GU$ generated by $z-\xi^k(z)$ for $z\in Z_\Fr(U_\zeta(\Gg))$, and set
\[
\GU^{\dot{k}}=\GU/\GJ^{\dot{k}},
\qquad
\wGU^{\dot{k}}=\varprojlim_n\GU/(\GJ^{\dot{k}})^n.
\]
For $\lambda\in\Lambda$ we have natural algebra homomorphisms
\begin{equation}
\label{eq:UD}
\GU^{\dot{k}}\to 
\GD^{\dot{k}}_{\dot{t}(\lambda)},
\qquad
\wGU^{\dot{k}}\to 
\wGD^{\dot{k}}_{\dot{t}(\lambda)}
\end{equation}
induced by \eqref{eq:ED}.

By $\modu(\CU_{\CB_\zeta})\cong\modu(\GU)$ we obtain from \eqref{eq:FV2} 
an exact functor 
\begin{equation}
\label{eq:FV3}
V\otimes(\bullet):
\modu(\GU)
\to
\modu(\GU)
\end{equation}
for $V\in\modu_\inte(U_\zeta^L(\Gg))$.
This induces
\begin{equation}
\label{eq:FV4}
V\otimes(\bullet):
\modu(\GU^{\dot{k}})
\to
\modu(\GU^{\dot{k}}),
\qquad
V\otimes(\bullet):
\modu(\wGU^{\dot{k}})
\to
\modu(\wGU^{\dot{k}}).
\end{equation}

By Proposition \ref{prop:tensVA2} we have the following.
\begin{proposition}
\label{prop:tensV}
Let $V\in\modu_\inte(U^L_\zeta(\Gg))$ and let 
\begin{equation}
\label{eq:filt1}
V=V_1\supset V_2\supset\cdots\supset V_m\supset V_{m+1}=\{0\}
\end{equation}
be a $U^L_\zeta(\Gb^-)$-stable filtration 
of $V$ such that $V_j/V_{j+1}$ is a one-dimensional $U^L_\zeta(\Gb^-)$-module with character $\xi_j\in\Lambda$.

Let $\lambda\in\Lambda$.
Let $\GM\in\modu(\GD^{\dot{k}}_{\dot{t}(\lambda)})$ 
(resp.\ $\modu(\wGD^{\dot{k}}_{\dot{t}(\lambda)})$), and regard it as an object of $\modu(\GU^{\dot{k}})$ 
(resp.\ $\modu(\wGU^{\dot{k}})$)
via \eqref{eq:UD}. 
Then $V\otimes\GM\in \modu(\GU^{\dot{k}})$
(resp.\ $\modu(\wGU^{\dot{k}})$) 
has 
 a functorial filtration 
\[
V\otimes \GM=
\GM_1\supset \GM_2\supset\cdots\supset 
\GM_m\supset 
\GM_{m+1}=\{0\}
\]
such that 
$\GM_j/\GM_{j+1}
\cong
\GO[\xi_j]\otimes_\GO
\GM
\in
\modu(\GD^{\dot{k}}_{\dot{t}(\lambda+\xi_j)})$
(resp.\ $\modu(\wGD^{\dot{k}}_{\dot{t}(\lambda+\xi_j)})$) 
for $j=1,\dots, m$.
\end{proposition}

\subsection{Translation functor}
Recall that we have a direct sum decomposition 
\begin{equation}
\label{eq:pU}
\modu(\wU_\zeta(\Gg)^{\dot{k}})
=
\bigoplus_{\lambda\in\Lambda/W_a}
\modu(\wU_\zeta(\Gg)^{\dot{k}}_{[\dot{t}(\lambda)]})
\end{equation}
of abelian categories (see \eqref{eq:Hd2}).
For $\lambda\in\Lambda$ we denote by
\begin{equation}
p_\lambda:
\modu(\wU_\zeta(\Gg)^{\dot{k}})
\to
\modu(\wU_\zeta(\Gg)^{\dot{k}}_{[\dot{t}(\lambda)]})
\end{equation} 
the projection
with respect to \eqref{eq:pU}.

Assume that 
$\lambda, \mu\in \Lambda$
are contained in the closure of the same alcove.
In this case we define an exact functor
\begin{align}
T_{\mu\lambda}:
\modu(\wU_\zeta(\Gg)^{\dot{k}}_{[\dot{t}(\lambda)]})
\to
\modu(\wU_\zeta(\Gg)^{\dot{k}}_{[\dot{t}(\mu)]})
\end{align}
by 
\begin{align*}
T_{\mu\lambda}(M)
=
{p}_{\mu}(L_\zeta(\nu)\otimes M),
\end{align*}
where $\nu\in \Lambda^+$ is given by
$
\{\nu\}=W(\mu-\lambda)\cap \Lambda^+
$.
This functor is called the translation functor.

\begin{remark}
For $\nu\in\Lambda^+$ the irreducible module $L_\zeta(\nu)$ appears in the composition series of $\Delta_\zeta(\nu)$ and 
$\nabla_\zeta(\nu)$ with multiplicity one, and other composition factors have smaller highest weights.
Hence we can show using Corollary \ref{cor:8ZX} and  Lemma \ref{lem:6J} that
\begin{align*}
T_{\mu\lambda}(M)
\cong
p_{\mu}(\Delta_\zeta(\nu)\otimes M)
\cong
p_{\mu}(\nabla_\zeta(\nu)\otimes M)
\end{align*}
for $\lambda$, $\mu$, $\nu$ as above.
\end{remark}
\begin{remark}
For $\lambda, \mu\in\Lambda$ as above and $w\in\tilde{W}_a$ we have $T_{\mu\lambda}=T_{w\mu,w\lambda}$.
Hence it is harmless to assume $\lambda,\mu\in\overline{A}_0$ as we do later.
\end{remark}
The proof of the following result is very classic and only included for completeness.
\begin{lemma}
\label{lem:8adX}
Assume that 
$\lambda, \mu\in \Lambda$
are contained in the closure of the same alcove.
Then $T_{\lambda\mu}$
is left and right adjoint to 
$T_{\mu\lambda}$.
\end{lemma}
\begin{proof}
By symmetry it is sufficient to show that 
$T_{\lambda\mu}$ 
is right adjoint to $T_{\mu\lambda}$.
Take $\nu_1, \nu_2\in\Lambda^+$ satisfying 
$
\{\nu_1\}=W(\mu-\lambda)\cap \Lambda^+$,
$
\{\nu_2\}=W(\lambda-\mu)\cap \Lambda^+$.
Then we have $L(\nu_1)\cong L(\nu_2)^*$.
Hence for 
$M\in\modu(U_\zeta(\Gg)^{\dot{k}}_{[\dot{t}(\lambda)]})$, 
$N\in\modu(U_\zeta(\Gg)^{\dot{k}}_{[\dot{t}(\mu)]}))$ we have
\begin{align*}
&\Hom(T_{\mu\lambda}M,N)
=
\Hom({p}_{\mu}(L(\nu_1)\otimes M),N)
\cong
\Hom(L(\nu_1)\otimes M,N)
\\
\cong&
\Hom(L(\nu_2)^*\otimes M,N)
\cong
\Hom(M,L(\nu_2)\otimes N)
\cong
\Hom(M,{p}_{\lambda}(L(\nu_2)\otimes N))
\\
=&
\Hom(M,T_{\lambda\mu}N).
\end{align*}
\end{proof}

\begin{proposition}
\label{prop:Dtran}
Assume that $\mu\in\Lambda$ is contained in the closure of the facet containing $\lambda\in\Lambda$.
Then for 
$\GM\in D^b(\modu(\wGD^{\dot{k}}_{\dot{t}(\lambda)}))$ 
we have
\begin{align*}
T_{\mu\lambda}R\Gamma(\GM)
=R\Gamma(\GO[\mu-\lambda]\otimes_\GO\GM).
\end{align*}
\end{proposition}
\begin{proof}
Take $\nu\in\Lambda^+$ such that $\{\nu\}=W(\mu-\lambda)\cap \Lambda^+$.
Set $V=\Delta_\zeta(\nu)$, and choose its filtration as in \eqref{eq:filt1}.
By Proposition \ref{prop:tensV}
there exist
 $\GM_j\in D^b(\modu(\wGU^{\dot{k}}))$ for $j=1,\dots, m+1$ such that 
\[
\GM_1=V\otimes \GM,\qquad
\GM_{m+1}=0,
\]
and distinguished triangles
\[
\GM_{j+1}\to\GM_j
\to\GO[\xi_j]\otimes_\GO\GM
\xrightarrow{+1}.
\]
By applying $p_{\mu}\circ R\Gamma$ 
we obtain $M_j\in D^b(\modu(\wU_\zeta(\Gg)^{\dot{k}}_{[\dot{t}(\mu)]})$ 
for $j=1,\dots, m+1$ such that 
\[
M_1=V\otimes R\Gamma(\GM),\qquad
M_{m+1}=0,
\]
and distinguished triangles
\[
M_{j+1}\to M_j
\to
p_{\mu}
(R\Gamma
(\GO[\xi_j]\otimes_\GO\GM))
\xrightarrow{+1}.
\]
Assume $p_{\mu}
(R\Gamma
(\GO[\xi_j]\otimes_\GO\GM))\ne\{0\}$.
By $R\Gamma
(\GO[\xi_j]\otimes_\GO\GM)\in 
D^b(\modu(\wU_\zeta(\Gg)^{\dot{k}}_{[\dot{t}(\lambda+\xi_j)]}))$ 
we have 
$\lambda+\xi_j\in \tilde{W}_a\mu$.
By $\mu-\lambda-\xi_j\in Q$ and $\ell\Lambda\cap Q=\ell Q$ we obtain 
$\lambda+\xi_j\in {W}_a\mu$.
By Lemma \ref{lem:6J} this implies $\xi_j=\mu-\lambda$.
Since $\mu-\lambda$ is an extremal weight of $V$, such $j$ is unique, and for this $j$ we have
\[
T_{\mu\lambda}R\Gamma(\GM)
=
R\Gamma(\GO[\xi_j]\otimes_\GO\GM)
=
R\Gamma(\GO[\mu-\lambda]\otimes_\GO\GM).
\]
\end{proof}

Let $J\subset I$.
We set
\[
\CP_J^{\dot{x}}=
\{P_J^-g\in\CP_J\mid g\dot{x}g^{-1}\in P_J^-\},
\]
and denote by $\wCP_J^{\dot{x}}$ its completion in 
$\tilde{G}_J$.
Recall that for $\lambda\in\Lambda^J$ the algebra $\sD^{\dot{k}}_{J,\dot{t}(\lambda)}$
(resp.\ 
$\wsD^{\dot{k}}_{J,\dot{t}(\lambda)}$) 
is a split Azumaya algebra with
splitting bundle 
\begin{gather*}
\GK^{\dot{k}}_{J,\dot{t}}[\lambda]
=\GO_J[\lambda]\otimes_{\GO_J}
(\CO_{\CP_J^{\dot{x}}}
\otimes_{\BC}
K^{\dot{k}}_{[\dot{t}]}),
\\
(\text{resp.}
\quad
\wGK^{\dot{k}}_{J,\dot{t}}[\lambda]
=\GO_J[\lambda]\otimes_{\GO_J}
(\CO_{\wCP_J^{\dot{x}}}
\otimes_{\wCO(\CX)^{\dot{k}}_{[\dot{t}]}}
\wK^{\dot{k}}_{[\dot{t}]})),
\end{gather*}
where we identify 
$\sD^{\dot{k}}_{J,\dot{t}(\lambda)}$ and
$\wsD^{\dot{k}}_{J,\dot{t}(\lambda)}$ with 
an $\CO_{\CP_J^{\dot{x}}}$-algebra 
and an
$\CO_{\wCP_J^{\dot{x}}}$-algebra 
respectively (see \ref{subsec:locVJ}).
By definition we have
\begin{equation}
\label{eq:split-inv}
(\pi^{\dot{x}}_J)^*\GK^{\dot{k}}_{J,\dot{t}}[\lambda]
\cong
\GK^{\dot{k}}_{\dot{t}}[\lambda],
\qquad
(\wpi^{\dot{x}}_J)^*\wGK^{\dot{k}}_{J,\dot{t}}[\lambda]
\cong
\wGK^{\dot{k}}_{\dot{t}}[\lambda],
\end{equation}
where 
\[
\pi^{\dot{x}}_J:\CB^{\dot{x}}\to\CP^{\dot{x}}_J,
\qquad
{\wpi}^{\dot{x}}_J:\wCB^{\dot{x}}\to\wCP^{\dot{x}}_J
\]
are natural morphisms.
Set 
$W_J=\langle s_j\mid j\in J\rangle\subset W$, 
$Q_J=\sum_{j\in J}\BZ\alpha_j\subset Q$, and let 
$W_{J,a}$ be the subgroup of $W_a$ generated by $W_J$ and $\tau_\lambda$ for $\lambda\in Q_J$.
We denote by $\Lambda^{J,\reg}$ the set of $\lambda\in\Lambda^J$ satisfying the following equivalent conditions:
\begin{itemize}
\item
$\dot{t}(\lambda)$ is $J$-regular,
\item
if $wt_\lambda=t_\lambda$ for $w\in W$, then $w\in W_J$,
\item
if $y\lambda=\lambda$ for $y\in W_a$, then $y\in W_{J,a}$.
\end{itemize}

By Theorem \ref{thm:Morita}
and
Theorem \ref{thm:BB3} we obtain an equivalence
\begin{align}
\label{eq:de1bJ}
&D^b(\modu(\wU_\zeta(\Gg)^{\dot{k}}
_{[\dot{t}(\lambda)]}))
\cong
D^b(\modu(
\CO_{{\wCP_J^{\dot{x}}}}
))
\qquad(\lambda\in\Lambda^{J,\reg})
\end{align}
of triangulated categories.

\begin{proposition}
\label{prop:Tp}
Assume that $\lambda\in\Lambda$ is contained in the fundamental alcove $A_0$ and $\mu\in\Lambda$ lies on the facet $F\subset\overline{A}_0$ corresponding to $J\subset I$.
Then we have the following commutative diagram:
\[
\xymatrix@C=50pt{
D^b(\modu(
\CO_{\wCB^{\dot{x}}}
))
\ar[r]^{R\wpi^{\dot{x}}_{J*}}
\ar[d]_{\cong}
&
D^b(\modu(
\CO_{\wCP^{\dot{x}}_{J}}
))
\ar[d]^{\cong}
\\
D^b(\modu(\wU_\zeta(\Gg)^{\dot{k}}_{[\dot{t}(\lambda)]}))
\ar[r]_{T_{\mu\lambda}}
&
D^b(\modu(\wU_\zeta(\Gg)^{\dot{k}}_{[\dot{t}(\mu)]})).
}
\]
\end{proposition}
\begin{proof}
Let $\CS\in D^b(\modu(
\CO_{\wCB^{\dot{x}}}
))$.
By Proposition \ref{prop:Dtran} and \eqref{eq:split-inv} we have
\begin{align*}
T_{\mu\lambda}R\Gamma
(\GK^{\dot{k}}_{\dot{t}}[\lambda]\otimes_{\CO_{\CB^{\dot{x}}}}\CS)
\cong&
R\Gamma(\GO[\mu-\lambda]\otimes_\GO\GK^{\dot{k}}_{\dot{t}}[\lambda]\otimes_{\CO_{\CB^{\dot{x}}}}\CS)
\\
\cong&
R\Gamma(\GK^{\dot{k}}_{\dot{t}}[\mu]\otimes_{\CO_{\CB^{\dot{x}}}}\CS)
\\
\cong&
R\Gamma(
(\wpi^{\dot{x}}_J)^*(\GK^{\dot{k}}_{J,\dot{t}}[\mu])\otimes_{\CO_{\CB^{\dot{x}}}}\CS)
\\
\cong&
R\Gamma(R\wpi^{\dot{x}}_{J*}
((\wpi^{\dot{x}}_J)^*(\GK^{\dot{k}}_{J,\dot{t}}[\mu])\otimes_{\CO_{\CB^{\dot{x}}}}\CS))
\\
\cong&
R\Gamma(
\GK^{\dot{k}}_{J,\dot{t}}[\mu]\otimes_{\CO_{\CP_J^{\dot{x}}}} R\wpi^{\dot{x}}_{J*}\CS).
\end{align*}
\end{proof}
Since $(\wpi^{\dot{x}}_J)^*$
is left adjoint to 
$R\wpi^{\dot{x}}_{J*}$, we obtain from 
Lemma \ref{lem:8adX} and Proposition \ref{prop:Tp} the following.
\begin{proposition}
\label{prop:Tp2}
Let $\lambda$, $\mu$, $J$ be as in Proposition \ref{prop:Tp}.
Then we have the following commutative diagram:
\[
\xymatrix@C=50pt{
D^b(\modu(
\CO_{\wCP_J^{\dot{x}}}
))
\ar[r]^{(\wpi^{\dot{x}}_J)^*}
\ar[d]_{\cong}
&
D^b(\modu(
\CO_{\wCB^{\dot{x}}}
))
\ar[d]^{\cong}
\\
D^b(\modu(\wU_\zeta(\Gg)^{\dot{k}}_{[\dot{t}(\mu)]}))
\ar[r]_{T_{\lambda\mu}}
&
D^b(\modu(\wU_\zeta(\Gg)^{\dot{k}}_{[\dot{t}(\lambda)]})).
}
\]
\end{proposition}

\subsection{Wall crossing functor}
Let $i\in I_a$.
For $\lambda_0\in\Lambda\cap A_0$,  
$\mu_0\in \Lambda\cap F_0^i$
the exact functor
\begin{align}
R_{\lambda_0/\mu_0}
=T_{\lambda_0\mu_0}\circ T_{\mu_0\lambda_0}
:
\modu(\wU_\zeta(\Gg)^{\dot{k}}_{[\dot{t}(\lambda_0)]})
\to
\modu(\wU_\zeta(\Gg)^{\dot{k}}_{[\dot{t}(\lambda_0)]})
\end{align}
is called 
the wall crossing functor with respect to the wall $\GH_i$ of $A_0$.
By Lemma \ref{lem:8adX} it is self adjoint and we have natural morphisms
$\Id\to R_{\lambda_0/\mu_0}$
and 
$R_{\lambda_0/\mu_0}\to \Id$
of functors.
We define functors
\begin{align}
\BFI^!_{\lambda_0/\mu_0}, 
\;
\BFI^*_{\lambda_0/\mu_0}:
D^b(\modu(\wU_\zeta(\Gg)^{\dot{k}}_{[\dot{t}(\lambda_0)]}))
\to
D^b(\modu(\wU_\zeta(\Gg)^{\dot{k}}_{[\dot{t}(\lambda_0)]}))
\end{align}
by 
\[
\BFI^!_{\lambda_0/\mu_0}
=
\mathrm{cone}(R_{\lambda_0/\mu_0}\to\Id)[-1],
\qquad
\BFI^*_{\lambda_0/\mu_0}
=
\mathrm{cone}(\Id\to R_{\lambda_0/\mu_0}).
\]

\begin{proposition}
\label{prop:UDi}
Let $i\in I_a$ and let $\lambda_0\in\Lambda\cap A_0$,  
$\mu_0\in \Lambda\cap F_0^i$.
Assume $w\in\tilde{W}_a$ satisfies $ws_i\lambda_0-w\lambda_0\in Q^+$.
Then we have the following commutative diagrams of functors:
\[
\xymatrix@C=80pt{
D^b(\modu(\wGD^{\dot{k}}_{\dot{t}(w\lambda_0)}))
\ar[r]^{\GO[ws_i\lambda_0-w\lambda_0]\otimes_\GO(\bullet)}
\ar[d]_{R\Gamma}
&
D^b(\modu(\wGD^{\dot{k}}_{\dot{t}(ws_i\lambda_0)}))
\ar[d]^{R\Gamma}
\\
D^b(\modu
(\wU_\zeta(\Gg)^{\dot{k}}_{[\dot{t}(\lambda_0)]})
)
\ar[r]_{\BFI^*_{\lambda_0/\mu_0}}
&
D^b(\modu
(\wU_\zeta(\Gg)^{\dot{k}}_{[\dot{t}(\lambda_0)]})
),
}
\]
\[
\xymatrix@C=80pt{
D^b(\modu(\wGD^{\dot{k}}_{\dot{t}(ws_i\lambda_0)}))
\ar[r]^{\GO[w\lambda_0-ws_i\lambda_0]\otimes_\GO(\bullet)}
\ar[d]_{R\Gamma}
&
D^b(\modu(\wGD^{\dot{k}}_{\dot{t}(w\lambda_0)}))
\ar[d]^{R\Gamma}
\\
D^b(\modu
(\wU_\zeta(\Gg)^{\dot{k}}_{[\dot{t}(\lambda_0)]})
)
\ar[r]_{\BFI^!_{\lambda_0/\mu_0}}
&
D^b(\modu
(\wU_\zeta(\Gg)^{\dot{k}}_{[\dot{t}(\lambda_0)]})
).
}
\]
\end{proposition}
\begin{proof}
For
$\GM\in D^b(\modu(\wGD^{\dot{k}}_{\dot{t}(w\lambda_0)}))$ we have
\[
R_{\lambda_0/\mu_0}(R\Gamma(\GM))
=
T_{w\lambda_0, w\mu_0}T_{w\mu_0, w\lambda_0}(R\Gamma(\GM))
=T_{w\lambda_0, w\mu_0}R\Gamma(\GO[w\mu_0-w\lambda_0]\otimes_\GO\GM)
\]
by Proposition \ref{prop:Dtran}.
Moreover, similarly to the proof of Proposition \ref{prop:Dtran} we can show that 
for 
$\GN\in D^b(\modu(\wGD^{\dot{k}}_{\dot{t}(w\mu_0)}))$ 
there exists a distinguished triangle 
\[
R\Gamma(\GO[w\lambda_0-w\mu_0]\otimes_\GO\GN)
\to
T_{w\lambda_0,w\mu_0}R\Gamma(\GN)
\to R\Gamma(\GO[ws_i\lambda_0-w\mu_0]\otimes_\GO\GN)
\xrightarrow{+1}
\]
in $D^b(\modu
(\wU_\zeta(\Gg)^{\dot{k}}_{[\dot{t}(\lambda_0)]})
)$.
Hence we obtain a distinguished triangle 
\[
R\Gamma(\GM
)
\to
R_{\lambda_0/\mu_0}(R\Gamma(\GM))
\to
R\Gamma(
\GO[ws_i\lambda_0-w\lambda_0]\otimes_\GO\GM)
\xrightarrow{+1}.
\]
The commutativity of the first diagram follows from this using the argument in the proof of \cite[Lemma 2.2.3 (c)]{BMR2}.
The proof of the second diagram is similar.
\end{proof}
\begin{corollary}
The functors $\BFI^!_{\lambda_0/\mu_0}$ and $\BFI^*_{\lambda_0/\mu_0}$ do not depend on the choice of $\mu_0$.
Moreover, they are mutually inverse equivalences.
\end{corollary}

For $i\in I_a$ and $\lambda_0\in\Lambda\cap A_0$
we define functors 
\begin{equation}
\BFI^!_{\lambda_0/i}, 
\;
\BFI^*_{\lambda_0/i}:
D^b(\modu(\wU_\zeta(\Gg)^{\dot{k}}_{[\dot{t}(\lambda_0)]}))
\to
D^b(\modu(\wU_\zeta(\Gg)^{\dot{k}}_{[\dot{t}(\lambda_0)]}))
\end{equation}
by 
$
\BFI^!_{\lambda_0/\mu_0}
=
\BFI^!_{\lambda_0/i}$, 
$\BFI^*_{\lambda_0/i}
=
\BFI^*_{\lambda_0/\mu_0}
$
for 
$\mu_0\in \Lambda\cap F_0^i$.

Using Proposition \ref{prop:UDi}
we can show the following as in \cite{BMR2}.
\begin{proposition}
For $\lambda_0\in \Lambda\cap A_0$ we have a weak action 
\begin{equation}
\label{eq:I}
\BFI_b:
D^b(\modu(\wU_\zeta(\Gg)^{\dot{k}}_{[\dot{t}(\lambda_0)]}))
\to
D^b(\modu(\wU_\zeta(\Gg)^{\dot{k}}_{[\dot{t}(\lambda_0)]}))
\qquad(b\in\tilde{\BB}_a)
\end{equation}
of $\tilde{\BB}_a$ on 
$D^b(\modu(\wU_\zeta(\Gg)^{\dot{k}}_{[\dot{t}(\lambda_0)]}))$
satisfying 
$\BFI_{b_{s_i}}=\BFI^*_{\lambda_0/i}$ for $i\in I_a$ and 
$\BFI_{b_{\omega}}=T_{\omega\lambda_0\lambda_0}$ for $\omega\in\Omega$.
Moreover, for any $\nu\in\Lambda^+$ we have the following commutative diagram of functors:
\[
\xymatrix@C=80pt{
D^b(\modu(\wGD^{\dot{k}}_{\dot{t}(\lambda_0)}))
\ar[r]^{\CO_\CB[\nu]\otimes_{\CO_\CB}(\bullet)}
\ar[d]_{R\Gamma}
&
D^b(\modu(\wGD^{\dot{k}}_{\dot{t}(\lambda_0+\ell\nu)}))
\ar[d]^{R\Gamma}
\\
D^b(\modu(\wU_\zeta(\Gg)^{\dot{k}}_{[\dot{t}(\lambda_0)]}))
\ar[r]_{\BFI_{\tau_\nu}}
&
D^b(\modu(\wU_\zeta(\Gg)^{\dot{k}}_{[\dot{t}(\lambda_0)]})).
}
\]
\end{proposition}
\subsection{Proof of Theorem \ref{thm:ab}}
Using Proposition \ref{prop:Tp}, Proposition \ref{prop:Tp2} for $J=\{i\}$ with $i\in I$ we can show the following as in \cite{Ri}.
\begin{proposition}
\label{prop:IJ}
For $\lambda_0\in \Lambda\cap A_0$
and $b\in \tBB_a$
we have the following commutative diagram of functors:
\[
\xymatrix@C=30pt{
D^b(\modu(
\CO_{\wCB^{\dot{x}}}
))
\ar[r]^{\BFJ_b}
\ar[d]_{\cong}
&
D^b(\modu(
\CO_{\wCB^{\dot{x}}}
))
\ar[d]^{\cong}
\\
D^b(
\modu(
\wU_\zeta(\Gg)^{\dot{k}}_{[\dot{t}(\lambda_0)]}
))
\ar[r]_{\BFI_{b}}
&
D^b(
\modu(
\wU_\zeta(\Gg)^{\dot{k}}_{[\dot{t}(\lambda_0)]}
)).
}
\]
\end{proposition}

Our main result Theorem \ref{thm:ab} is shown using the argument of \cite[1.6.5]{BM} as follows.
Let $\lambda_o\in A_0\cap\Lambda$.
By Proposition \ref{prop:Tp} for $J=I$ we have the following commutative diagram:
\[
\xymatrix@C=50pt{
D^b(\modu(
\CO_{\wCB^{\dot{x}}}
))
\ar[r]^{R\varrho^{\dot{x}}_{*}}
\ar[d]_{\cong}
&
D^b(\modu(
\CO_{\widehat{\{\dot{x}\}}}
))
\ar[d]^{\cong}
\\
D^b(\modu(\wU_\zeta(\Gg)^{\dot{k}}_{[\dot{t}(\lambda_0)]}))
\ar[r]_{T_{0\lambda_0}}
&
D^b(\modu(\wU_\zeta(\Gg)^{\dot{k}}_{[\dot{t}]})).
}
\]
Here, the right vertical equivalence is induced by the equivalence 
\[
\modu(
\CO_{\widehat{\{\dot{x}\}}})
\cong
\modu(\wU_\zeta(\Gg)^{\dot{k}}_{[\dot{t}]})
\]
of abelian categories (see Proposition \ref{prop:BrG}).
Hence by transferring the exotic $t$-structure of 
$D^b(\modu(
\CO_{\wCB^{\dot{x}}}))$
to 
$D^b(
\modu(
\wU_\zeta(\Gg)^{\dot{k}}_{[\dot{t}(\lambda_0)]}))
$ 
through \eqref{eq:de1b} we obtain a 
$t$-structure $(D_{\ex}^{\leqq0}, D_\ex^{\geqq0})$ of 
$D=D^b(
\modu(
\wU_\zeta(\Gg)^{\dot{k}}_{[\dot{t}(\lambda_0)]}))
$ 
given by
\begin{align*}
D_{\ex}^{\leqq0}
=&
\{
M
\in
D\mid
T_{0\lambda_0}({\BFI}_{b}M)\in D^{\leqq0}(\modu(\wU_\zeta(\Gg)^{\dot{k}}_{[\dot{t}]}))
\;\;(b\in\tBB_a^+)
\},
\\
D_{\ex}^{\geqq0}
=&
\{
M
\in
D
\mid
T_{0\lambda_0}(({\BFI}_b)^{-1}M)\in D^{\geqq0}(\modu(\wU_\zeta(\Gg)^{\dot{k}}_{[\dot{t}]}))
\;\;(b\in\tBB_a^+)
\}.
\end{align*}
It is sufficient to show that this coincides with the standard $t$-structure  $(D_{\st}^{\leqq0}, D_\st^{\geqq0})$ of $D$ given by
\begin{align*}
D_{\st}^{\leqq0}
=&
\{
M
\in
D\mid
H^i(M)=0\;(i>0)
\},
\\
D_{\st}^{\geqq0}
=&
\{
M
\in
D
\mid
H^i(M)=0\;(i<0)
\}.
\end{align*}
Note that the translation functor is exact.
Hence for $b\in\tBB_a^+$ the functor $\BFI_b$ 
(resp.\ $(\BFI_b)^{-1}$) is right (resp.\ left) exact by the definition of 
$\BFI_b$.
Hence we obtain 
$
D_{\st}^{\leqq0}\subset D_{\ex}^{\leqq0}$ and
$
D_{\st}^{\geqq0}\subset D_{\ex}^{\geqq0}$.
It follows  that 
$
D_{\st}^{\leqq0}= D_{\ex}^{\leqq0}$ and
$
D_{\st}^{\geqq0}= D_{\ex}^{\geqq0}$
by the definition of $t$-structures.

The proof of Theorem \ref{thm:ab} is now complete.
\section{{Lusztig's conjecture}}
\label{sec:Lusztig}
\subsection{}
For $\lambda\in\Lambda$ we denote by 
$\modu_{[\dot{t}(\lambda)]}(U_\zeta(\Gg)^{\dot{k}})$ the category of finitely generated $U_\zeta(\Gg)$-modules on which 
$z-\xi^{\dot{k}}(z)$ for $z\in Z_\Fr(U_\zeta(\Gg))$ acts as zero and 
$z-\xi_{[\dot{t}(\lambda)]}(z)$ for $z\in Z_\Har(U_\zeta(\Gg))$ acts nilpotently.
We also denote by 
$\modu_{[\dot{t}(\lambda)]}(U_\zeta(\Gg)^{\dot{k}};C)$ 
the subcategory of 
$\modu(U_\zeta(\Gg)^{\dot{k}};C)$ consisting of objects which belong to 
$\modu_{[\dot{t}(\lambda)]}(U_\zeta(\Gg)^{\dot{k}})$ 
as a $U_\zeta(\Gg)^{\dot{k}}$-module.
We have
\begin{align}
\label{eq:Uemb1}
\modu(U_\zeta(\Gg)_{[\dot{t}(\lambda)]}^{\dot{k}})
\subset&
\modu_{[\dot{t}(\lambda)]}(U_\zeta(\Gg)^{\dot{k}})
\subset
\modu_{[\dot{t}(\lambda)]}^{\dot{k}}(U_\zeta(\Gg))
\subset
\modu(\wU_\zeta(\Gg)_{[\dot{t}(\lambda)]}^{\dot{k}}),
\end{align}
\begin{align}
\label{eq:Uemb2}
\modu(U_\zeta(\Gg)_{[\dot{t}(\lambda)]}^{\dot{k}};C)
\subset
\modu_{[\dot{t}(\lambda)]}(U_\zeta(\Gg)^{\dot{k}};C)
\subset
\modu_{[\dot{t}(\lambda)]}^{\dot{k}}(U_\zeta(\Gg);C)
\\
\nonumber
\subset
\modu(\wU_\zeta(\Gg)_{[\dot{t}(\lambda)]}^{\dot{k}};C).
\end{align}
It is known by Jantzen \cite{JanM} that any irreducible $U_\zeta(\Gg)_{[\dot{t}(\lambda)]}^{\dot{k}}$-module $L$ admits a unique (up to shift of grading) $\Lambda_C$-grading by which $L\in\modu_{[\dot{t}(\lambda)]}^{\dot{k}}(U_\zeta(\Gg);C)$.
We denote by 
$\{L_\sigma\}_{\sigma\in\Theta}$
the collection of irreducible $U_\zeta(\Gg)_{[\dot{t}(\lambda)]}^{\dot{k}}$-modules with 
$\Lambda_C$-gradings
contained in $\modu_{[\dot{t}(\lambda)]}^{\dot{k}}(U_\zeta(\Gg);C)$. 
The group $\Lambda_C$ acts on $\Theta$ freely by
\[
L_{\gamma\sigma}=L_\sigma[\gamma]
\qquad(\gamma\in\Lambda_C, \;\sigma\in\Theta),
\]
where $[\gamma]$ denotes the shift of grading, 
and the set of the irreducible $U_\zeta(\Gg)_{[\dot{t}(\lambda)]}^{\dot{k}}$-modules is parametrized by $\Theta/\Lambda_C$.
For $\sigma\in\Theta$ 
we define $E_\sigma$ (resp.\ $\wE_\sigma$) to be the projective cover of $L_\sigma$ in 
$\modu_{[\dot{t}(\lambda)]}(U_\zeta(\Gg)^{\dot{k}};C)$
(resp.\ 
$\modu(\wU_\zeta(\Gg)_{[\dot{t}(\lambda)]}^{\dot{k}};C)$.
We have
\begin{equation}
\label{eq:wFF}
E_\sigma=\BC
\otimes_{Z_\Fr(U_\zeta(\Gg))}
\wE_\sigma
\end{equation}
with respect to $\xi^{\dot{k}}:Z_\Fr(U_\zeta(\Gg))\to\BC$.

If $\GA$ is an abelian category or a triangulated category,  we denote its Grothendieck group  by $K(\GA)$.

By \eqref{eq:Uemb2} we obtain
\begin{align}
\label{eq:KUemb2}
K(\modu(U_\zeta(\Gg)_{[\dot{t}(\lambda)]}^{\dot{k}};C))
\cong
K(\modu_{[\dot{t}(\lambda)]}(U_\zeta(\Gg)^{\dot{k}};C))
\cong
K(\modu_{[\dot{t}(\lambda)]}^{\dot{k}}(U_\zeta(\Gg);C))
\\
\nonumber
\subset
K(\modu(\wU_\zeta(\Gg)_{[\dot{t}(\lambda)]}^{\dot{k}};C)).
\end{align}
Note that 
$K(\modu_{[\dot{t}(\lambda)]}(U_\zeta(\Gg)^{\dot{k}};C))$ is a free $\BZ$-module with  basis 
$\{[L_\sigma]\}_{\sigma\in\Theta}$.
\subsection{}
Recall that we have equivalences
\begin{align}
D^b(\modu^{\dot{k}}_{[\dot{t}(\lambda)]}(U_\zeta(\Gg);C))
\cong
D^b(\modu_{\CB^{\dot{x}}}(\CO_{\widetilde{G}};C))
\qquad(\lambda\in\Lambda^\reg),
\\
D^b(\modu(\wU_\zeta(\Gg)^{\dot{k}}_{[\dot{t}(\lambda)]};C))
\cong
D^b(\modu(\CO_{\wCB^{\dot{x}}};C))
\qquad(\lambda\in\Lambda^\reg)
\end{align}
of triangulated categories and an equivalence 
\[
\modu(\wU_\zeta(\Gg)^{\dot{k}}_{[\dot{t}(\lambda)]};C)
\cong
\modu^\ex(\CO_{\wCB^{\dot{x}}};C)
\qquad(\lambda\in\Lambda\cap A_0)
\]
of abelian categories.
We have the following identifications of the Grothendieck groups:
\begin{align}
K(\modu(\CO_{\CB^{\dot{x}}};C))
=&
K(\modu_{\CB^{\dot{x}}}(\CO_{\widetilde{G}};C)),
\\
K(\modu(\CO_{\wCB^{\dot{x}}};C))
=&
K(D^b(\modu(\CO_{\wCB^{\dot{x}}};C)))
=
K(\modu^\ex(\CO_{\wCB^{\dot{x}}};C)),
\end{align}
and the following commutative diagram:
\begin{equation}
\xymatrix@C=15pt{
K(\modu(U_\zeta(\Gg)^{\dot{k}}_{[\dot{t}(\lambda)]};C))
\ar@{^{(}->}[r]^{}
\ar[d]_{\cong}
&
K(\modu(\wU_\zeta(\Gg)^{\dot{k}}_{[\dot{t}(\lambda)]};C))
\ar[d]^{\cong}
\\
K(\modu(\CO_{\CB^{\dot{x}}};C))
\ar@{^{(}->}[r]^{}
&
K(\modu(\CO_{\wCB^{\dot{x}}};C)).
}
\end{equation}

For $\sigma\in\Theta$ we denote by 
$\CL_\sigma$, $\CE_\sigma$, $\wCE_\sigma$ the objects of 
$\modu^\ex(\CO_{\wCB^{\dot{x}}};C)$ corresponding to 
the objects $L_\sigma$, $E_\sigma$, $\wE_\sigma$ 
of $\modu(\wU_\zeta(\Gg)^{\dot{k}}_{[\dot{t}(\lambda)]};C)$ respectively.
We have the following description of 
$\CL_\sigma$, $\CE_\sigma$, $\wCE_\sigma$.
\begin{proposition}
\label{prop:LE}
\begin{itemize}
\item[(i)]
The collection 
$\{\CL_\sigma\}_{\sigma\in\Theta}$ is exactly the set of irreducible objects of 
$\modu^\ex(\CO_{\wCB^{\dot{x}}};C)$ contained in 
$D^b(\modu_{\CB^{\dot{x}}}(\CO_{\widetilde{G}};C))$.
\item[(ii)]
For $\sigma\in\Theta$
the exotic sheaf 
$\wCE_\sigma$ is a projective cover of $\CL_\sigma$ in 
$\modu^\ex(\CO_{\wCB^{\dot{x}}};C)$.
\item[(iii)]
We have 
$\CE_\sigma=\CO_{\{\dot{x}\}}\otimes^L_{\CO_{G}}\wCE_\sigma$ for $\sigma\in\Theta$.
\end{itemize}
\end{proposition}
\begin{proof}
The assertions (i) and (ii) are obvious from the equivalence of categories.
We can easily show (iii) from \eqref{eq:wFF}.
\end{proof}
We have the following result due to Bezrukavnikov-Mirkovi\'{c} 
(see \cite[Proposition 5.1.4]{BM}).
\begin{theorem}[\cite{BM}]
\label{thm:ELF}
For any $\sigma\in\Theta$ the exotic sheaf 
$\wCE_\sigma$ is a locally free $\CO_{\wCB^{\dot{x}}}$-module.
\end{theorem}
\subsection{}
Let us identify $\CB^{\dot{x}}$ and $\wCB^{\dot{x}}$ with 
$\CB^{\dot{a}}$ and $\wCB^{\dot{a}}$ 
respectively via \eqref{eq:Bax}.
Here, $\dot{a}\in\Gg$ is the nilpotent element satisfying $\exp(\dot{a})=\dot{x}$.
Take $\dot{b},\dot{c}\in\Gg$ such that 
$[\dot{c},\dot{a}]=2\dot{a}$, 
$[\dot{c},\dot{b}]=-2\dot{b}$, 
$[\dot{a},\dot{b}]=\dot{c}$, and set
\[
\tS^{\dot{a}}
=
\{(B^-g,a)\in\tilde{\Gg}\mid
a\in\Gg_\nil\cap(\dot{a}+\Gz_\Gg(\dot{b}))\},
\]
where $\Gg_\nil$ denotes the set of nilpotent elements in $\Gg$ and 
$\Gz_\Gg(\dot{b})$ is the centralizer of $\dot{b}$ in $\Gg$.
This variety is called the Slodowy slice.
We may assume that $C$ is a maximal torus of the simultaneous centralizer of $\dot{a}$, $\dot{b}$, $\dot{c}$.
The right action of $C$ on $\tilde{\Gg}$ given by 
\[
(B^-g,a)h=(B^-gh,\Ad(h^{-1})(a))
\qquad((B^-g,a)\in\tilde{\Gg},\; h\in C)
\]
preserves $\CB^{\dot{a}}$, 
$\wCB^{\dot{a}}$ and $\tS^{\dot{a}}$.
We have the identifications
\[
\modu_{\CB^{\dot{x}}}(\CO_{\widetilde{G}};{C})
=
\modu_{\CB^{\dot{a}}}(\CO_{\widetilde{\Gg}};{C}),
\qquad
\modu(\CO_{\wCB^{\dot{x}}};{C})
=
\modu(\CO_{\wCB^{\dot{a}}};{C}).
\]

We define a right action of ${C}'=C\times\BC^\times$ on $\Gg$ and $\tilde{\Gg}$ by 
\[
a\cdot(h,z)=
z^2\Ad(h^{-1}\phi(z)^{-1})(a)
\quad
(a\in\Gg, \; (h,z)\in C'),
\]
\[
(B^-g,a)(h,z)=
(B^-g\phi(z)h,a\cdot(h,z))
\quad
((B^-g,a)\in\tilde{\Gg}, \; (h,z)\in C'),
\]
where the homomorphism $\phi:\BC^\times\to G$ is given by
$\phi(\exp(u))=\exp(u\dot{c})$ for $u\in\BC$.
We have natural lifts 
$\CL'_\sigma$, $\CE'_\sigma$, 
$\widehat{\CE}'_\sigma\in D^b(\modu(\CO_{\wCB^{\dot{a}}};C'))$
of
$\CL_\sigma$, $\CE_\sigma$, $\widehat{\CE}_\sigma\in D^b(\modu(\CO_{\wCB^{\dot{a}}};C))$
(see \cite[5.3.2]{BM}).
Moreover, for each $\sigma\in \Theta$ there exists uniquely a $C'$-equivariant 
locally free $\CO_{\tS^{\dot{a}}}$-module $\tCE'_\sigma$ whose restriction to the formal neighborhood of $\CB^{\dot{a}}$ in $\tS^{\dot{a}}$ coincides with the restriction of $\wCE'_\sigma$ (see \cite[5.3.1]{BM}).
By Proposition \ref{prop:LE} and Theorem \ref{thm:ELF} 
we have
\begin{equation}
\CE'_\sigma
=
\CO_{\{\dot{x}\}}\otimes^L_{\CO_{\Gg}}\tCE'_\sigma.
\end{equation}

Let $R_{C'}$ be the representation ring of $C'$, 
and let $\CR_{C'}$ be its quotient field.
The direct image with respect the closed embedding 
$\CB^{\dot{a}}\hookrightarrow\tS^{\dot{a}}$ gives an embedding
\begin{equation}
K(\modu(\CO_{\CB^{\dot{a}}};C'))
\subset
K(\modu(\CO_{\tS^{\dot{a}}};C'))
\end{equation}
of free $R_{C'}$-modules of the same rank 
$\dim_\BQ H^*(\CB^{\dot{a}};\BQ)$
(see \cite[Theorem 1.14]{LK}).
This induces the identification 
\begin{equation}
\CR_{C'}\otimes_{R_{C'}}K(\modu(\CO_{\CB^{\dot{a}}};C'))
=
\CR_{C'}\otimes_{R_{C'}}
K(\modu(\CO_{\tS^{\dot{a}}};C'))
\end{equation}
of $\CR_{C'}$-modules.
We denote by $f\mapsto f^\mathsf{v}$ the automorphism of $\CR_{C'}$ induced by $C'\ni c\mapsto c^{-1}\in C'$.
Set
\[
\nabla_{\dot{a}}
=\left(
\sum_r(-1)^r[\bigwedge^r\Gz_\Gg(\dot{b})]
\right)
\left(
\sum_r(-1)^r[\bigwedge^r\Gh]
\right)^{-1}
\in \CR_{C'}.
\]
In \cite[Conjecture 5.12, Conjecture 5.16]{LK} Lusztig conjectured the existence of certain canonical 
$\BZ[v,v^{-1}]$-bases 
$\mathbf{B}_{\CB^{\dot{a}}}$ and 
$\mathbf{B}_{S^{\dot{a}}}$ 
of 
$K(\modu(\CO_{\CB^{\dot{a}}};{C}'))$ and 
$K(\modu(\CO_{S^{\dot{a}}};{C}'))$
respectively.
Here $\BZ[v,v^{-1}]$ is identified with the representation ring of $\BC^\times$.
This conjecture was proved by 
Bezrukavnikov-Mirkovi\'{c} \cite{BM} as follows.
\begin{theorem}[\cite{BM}]
\label{thm:basis}
\begin{itemize}
\item[(i)] 
The set
$\{[\CL'_\sigma]\}_{\sigma\in\Theta}$ turns out to be Lusztig's conjectural canonical base $\mathbf{B}_{\CB^{\dot{a}}}$ of $K(\modu(\CO_{\CB^{\dot{a}}};C'))$.
\item[(ii)] 
The set 
$\{[\tCE'_\sigma]\}_{\sigma\in\Theta}$ 
 turns out to be Lusztig's conjectural canonical base 
 $\mathbf{B}_{\tS^{\dot{a}}}$ 
 of $K(\modu(\CO_{\tS^{\dot{a}}};C'))$.
\item[(iii)] 
We have
\[
[\CE'_\sigma]=\nabla_{\dot{a}}^\mathsf{v}[\tCE'_\sigma]
\qquad(\sigma\in\Theta)
\]
in
$K(\modu(\CO_{\CB^{\dot{a}}};C'))
\subset K(\modu(\CO_{\tS^{\dot{a}}};C'))$.
\end{itemize}
\end{theorem}

We define $n_{\tau,\sigma}\in\BZ[v,v^{-1}]$ 
($\sigma, \tau\in\Theta$) by
\begin{equation}
\label{eq:KLp}
[\CE'_\sigma]=\sum_{\tau\in\Theta}
n_{\tau,\sigma}[\CL'_\tau]
\qquad(\sigma\in\Theta)
\end{equation}
in $K(\modu(\CO_{\CB^{\dot{a}}};C'))$.
We note that there exists  a description of $n_{\tau,\sigma}$ in terms of $\mathbf{B}_{\tS^{\dot{a}}}$ and a certain geometrically defined bilinear form on
$K(\modu(\CO_{\tS^{\dot{a}}};C'))$, which we omit here
(see \cite{LK}, \cite{BM} for more details).
We obtain the following result
conjectured by Lusztig \cite[17.2]{LK}.
\begin{theorem}
\label{thm:character}
For $\sigma\in\Theta$ we have
\[
[E_\sigma]=\sum_{\tau\in\Theta}
n_{\tau,\sigma}(1)[L_\tau]
\qquad(\sigma\in\Theta)
\]
in $K(\modu_{[\dot{t}(\lambda)]}(U_\zeta(\Gg)^{\dot{k}};C))$.
\end{theorem}
\begin{proof}
By \eqref{eq:KLp} we have
\[
[\CE_\sigma]=\sum_{\tau\in\Theta}
n_{\tau,\sigma}(1)[\CL_\tau]
\qquad(\sigma\in\Theta)
\]
in $K(\modu(\CO_{\CB^{\dot{a}}};C))$.
The desired formula follows from this by the equivalence of categories.
\end{proof}


\bibliographystyle{unsrt}

\begin{thebibliography}{99}
\bibitem{APW}
Andersen, H., 
Polo, P., 
Wen, K.: 
Representations of quantum algebras. 
Invent. Math. {\textbf{104}} (1991), no. 1, 1--59.

\bibitem{ABG}
Arkhipov, S., 
Bezrukavnikov, R., 
Ginzburg, V.:
Quantum groups, the loop Grassmannian, and the Springer resolution. J. Amer. Math. Soc. \textbf{17} (2004), no. 3, 595--678. 

\bibitem{AZ}
Artin, M., 
Zhang, J.: 
Noncommutative projective schemes. 
Adv. Math. {\textbf{109}} (1994), 228--287.

\bibitem{BK1}
Backelin, E., 
Kremnizer, K.:
Quantum flag varieties, equivariant quantum $\DD$-modules, and localization of quantum groups.
Adv. Math. {\textbf{203}} (2006), 408--429.

\bibitem{BK2}
Backelin, E., 
Kremnizer, K.:
Localization for quantum groups at a root of unity. 
J. Amer. Math. Soc. {\bf21} (2008), no. 4, 1001--1018.

\bibitem{BK3}
Backelin, E., 
Kremnizer, K.:
Singular localization for quantum groups at generic $q$. 
Adv. Math. {\textbf{249}} (2013), 359--381.


\bibitem{BB1}
Beilinson, A., 
Bernstein, J.:
Localisation de $\Gg$-modules. 
C. R. Acad. Sci. Paris S\'{e}r. I Math. \textbf{292} (1981), no. 1, 15--18. 

\bibitem{BB2}
Beilinson, A., 
Bernstein, J.:
A generalization of Casselman's submodule theorem. Representation theory of reductive groups (Park City, Utah, 1982), 35--52, Progr. Math., \textbf{40}, Birkh\"{a}user Boston, Boston, MA, 1983.


\bibitem{BG}
Bernstein, J. N., 
Gelfand, S. I.: 
Tensor products of finite and infinite dimensional representations of semisimple Lie algebras. 
Compositio Math. \textbf{41} (1980), no. 2, 245--285. 

\bibitem{Be}
Bezrukavnikov, R.: 
Cohomology of tilting modules over quantum groups and 
$t$-structures on derived categories of coherent sheaves. 
Invent. Math. \textbf{166} (2006), 327--357.

\bibitem{BL}
Bezrukavnikov, R., 
Losev, I.:
Dimensions of modular irreducible representations of semisimple Lie algebras.
arXiv:2005.10030.

 
\bibitem{BMR}
Bezrukavnikov, R., 
Mirkovi\'{c}, I.,
Rumynin, D.: 
Localization of modules for a semisimple Lie algebra in prime characteristic.  With an appendix by Bezrukavnikov and Simon Riche. 
Ann. of Math. (2) \textbf{167} (2008), no. 3, 945--991.

\bibitem{BMR2}
Bezrukavnikov, R., 
Mirkovi\'{c}, I.,
Rumynin, D.: 
Singular localization and intertwining functors for reductive Lie algebras in prime characteristic. Nagoya Math. J. \textbf{184}(2006), 1--55.

\bibitem{BM}
Bezrukavnikov, R., 
Mirkovi\'{c}, I.: 
Representations of semisimple Lie algebras in prime characteristic and the noncommutative Springer resolution, With an appendix by Eric Sommers.
Ann. of Math. (2) \textbf{178} (2013), no. 3, 835--919.

\bibitem{BrG}
Brown, K., 
Gordon, I.: 
The ramification of centres: Lie algebras in positive characteristic and quantised enveloping algebras. 
Math. Z. \textbf{238} (2001), no. 4, 733--779. 

\bibitem{BrGoo}
Brown, K. A.; 
Goodearl, K. R.: 
Homological aspects of Noetherian PI Hopf algebras of irreducible modules and maximal dimension. 
J. Algebra \textbf{198} (1997), no. 1, 240--265.

\bibitem{BK}
 Brylinski, J.-L., 
 Kashiwara, M.: 
 Kazhdan-Lusztig conjecture and holonomic systems. 
 Invent. Math. \textbf{64} (1981), no. 3, 387--410.

\bibitem{Cal}
Caldero, P.: 
\'{E}l\'{e}ments ad-finis de certains groupes quantiques. 
C. R. Acad. Sci. Paris S\'{e}r. I Math. 316 (1993), no. 4, 327--
329.

\bibitem{DK1}
De Concini, C., 
Kac, V.: 
Representations of quantum groups at roots of 1. 
Operator algebras, unitary representations, enveloping algebras, and invariant theory (Paris, 1989), 471--506, Progr. Math., \textbf{92}, Birkh\"{a}user Boston, Boston, MA, 1990. 


\bibitem{DK2}
De Concini, C., 
 Kac, V.: 
 Representations of quantum groups at roots of 1: reduction to the exceptional case. 
Infinite analysis, Part A, B (Kyoto, 1991), 141--149, Adv. Ser. Math. Phys., \textbf{16}, World Sci. Publ., River Edge, NJ, 1992.


\bibitem{DKP}
De Concini, C., 
Kac, V., 
Procesi, C.:  
Quantum coadjoint action. 
J. Amer. Math. Soc. \textbf{5} (1992), no. 1, 151--189.

\bibitem{DP}
De Concini, C., 
Procesi, C.: 
Quantum groups. 
$D$-modules, representation theory, and quantum groups (Venice, 1992), 31--140, 
Lecture Notes in Math., \textbf{1565}, Springer, Berlin, 1993.

\bibitem{Gav}
Gavarini, F.:
Quantization of Poisson groups. 
Pacific J. Math. \textbf{186} (1998), no. 2, 217--266.

\bibitem{H}
Humphreys, J.: 
Linear algebraic groups. 
Graduate Texts in Mathematics, No. \textbf{21}. 
Springer-Verlag, New York-Heidelberg, 1975. xiv+247 pp.

\bibitem{JanM}
Jantzen, J.: 
Modular representations of reductive Lie algebras. 
Commutative algebra, homological algebra and representation theory (Catania/Genoa/Rome, 1998). 
J. Pure Appl. Algebra \textbf{152} (2000), no. 1-3, 133--185.

\bibitem{JanB}
Jantzen, J.: 
Representations of algebraic groups. Second edition. Mathematical Surveys and Monographs, \textbf{107}. 
American Mathematical Society, Providence, RI, 2003. xiv+576 pp.




\bibitem{Jo0}
Joseph, A.: 
Faithfully flat embeddings for minimal primitive quotients of quantized enveloping algebras. 
Quantum deformations of algebras and their representations (Ramat-Gan, 1991/1992; Rehovot, 1991/1992), 79--106, Israel Math. Conf. Proc., {\bf7}, Bar-Ilan Univ., Ramat Gan, 1993.


\bibitem{Kas}
Kashiwara, M.: 
Equivariant derived category and representation of real semisimple Lie groups. 
Representation theory and complex analysis, 137--234, 
Lecture Notes in Math., \textbf{1931}, 
Springer, Berlin, 2008.

\bibitem{KT}
Kashiwara, M.,  
Tanisaki, T.: 
Kazhdan-Lusztig conjecture for affine Lie algebras with negative level. II. Nonintegral case. Duke Math. J. \textbf{84} (1996), no. 3, 771--813. 

\bibitem{KL}
Kazhdan, D., 
Lusztig, G.: 
Tensor structures arising from affine Lie algebras. I, II, III, IV.
J. Amer. Math. Soc. \textbf{6} (1993), 905--947, 949--1011, \textbf{7} (1994), 335--381, 383--453.

\bibitem{Kos}
Kostant, B.:
On the tensor product of a finite and an infinite dimensional representation. 
J. Functional Analysis \textbf{20} (1975), no. 4, 257--285.


\bibitem{LR}
Lunts, V., Rosenberg, A.: 
Localization for quantum groups. 
Selecta Math. (N.S.) \textbf{5} (1999), 123--159. 


\bibitem{Lbook}
Lusztig, G.: 
Introduction to quantum groups. 
Progr. Math., \textbf{110}, Boston etc. Birk\"hauser, 1993.

\bibitem{LM}
Lusztig, G.: 
Monodromic systems on affine flag manifolds. 
Proc. Roy. Soc. London Ser. A
\textbf{445} (1994), 231--246; 
Errata, \textbf{450} (1995), 731--732.


\bibitem{LP}
Lusztig, G.: 
Periodic $W$-graphs. 
Represent. Theory \textbf{1} (1997), 207--279.



\bibitem{LK}
Lusztig, G.: 
Bases in equivariant K-theory II. 
Represent.Th. \textbf{3} (1999), 281--353.


\bibitem{M}
Manin, Yuri I.:
Topics in noncommutative geometry. 
M. B. Porter Lectures. 
Princeton University Press, Princeton, NJ, 1991. viii+164 pp. 


\bibitem{P}
Popescu, N.: 
Abelian categories with applications to rings and modules. 
London Mathematical Society Monographs, No. 3. 
Academic Press, London-New York, 1973. xii+467 pp. 

\bibitem{Ri}
Riche, S.: 
Geometric braid group action on derived categories of coherent sheaves. 
With a joint appendix with Roman Bezrukavnikov.
Represent. Theory \textbf{12} (2008), 131--169.

\bibitem{R}
 Rosenberg, Alexander L.: 
 Noncommutative algebraic geometry and representations of quantized algebras. 
Mathematics and its Applications, \textbf{330}. 
Kluwer Academic Publishers Group, Dordrecht, 1995. xii+315 pp. 



\bibitem{T0}
Tanisaki, T.: 
The Beilinson-Bernstein correspondence for quantized enveloping algebras. 
Math. Z. \textbf{250} (2005), 299--361.


\bibitem{T1}
Tanisaki, T.:
Differential operators on quantized flag manifolds at roots of unity. 
Adv. Math. \textbf{230} (2012), 2235--2294.


\bibitem{T2}
Tanisaki, T.:
Differential operators on quantized flag manifolds at roots of unity, II. Nagoya Math. J. \textbf{214} (2014), 1--52. 


\bibitem{TZ}
Tanisaki, T.:
The center of a quantized enveloping algebra at an even root of unity. Osaka J. Math. \textbf{53} (2016), no. 1, 47--81. 

\bibitem{TA}
Tanisaki, T.:
Affine open covering of the quantized flag manifolds at roots of unity.
J. Algebra \textbf{586} (2021), 62--98.

\bibitem{T3}
Tanisaki, T.:
Differential operators on quantized flag manifolds at roots of unity, III. 
Adv. Math. \textbf{392} (2021), Paper No. 107990, 51 pp.

\bibitem{TK}
Tanisaki, T.:
The Koszul complex and a certain induced module for a quantum group. 
Int. Math. Res. Not. IMRN 2024, no. 11, 9376--9410.

\bibitem{TR}
Tanisaki, T.:
The ring of differential operators on a quantized flag manifold.
J. Algebra \textbf{694} (2026), 1--28.

\bibitem{TC}
Tanisaki, T.:
Categories of D-modules on a quantized flag manifold.
arXiv:2308.08711.

\bibitem{V}
Ver\"{e}vkin, A.: 
On a noncommutative analogue of the category of coherent sheaves on a projective scheme. 
Algebra and analysis (Tomsk, 1989), 41--53, Amer. Math. Soc. Transl. Ser. 2, \textbf{151}, Amer. Math. Soc., Providence, RI, 1992.

\end{thebibliography}

\end{document}